%% file: arxiv_paper.tex
\documentclass{article}
\usepackage[T1]{fontenc} % Ensures proper character encoding
\usepackage[utf8]{inputenc}
\usepackage[polish, english]{babel}
\newcommand{\red}[1]{{\color{red}{#1}}} 

\usepackage{tikz}
\usepackage{xcolor}

\definecolor{myblue}{RGB}{30, 160, 210}
\definecolor{mygreen}{RGB}{80, 180, 80}
\definecolor{mypink}{RGB}{230, 100, 130}

\usepackage{amsmath}
\usepackage{amssymb}
\usepackage{amsfonts}
\usepackage{amsthm}
\usepackage{bbm}
\usepackage[mathscr]{eucal}
\usepackage{a4wide}
\usepackage{authblk}
\usepackage{float}
\usepackage{comment}
\usepackage{enumitem}

\usetikzlibrary{positioning, decorations.markings}

\usepackage{setspace}
\usepackage[normalem]{ulem}

\usepackage{mathtools}
\mathtoolsset{showonlyrefs}

\usepackage{hyperref}
\hypersetup{
	colorlinks = true,
	urlcolor = {magenta},
	citecolor = {red},
	linkcolor = {blue}
}
\usepackage{doi}
\usepackage[square,numbers]{natbib}

\usepackage{todonotes}

\usepackage{graphicx}
\usepackage{subcaption}
\usepackage{soul}

\newcommand{\RNum}[1]{\uppercase\expandafter{\romannumeral #1\relax}}

\newcommand{\pr}{\mathbb{P}}								%	short hand for blackboard P

\newcommand{\Prob}[1]{\pr\left(#1\right)}					%	Standard probability command, 
\newcommand{\CProb}[2]{\pr\left(#1 \;\middle|\; #2\right)}	%	Conditional probability command, 
\newcommand{\E}{\mathbb{E}}								%	short hand for blackboard E
\newcommand{\V}{\mathbb{V}}								%	short hand for blackboard V

\newcommand{\Exp}[1]{\E\left[#1\right]}					%	Standard expectation command, argument
\newcommand{\Var}[1]{\V\left[#1\right]}					%	Standard variance command, argument
\newcommand{\CExp}[2]{\E\left[#1 \;\middle|\; #2\right]}	%	Conditional expectation, second 
\newcommand{\1}{\mathbbm{1}}								%	Indication shorthand command
\newcommand{\ind}[1]{\1_{\{#1\}}}	%	Command for indicator where argument is a condition
\newcommand{\indE}[1]{\1_{#1}}					%	Command for indicator where the argument is an event 

\newcommand\numberthis{\addtocounter{equation}{1}\tag{\theequation}}

\allowdisplaybreaks

\newtheorem{theorem}{Theorem}[section]
\newtheorem{lemma}[theorem]{Lemma}

\newtheorem{proposition}[theorem]{Proposition}
\newtheorem{corollary}[theorem]{Corollary}

\theoremstyle{definition}
\newtheorem{definition}[theorem]{Definition}

\newtheorem{remark}[theorem]{Remark}

\numberwithin{equation}{section}

\newcommand{\rF}{\mathrm F}

\newcommand{\bC}{\mathbf C}
\newcommand{\bt}{\mathbf t}

\newcommand{\sT}{\mathscr T}

\newcommand{\cE}{\mathcal E}
\newcommand{\cG}{\mathcal G}

\newcommand{\sM}{\mathscr M}
\newcommand{\cP}{\mathcal P}

\newcommand{\cT}{\mathcal T}

\newcommand{\rU}{\mathrm U}
\newcommand{\cY}{\mathcal Y}

\newcommand{\T}{\mathbb T}

\newcommand{\N}{\mathbb N}
\newcommand{\R}{\mathbb R}
\newcommand{\G}{\mathbb G}

\newcommand{\kM}{\mathcal{M}}

\newcommand{\kH}{\mathcal{H}}

\def\l{\ell}

\newcommand{\bT}{\boldsymbol{T}}
\newcommand{\bA}{\boldsymbol{A}}

\newcommand*{\be}{\begin{equation}}
	\newcommand*{\ee}{\end{equation}}
\newcommand*{\ba}{\begin{aligned}}
	\newcommand*{\ea}{\end{aligned}}

\newcommand{\eps}{\varepsilon}

\newcommand{\cF}{\mathcal F}

\newcommand{\floor}[1]{\lfloor #1\rfloor}

\newcommand{\invisible}[1]{}

\AtBeginDocument{%
  \DeclareTextCommandDefault{\l}{\Lslash}%
  \DeclareTextCommandDefault{\L}{\Lslash}%
}

\title{Finding Adam in noisy trees}

\author[1]{Luc Devroye}
\affil[1]{3480 Rue University, Montréal, QC H3A 2A7, Canada; School of Computer Science, McGill University. }
\author[2]{Gábor Lugosi}
\affil[2]{ICREA, Pg. Lluis Companys 23, 
08010 Barcelona, Spain; 
Department of Economics and Business, Universitat Pompeu Fabra;
Barcelona School of Economics.}
\author[3]{Neeladri Maitra}
\affil[3]{214 Harker Hall, 1305 W Green St, Urbana, IL 61801, USA; Department of Mathematics, University of Illinois Urbana-Champaign.}

\date{}

\begin{document}

	\maketitle

\begin{abstract}
		We consider the problem of finding the root vertex of a random uniform attachment tree, when the union of the unlabeled tree and an Erd\H{o}s-Rényi random graph $\G(n,p)$ is observed. We prove that, as long as $p=o(\log n /n)$, for any $\varepsilon>0$, one can construct a confidence set of vertices of size $K(\varepsilon)$ that depends only on $\varepsilon$ and not on $n$, such that it contains the root with probability at least $1-\varepsilon$. This affirms a conjecture of \citet{CrXu21}. Our approach ranks vertices by their Jordan centrality in the largest component of the subgraph spanned by high-degree vertices. We show that the same approach works in other noise models as well.
	\end{abstract}

%\tableofcontents

%\noindent  \bigskip\\
%{\bf Keywords:}  First-passage percolation, long-range first-passage percolation, competition, competing first-passage percolation, continuous-time branching processes.
%        \\\\
%{\bf MSC Subject Classifications:} Primary: 60K35, Secondary: 60C05.

\section{Introduction}

An active area in combinatorial statistics is \emph{network archaeology}, concerned with large networks that evolve in time. Upon observing the current configuration of the network, one wishes to infer its past properties. For example, one may wish to study the origin of a rumor spreading in social networks, the spread of computer viruses in computer networks, or the spread of a disease upon observing the network of currently infected individuals. 

Large networks that change dynamically over time are often modelled using simple random dynamics, which, despite being relatively simple, can provide a quite accurate macroscopic description of real networks.
The existing literature on network archaeology---see 
\citet{Hai70},
\citet{SZ11}, \citet{ShZa16}, \citet{findingadam}, \citet{LuPe19},
\citet{ReDe19}, \citet{AdDeLuVe21},
\citet{CuDuKoMa15},  \citet{BuMoRa15}, \citet{BuElMoRa17}, \citet{JoLo16}, \citet{JoLo17a},  \citet{KhLo16},
\citet{BaBh20}, \citet{BaHu23}, \citet{CrXu21}, \citet{CoCuLaLaRi24}, \citet{AdFoKhLaTe24}, \citet{AdBrBrBrLu25}, \citet{BrGiLuSu25},
\citet{CoDeLu26},
\citet{BaBrJo26}--–, mostly focuses on the simplest possible kind of networks, that is, trees. In various models of growing random trees, it is well understood to what extent one can identify the tree's origin (i.e., the root) by observing a large unlabeled tree. These models include uniform attachment, linear preferential attachment, and diffusion over regular trees. Perhaps surprisingly, in all these models, the size of the tree does not play a role. In other words, there exist root-finding algorithms that can select a small number of nodes such that the root vertex is among them with high probability, regardless of how large the tree is.

A largely unexplored topic in network archaeology is robustness.
Mis-specified models, noisy observations, and censored data all pose highly nontrivial
challenges both from methodological and mathematical points of view. 
In a canonical model proposed by \citet{CrXu21a}, one aims at
finding the root in random recursive trees based on a noisy attachment model
in which the tree is observed with random edges added to it. More precisely, the statistician
observes the union of a random recursive tree and an Erd\H{o}s-Rényi random graph.
In preferential attachment trees, where the root vertex typically has a very high degree,
root-finding methods that are based on vertex degrees have near-optimal inference in the noiseless case
(\citet{BaHu23}, \citet{CoCuLaLaRi24}). Hence, these methods are highly noise-tolerant, since
sparse Erd\H{o}s-Rényi random graphs do not have high-degree vertices.
On the other hand, the case of the uniform attachment tree is significantly more challenging, since the root is not among the vertices with the highest degree (\citet{devroye1995strong}, \citet{addario2018high}, \citet{eslavadegree}).
Even though \citet{CrXu21a} determine the optimal root-vertex estimator, it is not known if bounded-size confidence sets exist for the root vertex, even if
the noise is a sparse Erd\H{o}s-Rényi random graph with constant average degree. 
\citet{CrXu21a} prove that if the noisy edge probability is less than $\log n/n$, then there exists a confidence set of size at most $n^\gamma$ for some $\gamma<0.8$. At the same time, Crane and Xu conjecture that if the noisy edge probability is $o(\log n/n)$, then the optimal confidence set has size $O_p(1)$.
The main result of this paper is precisely a proof of this conjecture. 
In particular, for every $\varepsilon \in (0,1)$, we construct a confidence set $K$ of size $|K|=O(\log(1/\varepsilon)/\varepsilon)$ that contains the root vertex with probability at least $1-\varepsilon$, independently of the size of the observed graph. The problem of root finding with Erd\H{o}s-Rényi noise is illustrated by Figure \ref{fig:main}.

\begin{figure}[htbp]
    \centering
    \begin{subfigure}[b]{0.6\textwidth}
        \centering
        \includegraphics[width=\textwidth]{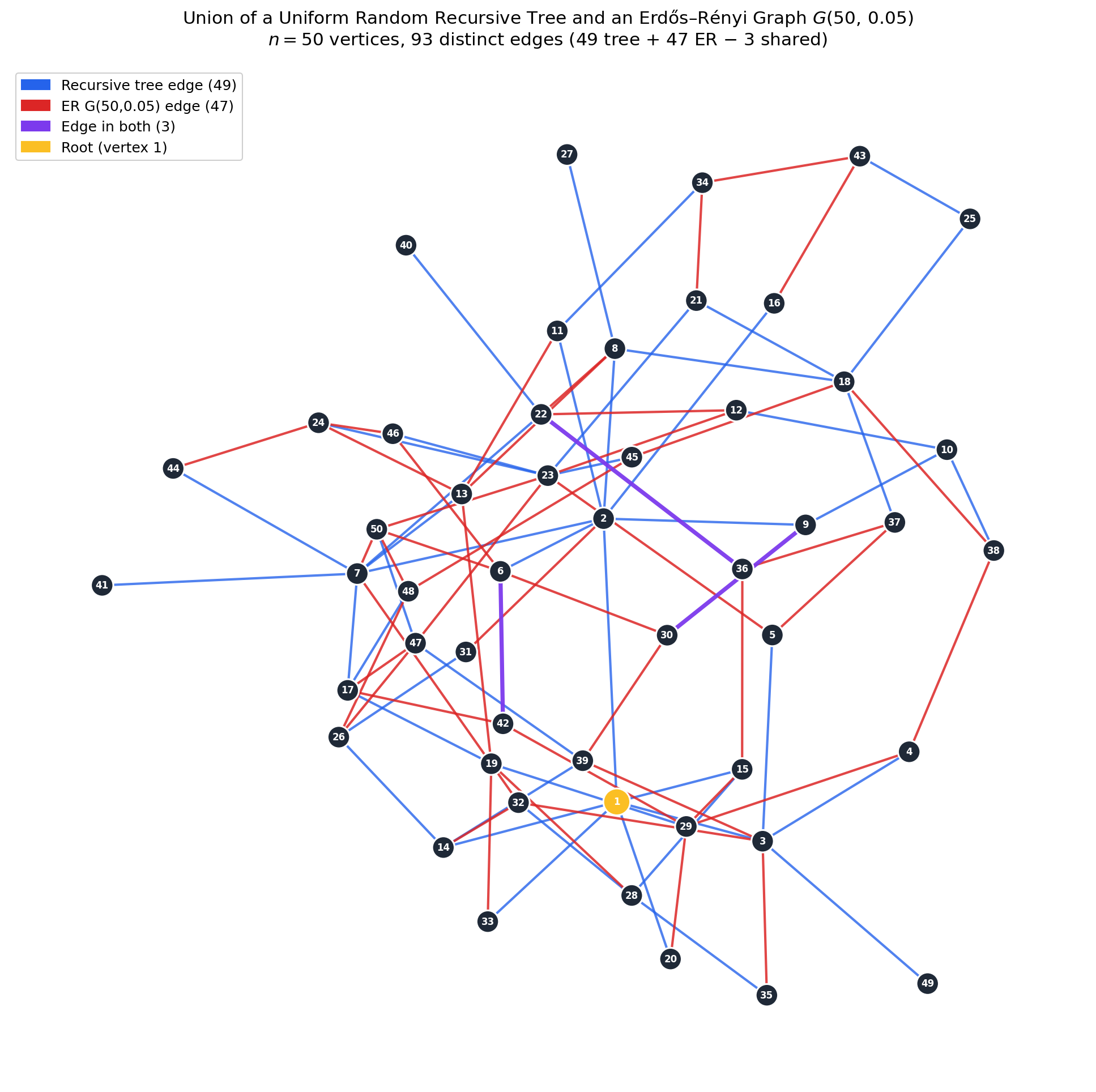}
        \caption{Union of a {\sc urrt} and $\G(50,0.05)$. The root vertex is colored yellow.}
        \label{fig:first}
    \end{subfigure}
    \hfill % Adds horizontal space between the figures
    \begin{subfigure}[b]{0.6\textwidth}
        \centering
        \includegraphics[width=\textwidth]{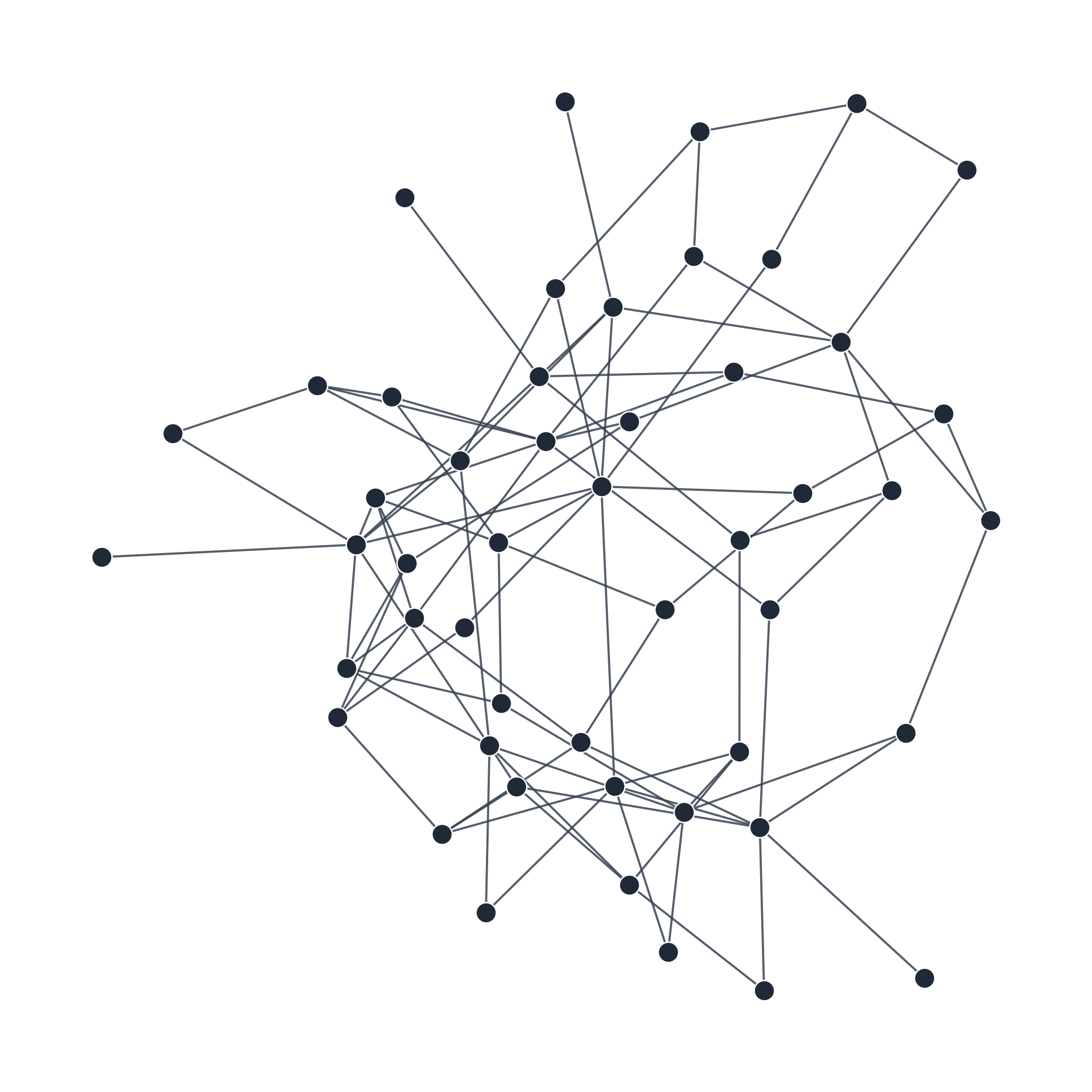}
        \caption{The observed graph.}
        \label{fig:second}
    \end{subfigure}
    \caption{The problem of root finding with Erd\H{o}s-Rényi noise. Upon observing the graph at the bottom, the statistician is asked to find a small set of vertices that contains vertex $1$. 
    }
    \label{fig:main}
\end{figure}

Our construction is based on ``filtering out'' the noise edges belonging to the Erd\H{o}s-Rényi graph by considering the subgraph spanned by vertices whose degree exceeds a carefully chosen threshold. We show that, with high probability, the largest component of the resulting graph is a tree that contains the root vertex. By listing the $|K|$ most central vertices of this tree  (according to their \emph{Jordan centrality}), we obtain the desired confidence set. We also show that our framework can handle noise models different from Erd\H{o}s-Rényi graphs. As an example, we consider \emph{random matching} noise, when the uniform attachment tree is observed with a random perfect matching added to the vertices; see Figure \ref{fig:main2} for an illustration.

\begin{figure}[htbp]
    \centering
    \begin{subfigure}[b]{0.6\textwidth}
        \centering
        \includegraphics[width=\textwidth]{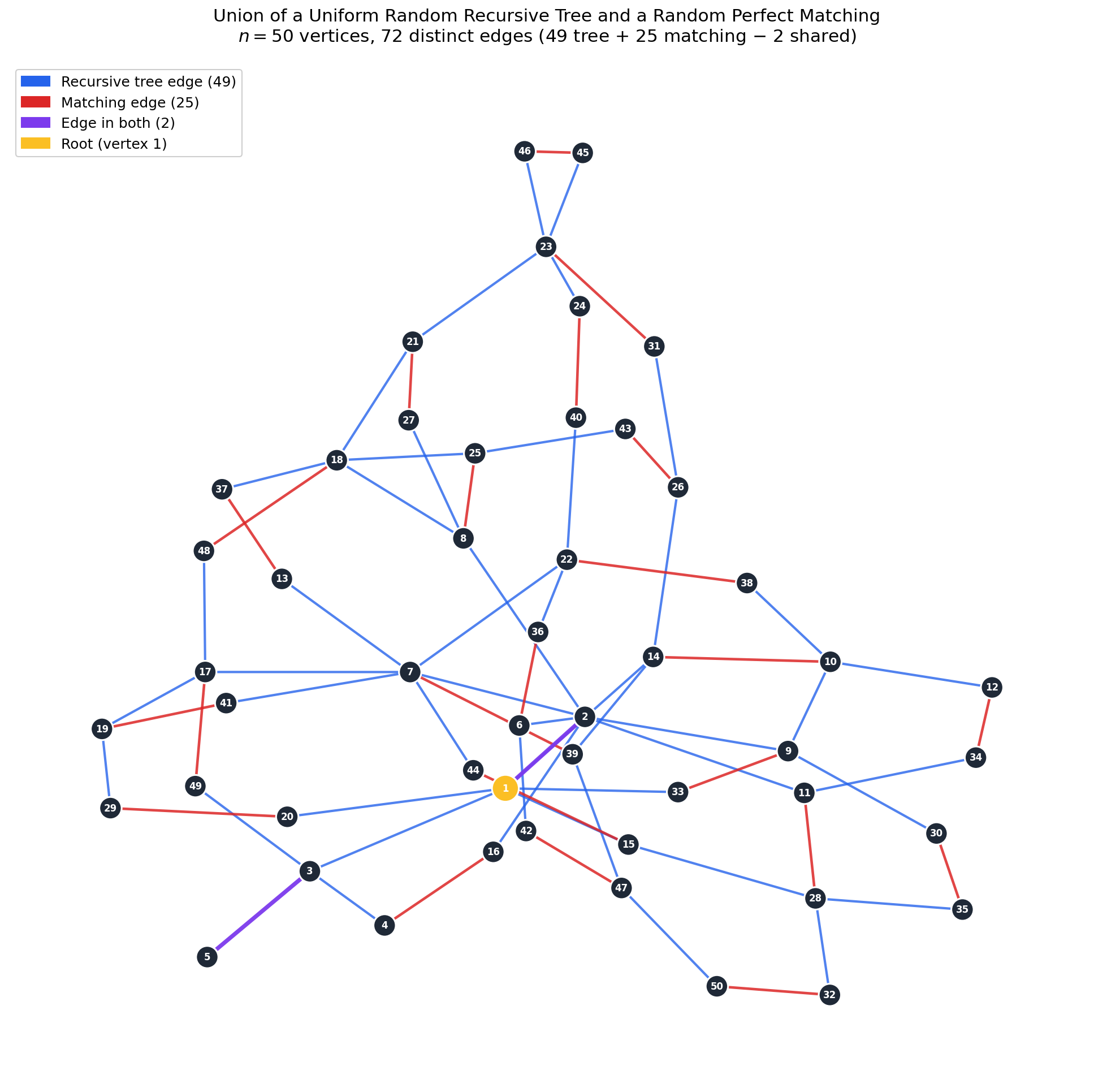}
        \caption{Union of a {\sc urrt} and a random perfect matching on $n=50$ vertices.}
%        \label{fig:first}
    \end{subfigure}
    \hfill % Adds horizontal space between the figures
    \begin{subfigure}[b]{0.6\textwidth}
        \centering
        \includegraphics[width=\textwidth]{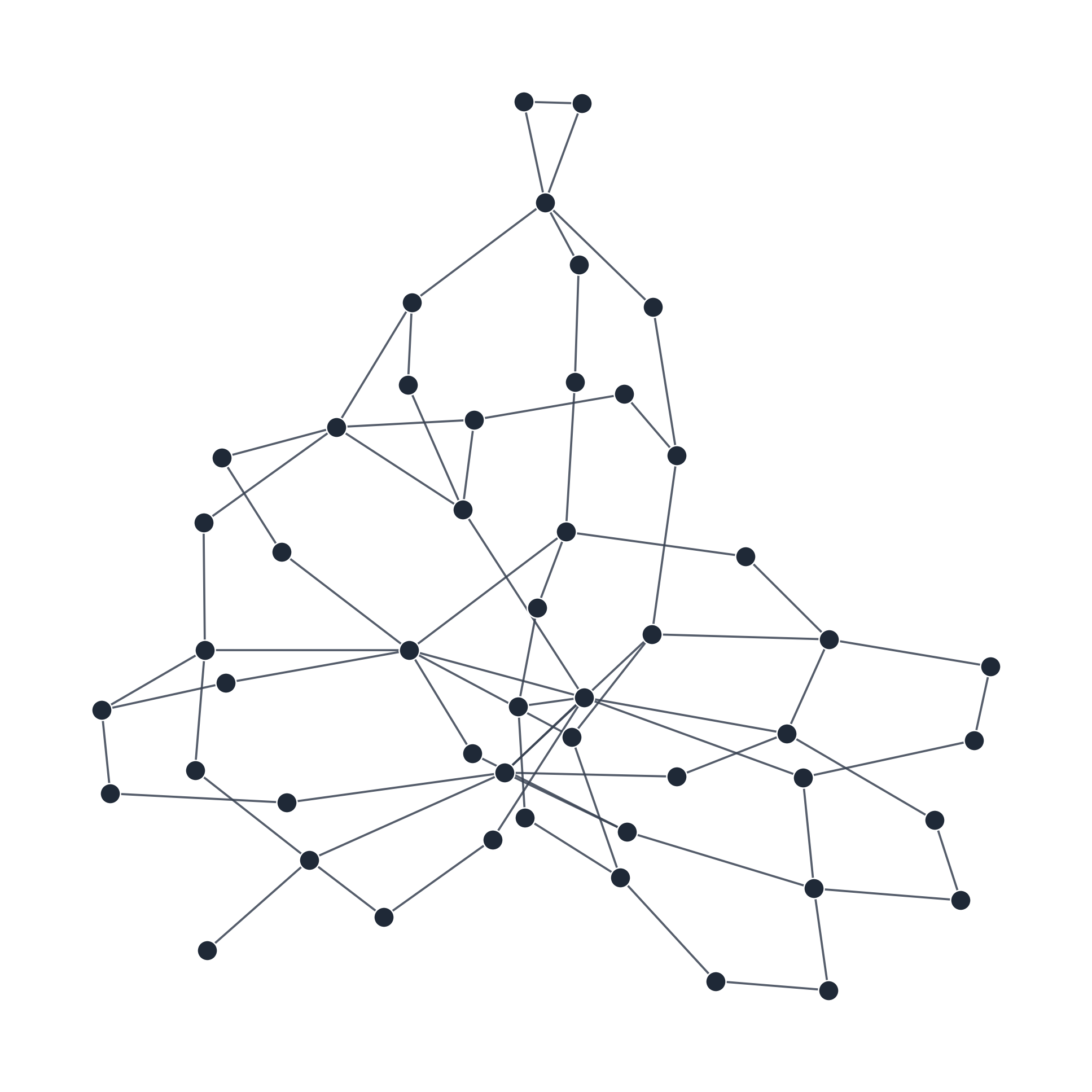}
        \caption{The observed graph.}
%        \label{fig:second}
    \end{subfigure}
    \caption{The problem of root finding with an added perfect matching.}
    \label{fig:main2}
\end{figure}

\subsection{Model definition and main results}
In this section, we state the key definitions and our main results. 

\begin{definition}
    A uniform random recursive tree (\textsc{urrt}) on the vertex set $[n]$ is a connected acyclic graph constructed recursively as follows.
    \begin{itemize}
        \item [(i)] Start with the graph $\cT_1$, where $\cT_1$ is the singleton graph on $\{1\}$ with an isolated vertex $1$.
        \item [(ii)] For each $i \geq 1$, given $\cT_{i}$, to construct $\cT_{i+1}$, introduce the vertex $i+1$, which randomly samples a parent from $\cT_i$ and connects to it via an edge.
    \end{itemize}
    In particular, for each $i\geq 1$, $\cT_i$ is a tree on the vertex set $[i]$.
\end{definition}
Our main interest in this paper is to infer the root vertex $1$.
%which, following \cite{findingadam}, we call the \emph{Adam} of the network. 
In particular, we consider the problem of finding the root when the unlabeled tree $\cT_n$ is observed in the presence of additional random edges. In order to address this problem, 
we first consider suitably pruned sub-forests $\cF_n$ of $\cT_n$, and assume that the combinatorial structure of the forest $\cF_n$ is observed \emph{without the vertex labels}. The first question we address is the following: Observing the pruned sub-forest, given an $\eps>0$, can one construct a confidence set $K$ whose size does not depend on $n$ but only on $\eps$, such that the root vertex lies in $K$ with probability at least $1-\eps$?

The study of the case $\cF_n=\cT_n$ was initiated by Bubeck, Devroye, and Lugosi \cite{findingadam}, who proved that the answer is yes and gave explicit upper bounds on the size of the confidence set $K(\eps)$. 
Addario-Berry, Fontaine, Khanfir, Langevin, and Têtu \cite{AdFoKhLaTe24} show that the  size of $K(\epsilon)$ can be as small as $e^{C\sqrt{\log(1/\epsilon)}}$ for a constant $C>0$.

%In this paper, we consider particular cases of the more general version of the problem with $\cF_n$, and in particular, apply our results in finding the root where the \textsc{urrt} we observe is \textit{polluted with noise}, that is, the tree $\cT_n$ is observed in the presence of extra random edges.

Let us now define the precise types of forests $\cF_n$ we consider. Throughout the paper, for any graph $G$ and a vertex $v$ in it, $d_G(v)$ denotes the degree of $v$ in $G$.

%\begin{definition}[$\alpha$-forests]\label{def:alpha_high_deg_for}
     %The sub-forest of $\cT_n$ spanned by vertices $v \in [n]$ satisfying $d_{\cT_n}(v)>(1-\alpha)\log n$, is called the `$\alpha$-forest' of $\cT_n$, and is denoted by $\cF_n(\alpha)$.
%\end{definition}

%\red{[ N: better to work with the next def. ]}

\begin{definition}[$\alpha$-forests]\label{def:alpha_high_deg_for}
     Let $\alpha\in (0,1)$. 
Define the  \emph{$\alpha$-forest} $\cF_n(\alpha)$ obtained by removing every edge from $\cT_n$ which has at least one end vertex $u$ with $d_{\cT_n}(u)\leq (1-\alpha)\log n$. 
\end{definition}

\begin{remark}\label{rem:site_perco_forest}
    Note that if we let $\cF_n^{(\rm V)}(\alpha)$ to be the subforest of $\cT_n$ spanned by all vertices $v$ with $d_{\cT_n}(v)>(1-\alpha)\log n$, then we obtain $\cF_n(\alpha)$ from $\cF_n^{(\rm V)}(\alpha)$ by including all the vertices that are not in it as isolated vertices.  
    %In other words, in Definition \ref{def:alpha_high_deg_for}, we are percolating by edges, while to create $\cF_n^{(\rm V)}(\alpha)$, we are percolating by vertices. 
    
    In particular, note that a vertex subset of size at least two forms a connected component in $\cF_n(\alpha)$ if and only if it does so in $\cF_n^{(\rm V)}(\alpha)$. 
    %Sometimes it is convenient to think of components of size at least two in $\cF_n(\alpha)$ as those in $\cF_n^{(V)}(\alpha)$.
\end{remark}

%$\cF_n(\alpha)$ can also be thought of as a forest of vertices with high \emph{degree centrality}. 
Our first theorem states that root finding is possible in the forest $\cF_n(\alpha)$.

\begin{theorem}\label{thm:forest}
Fix $\varepsilon>0$. For any $\alpha\in (0,1)$, observing only the combinatorial structure of the graph $\cF_n(\alpha)$ without the vertex labels, one can construct a confidence set $K=K(\varepsilon)$ whose size does not depend on $n$, and such that $\liminf_{n \to \infty}\Prob{1 \in K(\varepsilon)} \geq 1-\varepsilon$.
\end{theorem}

\begin{remark}[Size of $K$]\label{rem:size_K}
%\textcolor{blue}{for some $\delta>0$,} letting $\varepsilon=\frac{4\delta}{1-\delta}$ above, 
Our proofs show that can take $K(\eps)$ to have size at most $C\log(1/\epsilon)/\eps$ for a constant $C>0$. Indeed, our analysis shows that the bound for the confidence set obtained by taking the $K(\epsilon)$ most central vertices 
(according to Jordan centrality) in a \textsc{urrt} derived by \citet{findingadam} is inherited to the current setting. Note that a more careful analysis of Jordan centrality in a \textsc{urrt} yields the improved bound $C/\epsilon$, see \cite[Theorem 9]{CoDeLu26}. We believe that a similar improvement is possible, but we do not pursue this direction, in order to keep the arguments manageable. Similarly, we believe that by ordering vertices according to \emph{rumor centrality} instead of Jordan centrality may give confidence sets of size $e^{C\sqrt{\log(1/\epsilon)}}$ as in \cite{AdFoKhLaTe24}, but we leave this for future research.
\end{remark}

%\red{ [ do we need $\alpha<1-(\ln 4)^{-1}$ for this result? Isn't it needed for the next theorem for the `denoising' argument? The last theorem should be true for any $\alpha \in (0,1)$, right? ] }

This result opens doors towards root finding from noisy observations of the \textsc{urrt} $\cT_n$. To state these results, we begin with a definition.

\begin{definition}[Noisy recursive trees]\label{def:noisy_rec_trees}
For any graph $G=(V,E)$ on the vertex set $V=[n]$, define the graph $\cT_n(G)$ as the union of $\cT_n$ and $G$. That is, $\cT_n(G)$ has vertex set $[n]$, with an edge between two vertices present if and only if the corresponding edge is either present in $\cT_n$ or in $G$. 
%We call the graph $\cT_n(G)$ a `$G$-noisy recursive tree'.
\end{definition}

%We aim to observe the graph $\cT_n(G)$ without its vertex labels and try to find vertex $1$. 
We think of the edges coming from $G$ as \emph{noise}, hiding information about the structure of $\cT_n$ from the statistician. The canonical example, considered by Crane and Xu \cite{CrXu21}, is when $G$ is an
 Erd\H{o}s-R\'{e}nyi random graph.

\begin{theorem}\label{thm:ER_noise}
Fix $\varepsilon>0$. Consider $\cT_n(G)$, where $G=\G(n,\lambda/n)$ is an Erd\H{o}s-R\'{e}nyi random graph independent of $\cT_n$, with edge connection probability $\lambda/n$, where $\lambda=\lambda_n$ may depend on $n$. If $\lambda_n=o(\log n)$, then upon observing the graph $\cT_n(G)$ without its vertex labels, one can construct a confidence set $K=K(\varepsilon)$ of vertices, whose size does not depend on $n$, such that $\liminf_{n \to \infty}\Prob{1 \in K}\geq 1-\varepsilon$.
\end{theorem}

The condition $\lambda_n=o(\log n)$ is sufficient for root finding, but need not be necessary. We leave the question of the exact location for the transition from the possible to the impossible regime as an interesting problem for future research.

%\begin{problem}\label{prob:er_denser}
%Determine the exact threshold $\lambda^*_n$ such that root finding is possible from the noisy recursive tree $\cT_n(G)$ with $G=\G(n,\lambda_n)$ in the sense of Theorem \ref{thm:ER_noise}, if and only if $\lambda_n\leq  \lambda^*_n$.
%\end{problem}

\begin{comment}
Our techniques do not generalize when $\lambda_n=\Omega(\log n)$; see the discussion in Section \ref{sec:discussion}. They do however generalize straightforwardly to \emph{inhomogeneous random graphs}, where the presence of an edge $e$ is determined by a probability $p_e$, as long as
\begin{align*}
    \max_{e\in E(K_n)}p_e=o\left(\frac{\log n}{n} \right).
\end{align*}
In the display above we write $E(K_n)=\{\{i,j\}:i,j\in [n]\}$ to be the set of edges of the complete graph $K_n$.
\end{comment}

To illustrate the generality of our techniques, we show that root finding is also possible when $G$ is a \emph{random perfect matching}: 

\begin{theorem}\label{thm:matching_noise}
Fix $\varepsilon>0$, and assume that $n$ is even. Let $G$ be a uniformly sampled perfect matching on the complete graph on $[n]$, independent of $\cT_n$. Then, observing the unlabeled graph $\cT_n(G)$, it is possible to construct a confidence set $K(\varepsilon)$, such that $\Prob{1\in K(\varepsilon)}\geq 1-\varepsilon$.
\end{theorem}

\begin{remark}
The sizes of $K=K(\eps)$ in both Theorems \ref{thm:matching_noise} and \ref{thm:ER_noise} can be taken to be at most $C\log(1/\eps)/\eps$, as in Remark \ref{rem:size_K}. This is because, as our proof shows, the problem with noise can be broken down into an instance of Theorem \ref{thm:forest}, so that, with high probability, the same set $K$ can be found despite the presence of the noise.
\end{remark}

The techniques that we develop to prove Theorems \ref{thm:forest}, \ref{thm:ER_noise}, and \ref{thm:matching_noise} work under some reasonably general conditions on the noise graph $G$. In Section \ref{Sec:disc} we discuss the applicability of our techniques. 

The outline of the proof of Theorem \ref{thm:forest} is as follows. If all the degrees in the \textsc{urrt} $\cT_n$ were equal to their expected values, then the forest $\cF_n(\alpha)$ would consist of the vertices $\{1,2,\dots,n^\alpha\}$ 
%(as can be proved, e.g., using Lemma \ref{lem:degUandLtail} below). 
In this idealized scenario, $\cF_n(\alpha)$ is simply a \textsc{urrt} on $n^\alpha$ vertices, and standard techniques deliver the confidence set $K$, e.g., by ranking vertices by their Jordan centrality \cite{findingadam}. However, random fluctuations of the vertex degrees 
significantly change the structure of $\cF_n(\alpha)$.  Nevertheless, we show that $\cF_n(\alpha)$ contains an initial chunk of the \textsc{urrt}, up to the first $n^\gamma$ vertices for some $\gamma>0$. Additionally, we show that the connected component containing this chunk (and therefore the root) is the largest component in $\cF_n(\alpha)$. The main remaining challenge is to show that the largest component `behaves' like a \textsc{urrt} (of its own size) in a certain sense, so that ranking vertices by their Jordan centrality can be applied to it to obtain the confidence set $K$. This behavior is established using exchangeability arguments. The proof of Theorem \ref{thm:ER_noise} follows a similar structure, although establishing that the largest component is a  \textsc {urrt}-like object is more challenging.

%\red{[ N: continue from here ]} 

%\red{[ G: Sorry, Neel, I just uncommented some old text. I will add some stuff here. ]} \textcolor{blue}{[ N: ahh okay, no worries. How long till we're done? Are we almost there? :-) ]}

\paragraph{Notation.} For a finite set $A$, $\mathrm{U}(A)$ denotes the law of a uniformly distributed random variable on $A$. For $p\in [0,1]$, by $\mathrm{Ber}(p)$ denotes the law of a Bernoulli random variable with success parameter $p$. For $n\geq 1$ and $p \in [0,1]$, $\mathrm{Bin}(n,p)$ denotes a binomial random variable with $n$ independent trials each with success probability $p$. For any positive integer $N$, we write $[N]$ to denote the set $\{1,2,\dots,N\}$. We use $\preceq$ and $\succeq$ to denote stochastic ordering between random variables.

\paragraph{Organization of the rest of the paper.} We begin with some preliminary results on \textsc{urrt}s in Section \ref{sec:prelims}. In Section \ref{sec:root_find_highdeg}, we prove Theorem \ref{thm:forest}. In Section \ref{sec:root_find_ER}, we show Theorem \ref{thm:ER_noise}. In Section \ref{Sec:matching}, we establish Theorem \ref{thm:matching_noise}. Finally, we summarize our technique and informally discuss an example in Section \ref{Sec:disc}.

\section{Preliminary results}\label{sec:prelims}
In this section, we gather some technical tools. 
%Theorems \ref{thm:forest} and \ref{thm:ER_noise}. 

\subsection{Degrees and their deviations in uniform random recursive trees}
For any $\alpha \in (0,1)$, define the function
\begin{align*}
    f_{\alpha}:[0,1]\to \R,\;\;f_\alpha(x):=x-\alpha+(1-\alpha)\log\left( \frac{1-x}{1-\alpha}\right). \numberthis \label{eq:def_f_alpha}
\end{align*}
The function $f_\alpha$ (see Figure \ref{gfunction}) appears as a \emph{rate} function for large deviation events regarding degrees in the \textsc{urrt} $\cT_n$, as shown below.

For any $v\in [n]$, define
\begin{align*}
    x(v):=\frac{\log v}{\log n}. \numberthis \label{def:x(v)}
\end{align*}
Thus, note that $v=n^{x(v)}$ for any $v \in [n]$.
\begin{definition}[Degrees and offspring]\label{def:deg_off}
    For any $v\in [n]$, by $d_{\cT_n}(v)$ we denote the \emph{degree}, that is, the number of neighbors of $v$ in $\cT_n$. We denote by $c_{\cT_n}(v)$ the number of \emph{offspring} of $v$, i.e., the number of neighbors of $v$ with label larger than $v$ in $\cT_n$. 
\end{definition}

\begin{lemma}\label{lem:degUandLtail}
    For any $\alpha \in (0,1)$ and $v \in [n]$, we have
    \begin{align*}
&        \Prob{c_{\cT_n}(v)>(1-\alpha)\log n}\leq n^{f_\alpha(x(v))},\;\;\text{if}\;\;x(v)>\alpha~;\\& \Prob{c_{\cT_n}(v)<(1-\alpha)\log n}\leq n^{f_\alpha(x(v))},\;\;\text{if}\;\;x(v)<\alpha~.
    \end{align*}
    Furthermore, the above bounds hold when we replace the quantity $c_{\cT_n}(v)$ by $d_{\cT_n}(v)$.
\end{lemma}

%\textcolor{red}{LUC: you must choose! sometimes you use $v$ as in $v \in \cT_n$ and at other times as in $v > \alpha$ or $n^v$.  I like $v$ as a vertex, but then, all instances of the other $v$ must change.  To add the confusion, you use $c_{\cT_n}(v)$ and $c_{\cT_n}(n^v)$ in the lines below.  So please decide on the notation that avoids confusion.} \textcolor{blue}{N: I have changed the exponents of $n$ to $x$ (so that we always integrate over $x$ on $(0,1)$), and referred to $v$ as a vertex throughout. Thus $v=n^x$ (dropping ceiling/floor) is a vertex of the URRT, where $x \in [0,1]$.}
%\textcolor{red}{LUC: Yes, that is good. Unfortunately, $x$ is not unique when $v$ is given as $\frac{\log (v-1)} {\log (n)} \le x \le \frac{\log (v)} {\log (n)}$.}  
\begin{proof}
Observe that for any vertex $v \in \cT_n$, its number of offspring $c_{\cT_n}(v)$ is a sum $\sum_{i=v+1}^n X_i$ of independent random variables, where each $X_i$ is Bernoulli with success probability $1/i$. 
Let $v\in [n]$ be fixed such that $x(v)>\alpha$. By a Chernoff bound, for any $t>0$, we have
\begin{align*}
    \Prob{c_{\cT_n}(v)>(1-\alpha)\log n}\leq \frac{\prod_{i=n^{x(v)}+1}^n\left((e^t-1)/i+1 \right)}{\exp\left(t(1-\alpha)\log n \right)}~,
\end{align*}
%\textcolor{red}{LUC: the argument $n^x$ is not integer.}\textcolor{blue}{[N: addressed; on the right hand side I've put ceiling, for the left hand side $c_{\cT_n}(x)$ is the same as $c_{\cT_n}(\lceil n^x \rceil)$, as mentioned before the proof]}
which, using the inequality $1+y\leq e^y$ applied for $y=(e^t-1)/i$, the upper bound $\sum_{i=n^{x(v)}+1}^n\frac{1}{i}\leq (1-x(v))\log n$, and choosing $t=\log\left(\frac{1-\alpha}{1-x(v)} \right)>0$ yields that the last upper bound is at most $n^{f_\alpha(x(v))}$. 

Let $x<\alpha$. From Chernoff's lower tail bound,
\begin{align*}
    \Prob{c_{\cT_n}(v)<(1-\alpha)\log n}&=\Prob{-t\sum_{i=n^{x(v)}+1}^n X_i>-t(1-\alpha)\log n}\\&\leq \exp\left(t(1-\alpha)\log n \right)\prod_{i=n^{x(v)}+1}^n\left(1-\frac{1-e^{-t}}{i}\right).
\end{align*}
Apply the inequality $1+y<e^y$ with $y=-\frac{1-e^{-t}}{i}$, use the upper bound $\sum_{i=n^{x(v)}+1}^n\frac{1}{i}\leq (1-x(v))\log n$ as before and choose $t=\log \left(\frac{1-x(v)}{1-\alpha} \right)>0$ to obtain the desired bound of $n^{f_\alpha(x(v))}$.

The final assertion follows easily by noting that $c_{\cT_n}(1)=d_{\cT_n}(1)$, and $c_{\cT_n}(i)+1=d_{\cT_n}(i)$ whenever $i \neq 1$.
%, and that both $x(v)>\alpha$ and $x(v)<\alpha$ are \emph{open} conditions.
\end{proof}

Next, we record some properties of the rate function $f_\alpha$.

\begin{lemma}
    Fix any $\alpha \in (0,1)$, and recall the function $f_\alpha$ from \eqref{eq:def_f_alpha}. $f_\alpha$ defines a strictly concave function on $[0,1]$, with $\lim_{x \nearrow 1}f_\alpha(x)=-\infty$, $f_\alpha$ has a unique root at $x=\alpha$, and $f_\alpha<0$ for all $x \neq \alpha$. 
\end{lemma}
\begin{proof}
    Strict concavity may be checked via a straightforward verification that $f''_\alpha(x)<0$.  $x\nearrow 1$ is an easy calculus exercise. Now, $\alpha$ is a root of $f_\alpha(x)$. Since the derivative $f'_\alpha=1-(1-\alpha)/(1-v)$ changes its sign about $\alpha$, this root must also be unique because of concavity.
\end{proof}
For any $\alpha\in [0,1)$ let $g_\alpha(x):=x+f_\alpha(x)$, and define 
\begin{align*}
    \gamma=\gamma(\alpha):=\inf\{x>0:g_\alpha(x)=0\}=\inf\{x>0:x+f_\alpha(x)=0\}~. \numberthis \label{eq:def_gamma}
\end{align*}

\begin{figure} [H]
\centering
\includegraphics[scale=0.86]{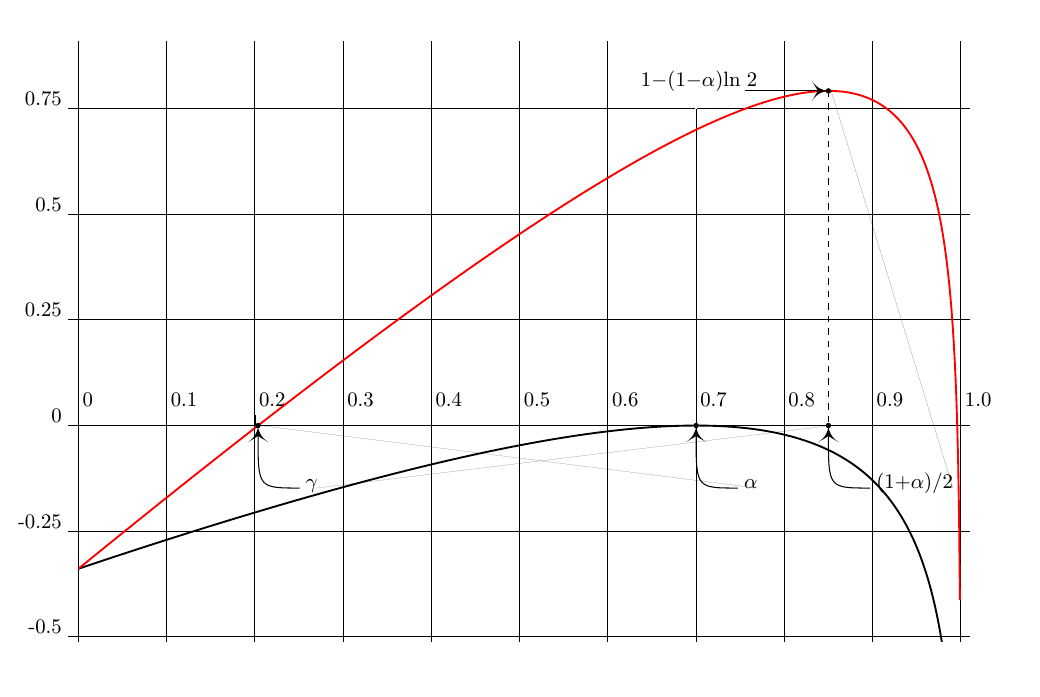}
\caption{The function $g_\alpha (x) = x + f_\alpha(x)$ versus $x \in [0,1)$ is shown in red, while $f_\alpha(x)$ is drawn in black. The parameter $\gamma=\gamma(\alpha)$ is the leftmost point $x$ with $g_\alpha (x) = 0$.}
\label{gfunction}
\end{figure}
%\textcolor{red}{LUC: I added a few bits to the figure.}
%\red{[ introduced a $\log n$ in the error, check and correct wherever the lemma below was used ]}

\begin{lemma}\label{lem:gamma_cont}
For all $\alpha \in (0,1)$,
$g_\alpha (x) = x + f_\alpha(x)$ is a continuous strictly concave function on $[0,1)$, attaining a unique maximum value of $1-(1-\alpha) \log 2$
at $(1+\alpha)/2$. 
The unique solution on $[0,(1+\alpha)/2]$  of $g_\alpha (x)=0$, denoted by $\gamma (\alpha)$, satisfies $\gamma (\alpha) < \alpha/2$.  Finally, $\alpha \mapsto \gamma(\alpha)$ defines a continuous function on $(0,1)$.
\end{lemma}
\begin{proof}
As $f_\alpha$ is strictly concave with a unique maximum at $x = \alpha$, $x+f_\alpha (x)$ too is strictly concave.  
%By continuity of $f_\alpha$ on $[0,1)$, $g_\alpha$ too is continuous on that interval. 
The function $x+f_\alpha (x)$ 
has derivative
$2 - (1-\alpha)/(1-x)$, which is montonically decreasing, reaching the value $0$ at $x=(1+\alpha)/2$. 
Thus, $g_\alpha$ is strictly increasing on $[0, (1+\alpha)/2]$ from $f_\alpha (0) = -\alpha + (1-\alpha) \log (1-\alpha) < 0$ to $g_\alpha((1+\alpha)/2) = 1 − (1 − \alpha) \log 2 > 0$ on that interval, and thus, there is a unique point, $\gamma (\alpha)$, in $(0,\alpha)$ where $x+f_\alpha (x)=0$. 
Finally, $\gamma (\alpha) < \alpha/2$ since $\alpha/2 + f_\alpha (\alpha/2) = (1-\alpha) \log ((1-\alpha/2)/(1-\alpha)> 0$. 
\end{proof}

We conclude this section by showing that all sufficiently early vertices have large degrees in the \textsc{urrt} $\cT_n$.
\begin{lemma}\label{lem:chernoff_deg_tail}
    For any $\gamma'<\gamma(\alpha)$,
    %where $\gamma(\alpha)$ is defined as in \eqref{eq:def_gamma},
    \begin{align*}
    \Prob{\min_{v: v\leq n^{\gamma'}}  d_{\cT_n}(v)<(1-\alpha)\log n  }
    =O\left(n^{\gamma'+f_\alpha(\gamma')}  \right)=o(1)~.
    \end{align*}
    The same inequality holds when $d_{\cT_n}(v)$ is replaced by $c_{\cT_n}(v)$.
\end{lemma}

\begin{proof}
Since $\gamma$ is the first zero of $g_\alpha$, and $g_\alpha(0)=f_\alpha(0)=-\alpha+(1-\alpha)\log(1/(1-\alpha))<0$ for any $\alpha \in (0,1)$, $g_\alpha(\gamma')<0$. Thus, $n^{g_\alpha(\gamma')}=o(1).$ Furthermore, we also have $\gamma(\alpha)<\alpha$ from 
Lemma \ref{lem:gamma_cont}.
%\eqref{eq:gamma_less_alpha}. 
By the union bound,
\begin{align*}
\Prob{\min_{v: v\leq n^{\gamma'}}  d_{\cT_n}(v)<(1-\alpha)\log n  }
 &\leq \sum_{i=1}^{\lfloor n^{\gamma'}\rfloor}\Prob{d_{\cT_n}(v)<(1-\alpha)\log n}\\&\leq n^{\gamma'}\Prob{d_{\cT_n}(\lfloor n^{\gamma'}\rfloor )<(1-\alpha)\log n},\end{align*}
where in the last inequality we use that for any $u,v\in [n]$ with $v>u$ one has $d_{\cT_n}(u)\succeq d_{\cT_n}(v)$. Recalling the definition of $x(v)$ from \eqref{def:x(v)}, observe that $x(\lfloor n^{\gamma'} \rfloor)\leq \gamma'$. Using  that $f_\alpha(\cdot)$ is increasing on $(0,\gamma(\alpha))$, by Lemma \ref{lem:degUandLtail}, the last display is
\begin{align*} O\left( n^{\gamma'+f_\alpha(\gamma')}\right)=O\left(n^{g_\alpha(\gamma')}\right)=o(1)~.
\end{align*} The argument with $d_{\cT_n}(v)$ replaced by $c_{\cT_n}(v)$ is analogous.
\end{proof}

\subsection{Subtrees in uniform random recursive trees} In this section, we establish some properties on the sizes of subtrees in \textsc{urrt}s. 
For any graph $G=(V(G), E(G))$, a subset of edges $S\subset E(G)$, and a vertex $v\in V(G)$, denote by $G(v; S)$ the connected component containing $v$ in the subgraph of $G$ constructed by removing all edges in $S$. For a subgraph $H$ of $G$, and $v \in V(G)$ we define
\begin{align*}
    H(v;S):=H(v;S\cap E(H))~.
\end{align*}
Note that $H(v; S)$ may be the empty graph, e.g., if $v \notin H$.
%\red{[ check if we actually need the def for $v\notin H$ ]}
Observe that if $\cT_n$ is a \textsc{urrt} on the vertex set $[n]$, and if we define
\begin{align*}
    E_v:=\{\{i,j\}\in E(\cT_n):i,j\in [v]\}~,
\end{align*}
then $\cT_n(v;E_v)$ is precisely the descendant subtree of $v$ in $\cT_n$, i.e., the subtree of $\cT_n$ spanned by all vertices with label larger than $v$ whose unique path to the root $1$ passes through $v$. To see this, note that in removing the edges of $E_v$, we remove the edge connecting $v$ to its parent, while the edges in the subtree of $v$ are untouched.  
  We think of the subtree $\cT_n(v; E_v)$ as rooted at $v$, and in particular, $\cT_n$ is rooted at vertex $1$.

\subsubsection{Coupling with uniforms}\label{sec:coup_with_unifs}
To analyze subtree sizes in \textsc{urrt}s, it is useful to couple the \textsc{urrt} with a sequence of $\mathrm{U}[0,1]$ random variables, which we define next. First, observe that the \textsc{urrt} $\cT_n$ is completely determined by a sequence of random variables $\cY_i\stackrel{d}{=} \rU[i-1]$, for $i=2,3,\dots,n$, by simply letting $p(i)=\cY_i$, for $i=2,3,\dots,n$, where recall $p(i)$ is the parent of $i$ in $\cT_n$.

We may couple the variables $\cY_i$ with a sequence of independent uniform random variables $U_1,\dots, U_{n-1}\stackrel{d}{=} \rU[0,1]$ on the interval $[0,1]$ as follows. First, let $E^{(1)}_1=[0,1]$, and for $i=2,3,\dots,n$ let $E^{(i)}_1,\dots,E^{(i)}_i$ denote the \emph{spacings} formed by the random variables $U_1,\dots,U_{i-1}$ on the interval $[0,1]$. Formally, let the order statistic of $U_1,\dots,U_{i-1}$ be $U^{(1)}_{i-1}<U^{(2)}_{i-1}<\dots<U^{(i-1)}_{i-1}$. Then with the convention $U^{(0)}_{i-1}=0$ and $U^{(i)}_{i-1}=1$, we define intervals
\begin{align*}
    E^{(i)}_j:=[U^{(j-1)}_{i-1},U^{(j)}_{i-1}]\textrm{ for all }j=1,2,\dots,i~. \numberthis \label{eq:def_spacings}
\end{align*}
We couple $\cY_2,\dots,\cY_n$ with $U_1,\dots,U_{n-1}$ as
\begin{align*}
    \cY_i\textrm{ is the random index }I_{i-1}\textrm{ such that }U_{i-1}\in E^{(i-1)}_{I_{i-1}}\numberthis \label{eq:coupling_YandU}, 
\end{align*}
for all $i=2,3,\dots,n$. It is straightforward to check that $I_{i-1}$ is uniformly distributed on $[i-1]$, and that $I_1,\dots, I_{n-1}$ are independent, so that this is a valid coupling.

This coupling lets us control subtree sizes in the \textsc{urrt} $\cT_n$. In particular, we have the following:
\begin{lemma}[Subtrees and uniforms]\label{lem:subtrees_unifs}
    Under the coupling \eqref{eq:coupling_YandU}, for any $i \in [n]$, $\cT_n(i;E_i)\setminus\{i\}=\{j\geq i+1:U_{j-1}\in E^{(i)}_i\}$. In particular, $|\cT_n(i;E_i)|\stackrel{d}{=}1+\mathrm{Bin}(n-i,|E^{(i)}_i|)$. 
\end{lemma}
%\begin{proof}
%    As the first assertion implies the second due to the independence of the variables $U_1,\dots, U_{n-1}$, we only prove the first. A path in $\cT_n$ with decreasing labels is called \emph{decreasing}. Thus, $\cT_n(i; E_i)\setminus\{i\}$ is precisely the collection of all $j\geq i+1$ that have a decreasing path to $i$.
    
 %   Now consider any $j\geq i+1$ such that $U_{j-1}\in E^{(i)}_i$. Observe that there is a unique decreasing path from $j$ to $i$ in $\cT_n$ given by $(j^{(0)}=j,j^{(1)},\dots,j^{(l)}=i)$, where $j^{(1)},\dots,j^{(l)}$ are the unique collection of indices satisfying the nested inclusions $E^{(j-1)}_{j^{(1)}}\subseteq E^{(j-2)}_{j^{(2)}}\subseteq \dots \subseteq E^{(j-l)}_{j^{(l)}}=E^{(i)}_i$. The same observation shows that any $j\geq i+1$ with $U_{j-1}\in E^{(i)}_{k}$ with $k\neq i$ cannot have a decreasing path to $i$ since $E^{(i)}_k \cap E^{(i)}_i=\varnothing$. Thus, the vertices $j\geq i+1$ that have a decreasing path to $i$ are precisely those indices $j\geq i+1$ that satisfy $U_{j-1} \in E^{(i)}_i$.
%\end{proof}

%\red{[ pic here? ]}

\subsubsection{Subtree sizes}
In this section, we use Lemma \ref{lem:subtrees_unifs} to prove bounds on the sizes of subtrees in $\cT_n$. We first recall some standard facts of uniform spacings; see, e.g., \citet{BiDe15}.
%\textcolor{red}{***G: This is all known stuff, so I removed the proof and added a reference.***}

\begin{lemma}[Properties of uniform spacings]\label{lem:properties_uspacings}
 Recall for each $i=2,\dots,n$ the variables $|E^{(i)}_j|=U^{(j-1)}_{i-1}-U^{(j)}_{i-1}$, where $U^{(1)}_{i-1}<\dots<U^{(i-1)}_{i-1}$ are the order statistics of $U_1,\dots,U_{i-1}$.
 \begin{itemize}
     \item [(i)] The variables $|E^{(i)}_1|,\dots |E^{(i)}_i|$ are identically distributed and satisfy $$\Prob{|E^{(i)}_1|>x}=(1-x)^{i-1},x\in [0,1].$$
     \item [(ii)] $|E^{(i)}_1|\stackrel{d}{=}1-U^{1/(i-1)},$ where $U$ is uniformly distributed on $[0,1]$.
     \item [(iii)]  $|E^{(1)}_1|=1$, and for $i\geq 2$, $|E^{(i)}_1|\prec E/(i-1)$, where $E\stackrel{d}{=} \mathrm{Exp}(1)$ is exponentially distributed.      
 \end{itemize}
\end{lemma}
%\begin{proof}
%    Note that (ii) follows from (i). For (iii), recall that $U\stackrel{d}{=}e^{-E}$, so that using (ii) $|E^{(i)}_1|\stackrel{d}{=}1-e^{-E/(i-1)}\leq E/(i-1)$ by the inequality $1+x\leq e^x$. The identical distribution claim of (i) is also straightforward by distributional symmetry. For the tail behaviour in (i), note that $|E^{(i)}_1|$ is distributed as the smallest of $(i-1)$ independent uniforms on $[0,1]$. 
%\end{proof}

Next, we use Lemma \ref{lem:properties_uspacings} to prove an upper bound for the $k$-th moment of the subtree size of a vertex in a \textsc{urrt}.

\begin{proposition}[Subtree moment upper bound]\label{prop:subtree_mom_UB}
    There are constants $C_1=C_1(k), C_2= C_2(k)$ such that for all $1\leq i \leq n$, $\Exp{|\cT_n(i;E_i)|^k}\leq C_1+C_2\left(\frac{n-i}{i} \right)^k$.
\end{proposition}

To prove this proposition, first we recall the following moment bound from \citet{latala1997estimation}.
\begin{corollary}[Corollary 3 of \cite{latala1997estimation} applied with $X_i\stackrel{d}{=}\mathrm{Ber}(p)$]\label{cor:latala_mom_bound}
Let $X\stackrel{d}{=} \mathrm{Bin}(n,p)$. There is a universal constant $c$ such that for all $n,k\geq 1$ and $p\in [0,1]$
\begin{align*}
    \Exp{X^k}\leq \left(c\frac{k}{\ln k} \right)^k \cdot \max\{(np)^k,np\}.
\end{align*}
\end{corollary}

\begin{proof}[Proof of Proposition \ref{prop:subtree_mom_UB}]
   The result is trivial for $i=1$, so we assume $i>1$. By Lemmas \ref{lem:subtrees_unifs} and \ref{lem:properties_uspacings}, 
\begin{align*}
|\cT_n(i;E_i)|\stackrel{d}{=}1+\mathrm{Bin}(n-i,|E^{(i)}_i|)\stackrel{d}{=}1+\mathrm{Bin}(n-i,|E^{(i)}_1|)\prec 1+\mathrm{Bin}(n-i,\min\{1, E/(i-1)\}),
\end{align*}
where $E$ is exponential,
    so that for $k \geq 1$,
\begin{align*}
\Exp{|\cT_n(i;E_i)|^k}\leq 2^{k-1}\left(1+\Exp{Y^k} \right),
        \numberthis \label{eq:k_mom_subtree_UB_1}
\end{align*}
where
$Y\stackrel{d}{=}\mathrm{Bin}(n-i,\min\{1,E/(i-1)\})$.
    By Lata{\l}a's inequality,
%and bounding $\min\{1,E/(i-1)\}$ by  $E/(i-1)$,
\begin{align*}
\Exp{Y^k}&\leq M(k)\left(\Exp{\left(\frac{(n-i)E}{i-1} \right)^k}+\Exp{\frac{(n-i)E}{i-1} } \right)\\
        &=M(k)\left(\Exp{\left(\frac{(n-i)E}{i-1} \right)^k}+\frac{n-i}{i-1}  \right),
    \end{align*}
    where $M(k)\geq 1$ is a constant depending only on $k$. Note that
    \begin{align*}
\Exp{\left(\frac{(n-i)E}{i-1} \right)^k}=k!\left(\frac{n-i}{i-1}\right)^k.
    \end{align*} 
Further, if $n-i>i-1$, we have $\frac{n-i}{i-1}\leq k!\left(\frac{n-i}{i-1} \right)^k$, and if $n-i\leq i-1$, we have 
    $$
    \frac{n-i}{i-1}+ k!\left(\frac{n-i}{i-1} \right)^k\leq 1+k!~,
    $$
    so that
    \begin{align*}
        \Exp{Y^k}\leq M(k)\cdot\max\left\{2k!\left(\frac{n-i}{i-1} \right)^k,1+k! \right\}.
    \end{align*}
    Thus from \eqref{eq:k_mom_subtree_UB_1}, taking $C_1(k)=2^{k-1}+2^{k-1}M(k)(1+k!)$ and $C_2(k)=2^{2k}M(k)k!$,
    \begin{align*}
        \Exp{|\cT_n(i;E_i)|^k}&\leq 2^{k-1}+2^{k-1}M(k)(1+k!)+2^{k}M(k)k!\left(\frac{n-i}{i-1}\right)^k\\
        &\leq C_1(k)+C_2(k)\left(\frac{n-i}{i} \right)^k~.
    \end{align*}
\end{proof}

Next, we need a uniform tail bound on the subtree size of any vertex in the \textsc{urrt}.

\begin{proposition}\label{prop:unif_subtree_tail}
     As $n\to \infty$, $\Prob{\sup_{i \in [n]}|\cT_n(i;E_i)|/(n/i)>6 \log n}=O\left(\frac{\sqrt{\log n}}{n^{0.9}} \right)$. %In particular, for any $\lambda\geq 6 \log n$, we have $\Prob{\sup_{v \in [n]}|\cT_n(v;[v])|/(n/v)>\lambda}=o(1)$.
\end{proposition}
\begin{proof}
By a union bound, for any $\lambda>1$,
\begin{align*}
\Prob{\sup_{i \in [n]}|\cT_n(i;E_i)|/(n/i)>\lambda}
&\leq \sum_{i=2}^{n-1}\Prob{|\cT_n(i;E_i)|>\frac{\lambda n}{i} }\\
&\leq \sum_{i=2}^{n-1}\Prob{Y_i>\frac{\lambda n}{i} -1},\numberthis \label{eq:subtreetail_UB}
\end{align*}
since the terms corresponding to $i=1$ and $i=n$ equal $0$, and where for each $i \geq 2$, $Y_i \stackrel{d}{=} \mathrm{Bin}(n-i,\min\{1,E/(i-1)\})$, and we use Lemmas \ref{lem:subtrees_unifs} and \ref{lem:properties_uspacings} for the last inequality. Now, for any $i\geq 2$, by a Chernoff bound, for any $t>0$,
\begin{align*}
&\Prob{Y_i>\frac{\lambda n}{i} -1}%\\
%&=\Exp{\CProb{Y_i>\frac{\lambda n}{i} -1}{E}}
\\&\leq \Exp{\exp\left(-t\left( \frac{\lambda n}{i} -1 \right) \right)\CExp{e^{tY_i}}{E}}\\
&=\Exp{\exp\left(-t\left( \frac{\lambda n}{i} -1 \right)+(n-i)\log\left(1+\min\left\{1,\frac{E}{i-1}\right\}(e^t-1)  \right)\right)}.    
\end{align*}
Let us now fix the choice $\lambda=\lambda_n=6\log n$, and note that
\begin{align*}
&\Exp{\exp\left(-t\left( \frac{\lambda n}{i} -1 \right)+(n-i)\log\left(1+\min\left\{1,\frac{E}{i-1}\right\}(e^t-1)  \right)\right)}\\
&\leq\Exp{\exp\left(-t\left( \frac{\lambda n}{i} -1 \right)+(n-i)\log\left(1+\min\left\{1,\frac{E}{i-1}\right\}(e^t-1)  \right)\right)
\ind{\frac{\lambda n}{i} -1 > \frac{E(n-i)}{i-1}}  }
\\
&\hspace{10 pt}+\Prob{E>\left( \frac{\lambda n}{i} -1 \right)\frac{i-1}{n-i}}\\
&\leq \Exp{\exp\left(-t\left( \frac{\lambda n}{i} -1 \right)+(n-i)\log\left(1+\frac{E(e^t-1)}{i-1} \right)\right)
\ind{\frac{\lambda n}{i} -1 > \frac{E(n-i)}{i-1}} }
+O\left(n^{-6} \right)\\
&\leq \Exp{\exp\left(-t\left( \frac{\lambda n}{i} -1 \right)+\frac{E(n-i)}{i-1}(e^t-1) \right)
\ind{\frac{\lambda n}{i} -1 > \frac{E(n-i)}{i-1}} }
+o\left(n^{-5} \right),
\end{align*}
    where for the last inequality above we use $1+x\leq e^x$, and note that the constant hidden by the $o(\cdot)$ term above can be taken to be independent of $i$. Thus, to conclude the proof, it suffices to show that 
    \begin{align*}
        \sum_{i=2}^{n-1}\Exp{\exp\left(-t\left( \frac{\lambda n}{i} -1 \right)+\frac{E(n-i)}{i-1}(e^t-1) \right)\ind{\frac{\lambda n}{i} -1 > \frac{E(n-i)}{i-1}} }
        =O\left(\frac{\sqrt{\log n}}{n} \right). \numberthis \label{eq:toshow_prop_unif_subtreetail}
    \end{align*}
  On the event $\frac{\lambda n}{i} -1 > \frac{E(n-i)}{i-1}$, $\log\left(\frac{\lambda n/i-1}{E(n-i)/(i-1)} \right)>0$, so letting $t=\log\left(\frac{\lambda n/i-1}{E(n-i)/(i-1)} \right)$, we have
\begin{align*}
&\Exp{\exp\left(-t\left( \frac{\lambda n}{i} -1 \right)+\frac{E(n-i)}{i-1}(e^t-1) \right)\ind{\frac{\lambda n}{i} -1 > \frac{E(n-i)}{i-1}}}\\
&\leq \Exp{\exp\left(\left( \frac{\lambda n}{i} -1 \right)-\frac{E(n-i)}{i-1}-\left( \frac{\lambda n}{i} -1 \right)\log\left(\frac{\frac{\lambda n}{i} -1}{E(n-i)/(i-1)} \right)\right)}\\
&\leq \exp\left(\left( \frac{\lambda n}{i} -1 \right)\left(1-\log\left(\frac{(\lambda n-i)(i-1)}{i(n-i)} \right) \right) \right)\Exp{e^{-\frac{n-i}{i-1}E+\left( \frac{\lambda n}{i} -1 \right)\log E}}.
  \end{align*}
Note that 
\begin{align*}
\Exp{e^{-\frac{n-i}{i-1}E+\left( \frac{\lambda n}{i} -1 \right)\log E}}
&=\int_{0}^{\infty}e^{-x}x^{\lambda n/i -1}e^{-x(n-i)/(i-1)}dx\\
&=\frac{\Gamma(\lambda n/i)}{(1+(n-i)/(i-1))^{\lambda n/i}}\\
&=\frac{\Gamma(\lambda n/i)}{((n-1)/(i-1))^{\lambda n/i}},
  \end{align*}
where
$\Gamma(z):=\int_{0}^\infty t^{z-1}e^{-t}dt$ for $z>0$. Thus, the left-hand side of \eqref{eq:toshow_prop_unif_subtreetail} is at most
\begin{align*}
\sum_{i=2}^{n-1} \exp\left(\frac{\lambda n}{i}\left(1-\log\left(\frac{(\lambda n-i)(i-1)}{i(n-i)} \right)-\log \left(\frac{n-1}{i-1} \right) \right)-1+\log\left(\frac{(\lambda n-i)(i-1)}{i(n-i)} \right) \right)\Gamma\left( \frac{\lambda n}{i} \right).
  \end{align*}
Using the standard upper bound $\Gamma(z)\leq \sqrt{2\pi/z}(z/e)^ze^{1/12z}$ valid whenever $z>0$, for $z=\lambda n/i >1$, the last sum is at most 
\begin{align*}
\sum_{i=2}^{n-1}\exp\left(\frac{\lambda n}{i}\left(\log\left(\frac{\lambda n}{i}\right)-\log\left(\frac{(\lambda n-i)(n-1)}{i(n-i)} \right)  \right) \right)\frac{\sqrt{2\pi}}{e^{11/12}}\frac{(\lambda n-i)(i-1)}{i(n-i)}\sqrt{\frac{i}{\lambda n}}.
\end{align*}
Since for all large $n$, $\frac{\sqrt{2\pi}}{e^{11/12}}\frac{(\lambda n-i)(i-1)}{i(n-i)}\sqrt{\frac{i}{\lambda n}}\leq C\sqrt{\lambda}n$ for some universal constant $C>0$, the last sum is at most
\begin{align*}
&C\sqrt{\lambda}n \sum_{i=2}^{n-1}\exp\left(-\frac{\lambda n}{i}\log\left(\frac{(\lambda n-i)(n-1)}{\lambda n(n-i)} \right) \right)\\&\leq C\sqrt{\lambda}n\sum_{i=2}^{n-1}\exp\left(-\frac{\lambda n}{i}\log\left(1+\frac{i}{n}\left(1-\frac{1}{\lambda}\right) \right)-\frac{\lambda n}{i}\log\left(\frac{n-1}{n} \right) \right)~, \numberthis \label{eq:sum_UB}
\end{align*}
where to obtain the inequality above, since $\lambda=6 \log n\to \infty$, we use $$\frac{(\lambda n-i)n}{\lambda n(n-i)}=\frac{1-i/\lambda n}{1-i/n}\geq 1+\frac{i}{n}\left(1-\frac{1}{\lambda}\right).$$ 
Finally, for $x\in [0,1)$ using the inequalities $\log(1+x)\geq x-x^2/2$ and $\log(1-x)\geq -x/(1-x)$, the right-hand side of \eqref{eq:sum_UB} is bounded from above by
\begin{align*}
C\sqrt{\lambda}n\sum_{i=2}^{n-1}\exp\left(-\lambda+1-\frac{\lambda i}{2n}\left(1-\frac{1}{\lambda}\right)^2+\frac{\lambda n}{i(n-1)} \right).
\end{align*}
  Recalling $\lambda=6 \log n$, we observe that for all large $n$, the term inside the exponential above is at most $-(2.9)\log n$ for any $i$, simply using $i\geq 2$. Consequently, the last display is $O\left(\frac{\sqrt{\log n}}{n^{0.9}}\right)$, finishing the proof.
\end{proof}

\subsection{Exchangeability of subtree functionals}

%The \textsc{urrt} enjoys a lot of distributional symmetries. We need a particularly important such property. Towards describing this, 
Consider $\cT^*$, the set of all isomorphism classes of rooted trees $(\mathbf{t},o)$,\footnote{In other words, we view a tuple of the form $(\mathbf{t},o)$ up to {isomorphism classes}, i.e.,  $(\mathbf{t}_1,o_1)$ and $(\mathbf{t}_2,o_2)$ are the \emph{same} if there is a bijection from the vertices of $\mathbf{t}_1$ to those of $\mathbf{t}_2$ with $o_1$ being mapped to $o_2$.} where $\mathbf{t}$ is a tree on a countable vertex set, and $o\in \mathbf{t}$ is a distinguished vertex that we call the \emph{root} of $\mathbf{t}$.

For any $r\geq 1$ and $i\in \{1,\dots,r\}$ consider the rooted tree $(\cT_n(i;E_r),i)$, where recall $E_r=\{\{u,v\}\in E(\cT_n):u,v\leq r\}\subseteq E(\cT_n)$. In words, $(\cT_n(i;E_r),i)$ is the descendant subtree of $i$ in $\cT_n$ after removing all edges between $1,\dots,r$, which we view as an element of $\cT^*$, rooted at $i$.

\begin{proposition}[Exchangeability of subtree functionals]\label{prop:exch_subtree_funcs}
    Fix $r\geq 1$. Consider a probability space $(\Omega,\cF,\mu)$, and random variables $X_1,\dots,X_k$ defined on this space with marginal law $\mu$ and such that the joint law of the random vector $(X_1,\dots,X_r)$ is \emph{exchangeable}. Let $(\kM,\cF_{\kM},m)$ be a measure space, and let $f:\Omega \times \cT^* \to 
    \kM$ be measurable under the product sigma algebra on $\cF\times \cF(\cT^*)$, where $\cF(\cT^*)$ is the discrete sigma algebra on $\cT^*$. Then, the $\kM^r$-valued random vector 
    \begin{align*}
    (f(X_1,(\cT_n(1;E_r),1)),\dots,f(X_k,(\cT_n(r;E_r),r)))    
    \end{align*}
    is an exchangeable vector.
\end{proposition}

To prove the proposition, we need a couple of intermediate lemmas. First, for any $n\geq 1$, let $\rho_n$ denote the distribution of the rooted random tree $(\cT_n,1)$ (seen as an element in $\cT^*$) where $\cT_n$ is our usual \textsc{urrt} with root $1$. The following lemma is a straightforward consequence of the construction of the \textsc{urrt} $\cT_n$, whose proof we omit.
\begin{lemma}[Conditionally independent subtrees]\label{lem:cond_indep_subtrees}
    Conditionally on the event \begin{align*}
\cG_r:=\bigcap_{i=1}^r\{|\cT_n(i;E_r)|=x_i\},        
    \end{align*}
    the conditional distribution of the random vector $((\cT_n(1;E_r),1),\dots,(\cT_n(r;E_r),r))$ taking values in $(\cT^*)^r$ follows law $\rho_{x_1}\times \dots \times \rho_{x_r}$. In other words, given $\cG_r$, the subtrees $(\cT_n(i; E_r), i)$ are conditionally independent \textsc{urrt}s of the respective correct sizes rooted at their respective roots.
\end{lemma}

We need one further result on subtree sizes.
\begin{lemma}[Exchangeability of subtree sizes]\label{lem:exch_subtree_sizes}
    The random vector $(|\cT_n(1;E_r)|,\dots,|\cT_n(r;E_r)|)$ is exchangeable.
\end{lemma}
\begin{proof}
    Note that the vector $(|\cT_n(1; E_r)|,\dots,|\cT_n(r; E_r)|)$ is distributed as the composition vector of a standard Pólya urn after $n-r$ steps, where initially it has $r$ balls of $r$ different colors. The result follows. 
\end{proof}
%We can now prove Proposition \ref{prop:exch_subtree_funcs}.

\begin{proof}[Proof of Proposition \ref{prop:exch_subtree_funcs}] For any $\underline{x}=(x_1,\dots,x_r)\in \Omega^r$ and any permutation $\tau$ of $[r]$, we denote $\underline{x}_\tau=(x_{\tau(1)},\dots,x_{\tau(r)})$. In particular, letting ${\rm id}:[r] \to [r]$ be the identity permutation, we use $\underline{x}$ and $\underline{x}_{\rm id}$ interchangeably to mean $(x_1,\dots,x_r)\in \Omega^r$.  

For permutations $\sigma,\tau$ of $[r]$, $\underline{x}=(x_1,\dots,x_r)\in\Omega^r$, and measurable sets $B_1,\dots, B_r\subseteq \kM$, denote 
\begin{align*}
\underline{f}_\sigma(\underline{x}_\tau)&=(f(x_{\tau(1)},(\cT_n(\sigma(1);E_r),\sigma(1))),\dots,f(x_{\tau(r)},(\cT_n(\sigma(r);E_r),\sigma(r)))), \\ \textrm{ and }
    \underline{B}_\sigma&=(B_{\sigma(1)},\dots,B_{\sigma(r)})~.
\end{align*}
We need to show that for a permutation $\sigma$ of $[r]$ and measurable sets $B_1,\dots,B_r \subseteq \kM$, 
\begin{align*}
    \Prob{\underline{f}_{\rm id}(\underline{X}_{\rm id})\in \underline{B}_\sigma}=\Prob{\underline{f}_{\sigma}(\underline{X}_{\sigma})\in \underline{B}_\sigma}.\numberthis \label{eq:exch_toshow}
\end{align*}
Using the independence of $\underline{X}_{\rm id}=(X_1,\dots,X_r)$ from $\cT_n$, letting $\mu_{\tau}$ be the law of $\underline{X}_{\tau}=(X_{\tau(1)},\dots,X_{\tau(r)})$, we have
\begin{align*}
    \Prob{\underline{f}_{\rm id}(\underline{X}_{\rm id})\in \underline{B}_\sigma}=\int_{\Omega^r}\Prob{\underline{f}_{\rm id}(\underline{x})\in \underline{B}_\sigma}\mu_{\rm id}(d\underline{x})~.
\end{align*}
Defining
    $$
    A_{\sigma(i)}:=\{t\geq 0:\exists\;(\mathbf{t},o)\in \cT^*\textrm{ with }|\mathbf{t}|=t\textrm{ and } f(x_{i},(\mathbf{t},o))\in B_{\sigma(i)}\}, 1 \le i \le r,
    $$
and $\underline{A}_\sigma=A_{\sigma(1)}\times \dots \times A_{\sigma(r)}$, 
we may write
    \begin{align*}
        \Prob{\underline{f}_{\rm id}(\underline{x})\in \underline{B}_\sigma}&=\sum_{\underline{t}_{\rm id}\in \underline{A}_\sigma}\CProb{\underline{f}_{\rm id}(\underline{x})\in \underline{B}_\sigma}{\underline{|\cdot|}_{\rm id}=\underline{t}_{\rm id}}\Prob{\underline{|\cdot|}_{\rm id}=\underline{t}_{\rm id}},
    \end{align*}
    where $t_1,\dots,t_r\geq 0$, and for a permutation $\sigma$,
    \begin{align*}
        &\underline{t}_\sigma=(t_{\sigma(1)},\dots,t_{\sigma(r)}), \textrm{ and } \underline{|\cdot |}_\sigma=(|\cT_n(\sigma(1);E_r)|,\dots,|\cT_n(\sigma(r);E_r)|).
    \end{align*}
    By Lemma \ref{lem:exch_subtree_sizes} and Lemma \ref{lem:cond_indep_subtrees}, we have, respectively, 
\begin{align*}
    \Prob{\underline{|\cdot|}_{\rm id}=\underline{t}_{\rm id}}=\Prob{\underline{|\cdot|}_{\sigma}=\underline{t}_{\rm id}}
\end{align*}
and
\begin{align*}
    \CProb{\underline{f}_{\tau}(\underline{x}_{\rm id})\in \underline{B}_\sigma}{\underline{|\cdot|}_{\tau}=\underline{t}_{\rm id}}=\prod_{i=1}^r\Prob{f(x_{i},\mathbf{T}(t_i))\in B_{\sigma(i)}},
\end{align*}
    where $\mathbf{T}(t_i)$ is distributed as $\rho_{t_i}$ the law of $(\cT_{t_i},1)$. We conclude that
    \begin{align*}
       \Prob{\underline{f}_{\rm id}(\underline{x})\in \underline{B}_\sigma}&=\sum_{\underline{t}_{\rm id}\in \underline{A}_\sigma}\CProb{\underline{f}_{\rm id}(\underline{x})\in \underline{B}_\sigma}{\underline{|\cdot|}_{\rm id}=\underline{t}_{\rm id}}\Prob{\underline{|\cdot|}_{\rm id}=\underline{t}_{\rm id}}\\&=\sum_{\underline{t}_{\rm id}\in \underline{A}_\sigma} \left(\prod_{i=1}^r\Prob{f(x_i,\mathbf{T}(t_i))\in B_{\sigma(i)}} \right)\Prob{\underline{|\cdot|}_{\sigma}=\underline{t}_{\rm id}}\\&=\sum_{\underline{t}_{\rm id}\in \underline{A}_\sigma}\CProb{\underline{f}_{\sigma }(\underline{x})\in \underline{B}_\sigma}{\underline{|\cdot|}_{\sigma}=\underline{t}_{\rm id}}\Prob{\underline{|\cdot|}_{\sigma}=\underline{t}_{\rm id}}\\&=\Prob{\underline{f}_{\sigma}(\underline{x})\in \underline{B}_\sigma}.
    \end{align*}
Note that
\begin{align*}
    \Prob{\underline{f}_{\rm id}(\underline{X}_{\rm id})\in \underline{B}_{\sigma}}&=\int_{\Omega^r}\Prob{\underline{f}_{\rm id}(\underline{x})\in \underline{B}_\sigma}\mu_{\rm id}(d\underline{x})\\&=\int_{\Omega^r}\Prob{\underline{f}_{\sigma}(\underline{x})\in \underline{B}_\sigma}\mu_{\rm id}(d\underline{x})\\&=\int_{\Omega^r}\Prob{\underline{f}_{\sigma}(\underline{x})\in \underline{B}_\sigma}\mu_{\sigma}(d\underline{x})=\Prob{\underline{f}_{\sigma}(\underline{X}_{\sigma})\in \underline{B}_{\sigma}},
\end{align*}
finishing the proof, where we used the fact that $\mu_{\rm id}=\mu_{\sigma}$ by the exchangeability of $\underline{X}$.
\end{proof}

\begin{corollary}[Negative correlation of real functionals]\label{cor:neg_corr}
%Under the setting of Proposition \ref{prop:exch_subtree_funcs}, take the co-domain $\kM$ of $f$ to be $\R$, endowed with its natural Borel sigma-algebra, 
Fix $r\geq 2$, and consider the random variables
    \begin{align*}
        Y_i:=f(X_i,(\cT_n(i;E_r),i)),\quad Z:=\sum_{i=1}^r Y_i~.
    \end{align*}
    Then for all $i \neq j$,
    \begin{align*}
        \CExp{(Y_i-\CExp{Y_i}{Z})(Y_j-\CExp{Y_j}{Z})}{Z}\leq 0~.
    \end{align*}
\end{corollary}
\begin{proof}
    This follows from a similar argument as in Aldous \cite[Eq.\ (1.7)]{aldous2006exchangeability}. Note that by Proposition \ref{prop:exch_subtree_funcs}, $(Y_1,\dots,Y_r)$ is exchangeable, and thus so is $(Y_1,\dots,Y_r)$ given $Z$, as the total sum $Z$ is permutation-invariant. Further, given $Z$, the random variable $\sum_{i=1}^r (Y_i-\CExp{Y_i}{Z})=0$ a.s., so that
    \begin{align*}
        0&=\CExp{\left(\sum_{j=1}^r(Y_i-\CExp{Y_i}{Z})\right)^2}{Z}\\&=r\CExp{(Y_1-\CExp{Y_1}{Z})^2}{Z}+r(r-1)\CExp{(Y_i-\CExp{Y_i}{Z})(Y_j-\CExp{Y_j}{Z})}{Z}.
    \end{align*}
    Thus, $$\CExp{(Y_
i-\CExp{Y_i}{Z})(Y_j-\CExp{Y_j}{Z})}{Z}
=
\frac{-\CExp{(Y_1-\CExp{Y_1}{Z})^2}{Z}}{r-1}
\leq 0~.$$
\end{proof}
We need one more consequence for later use. For any measurable space $(S,\cF_S)$, let $\cP(S)$ denote the space of all probability measures on $S$. 
%We generally equip this space with the Borel sigma algebra coming from the topology of weak convergence. 

\begin{proposition}[Sample exchangeability from exchangeable measures]\label{prop:sample_exch}
    Let $\mu:\cT^* \to \cP(\cT^*)$ be measurable. Given $(\bt,o)\in\cT^*$, we denote the image of it under $\mu$ as $\mu_{(\bt,o)}$. Fix $r\geq 1$. Conditionally on $((\cT_n(1;E_r),1),\dots,(\cT_n(r;E_r),r))$, let $(\sT_1,\dots,\sT_r)$ be a vector of random trees with law
    \begin{align*}
        \mu_{(\cT_n(1;E_r),1)}\times \dots \times \mu_{(\cT_n(r;E_r),r)}, 
    \end{align*}
    i.e., the trees $\sT_i$ form a conditionally independent collection, given $((\cT_n(1;E_r),1),\dots,(\cT_n(r;E_r),r))$, where $\sT_i$ has marginal law $\mu_{(\cT_n(i;E_r),i)}$ for each $i\in [r]$. Then, the vector $(\sT_1,\dots,\sT_r)$ is exchangeable.
\end{proposition}
\begin{proof}
For any permutation $\tau$ of $[r]$, denote $\sT_\tau=(\sT_{\tau(1)},\dots,\sT_{\tau(r)})$.
    Note that for any $B\subseteq (\cT^*)^r$,
    \begin{align*}
       &\Prob{\sT_{\tau}\in B}= \int_{B}\ind{((\bt_1,o_1),\dots,(\bt_2,o_2))\in B} \prod_{i=1}^rd\mu_{(T_i,u_i)}(\bt_i,o_i) \cdot d\pi_\tau((T_1,u_1),\dots,(T_r,u_r)),\numberthis\label{eq:sample_exch_1}
    \end{align*}
    where $\pi_\tau$ denotes the law of the vector
    \begin{align}
        ((\cT_n(\tau(1);E_r),\tau(1)),\dots,(\cT_n(\tau(r);E_r),\tau(r))).
    \end{align}
    In Proposition \ref{prop:exch_subtree_funcs}, setting $\kM=\cT^*$ and $f:\Omega\times \cT^* \to \cT^*$,   that is, $f(x,(\bt,o))=(\bt,o)$
    (the projection map on the second coordinate),
    we obtain that $\mu_\tau=\mu_\sigma$ for any  permutation $\sigma$ of $[r]$. Using this in \eqref{eq:sample_exch_1}, we observe
    \begin{align*}
        \Prob{\sT_{\tau}\in B}= \int_{B}\ind{((\bt_1,o_1),\dots,(\bt_2,o_2))\in B} \prod_{i=1}^rd\mu_{(T_i,u_i)}(\bt_i,o_i) \cdot d\pi_\sigma((T_1,u_1),\dots,(T_r,u_r))=\Prob{\sT_{\sigma}\in B}~.
    \end{align*}
\end{proof}

\section{Root finding in high-degree forests}\label{sec:root_find_highdeg}
In this section, we prove Theorem \ref{thm:forest}.

\subsection{Reduction to offspring cutting}
The first step of the proof of Theorem \ref{thm:forest} is to show that the structure of the $\alpha$-forest does not change too much if we retain vertices with a large number of \textit{offspring}, instead of retaining those with a large degree, as in the definition of the construction of the $\alpha$-forest (recall Definition \ref{def:alpha_high_deg_for}).
Recall that in any rooted tree $(T,o)$, $c_T(u)$ denotes the number of \emph{offspring} (or \emph{children}) of vertex $u$ in $T$.
%i.e., the number of neighbors of $u$ in $T$ that have graph-distance larger (by $1$) than that of $u$ from the root $o$.
%\begin{definition}[$\alpha$-offspring forests]\label{def:alpha_off_forests}
    %For the \textsc{urrt} $\cT_n$ on $[n]$ with root vertex $1$, the subforest of $\cT_n$ spanned by the vertices $v \in [n]$ with $c_{\cT_n}(v)>(1-\alpha)\log n$ is called the `$\alpha$-offspring forest' of $\cT_n$, and is denoted by $\mathrm{F}_n(\alpha)$. 
%\end{definition}

%\red{
%[
\begin{definition}[$\alpha$-offspring forests]
  The subforest of $\cT_n$ constructed by removing the edges with at least one end vertex satisfying $c_{\cT_n}(v)\leq (1-\alpha)\log n$ is called the `$\alpha$-offspring forest' of $\cT_n$, and is denoted by $\mathrm{F}_n(\alpha)$. 
\end{definition}
%]

%}

\begin{remark}
    As in Remark \ref{rem:site_perco_forest}, it is useful to view components that are not isolated vertices in $\rF_n(\alpha)$ as components in the forest $\rF_n^{(\rm V)}(\alpha)$
    spanned by vertices with more than $(1-\alpha)\log n$ offspring.
\end{remark}

Our first step is to claim that $\alpha$-forests are well approximated by $\alpha$-offspring forests. 

\begin{lemma}[Approximating by offspring forests]\label{lem:approx_offspring_forests}
    Fix $\alpha \in (0,1)$ and $\delta >0$. Then a.s., for all large $n$,  $\cF_n(\alpha-\delta)\subseteq \mathrm{F}_n(\alpha)\subseteq \cF_n(\alpha)$.
\end{lemma}
\begin{proof}
    The lemma follows by  noting that $c_{\cT_n}(v)=d_{\cT_n}(v)-1$ if $v\neq 1$ and $c_{\cT_n}(1)=d_{\cT_n}(1)$.
\end{proof}
%The last result helps translate the results for $\mathrm{F}_n(\alpha)$ to $\cF_n(\alpha)$. %We take this route because working with the former is less challenging than the latter.

\subsection{Operations outputting the root component}

Recall the notion of an $\alpha$-forest $\cF_n(\alpha)$ from Definition \ref{def:alpha_high_deg_for}. For any vertex $i \in \cT_n$, let us denote by  $\mathbf{C}_\alpha(i)$ and $C_\alpha(i)$ the connected components of $i$ in $\cF_n(\alpha)$ and $\mathrm{F}_n(\alpha)$, respectively. 
Observe that
\begin{align*}
\mathbf{C}_\alpha(i)=\cT_n(i;E^{\rm deg}_\alpha)\textrm{ and }{C}_\alpha(i)=\cT_n(i;E^{\rm off}_\alpha),
\end{align*}
where we define
\begin{align*}
    &E^{\rm deg}_\alpha:=\{\{u,v\}\in E(\cT_n):d_{\cT_n}(u)\wedge d_{\cT_n}(v)\leq (1-\alpha)\log n\},\\&E^{\rm off}_\alpha:=\{\{u,v\}\in E(\cT_n):c_{\cT_n}(u)\wedge c_{\cT_n}(v)\leq (1-\alpha)\log n\}.
\end{align*}

\begin{remark}[Sandwiching of root components]\label{rem:sandwiching_root_comps}
    By Lemma \ref{lem:approx_offspring_forests}, for any $\delta>0$, for all $n$ large enough, $\bC_{\alpha-\delta}(1)\subseteq C_\alpha(1)\subseteq \bC_{\alpha}(1)$ with probability $1$. 
\end{remark}

%where by convention, $\mathbf{C}_\alpha(i)=\varnothing$ if $i \notin \cF_n(\alpha)$ and $C_\alpha(i)=\varnothing$ if $i\notin \rF_n(\alpha)$. Recall that we view $\cT_n$ 
We view both the graphs $\mathbf{C}_\alpha(1)$ and $C_\alpha(1)$ as being rooted at $1$. 
It is useful to view the process of obtaining $\mathbf{C}_\alpha(1)$ and $C_\alpha(1)$ from $\cT_n$ as \emph{general operations}. Recall that $\cT^*$ denotes the space of (rooted isomorphism classes of) all rooted trees $(\bt,o)$. 
%For this, consider $\cT^*$, the set of all rooted trees $(\mathbf{t},o)$ (where we view these objects up to isomorphism classes, i.e., to us $(\mathbf{t}_1,o_1)$ and $(\mathbf{t}_2,o_2)$ are the same if there is a bijection from the vertices of $\mathbf{t}_1$ to those of $\mathbf{t}_2$ with $o_1$ being mapped to $o_2$), where $\mathbf{t}$ is a tree on a countable vertex set, and $o\in \mathbf{t}$ is some distinguished vertex that we call the \emph{root} of $\mathbf{t}$. 
%Let us now define some useful operators on this space.

\begin{definition}[The operator $\psi^r_{\rm deg}$]\label{def:psi_deg_op}
   For any $r\geq 1$, define the operator $\psi^r_{\rm deg}:\cT^*\to \cT^*$ as follows. For any $(\mathbf{t},o)\in \cT^*$, $(\psi^r_{\rm deg}(\mathbf{t},o),o)\in \cT^*$ is a rooted tree with root $o$, where $\psi^r_{\rm deg}(\mathbf{t},o)$ is the connected component of $o$ when all edges in $\mathbf{t}$ that have at least one end-vertex with \emph{degree} at most $r$ has been removed.
\end{definition}

\begin{definition}[The operator $\psi^r_{\rm off}$]\label{def:psi_off_op}
   For any $r\geq 1$, define the operator $\psi^r_{\rm off}:\cT^*\to \cT^*$ as follows. For any $(\mathbf{t},o)\in \cT^*$, $(\psi^r_{\rm deg}(\mathbf{t},o),o)\in \cT^*$ is a rooted tree, where $\psi^r_{\rm deg}(\mathbf{t},o)$ is the connected component of $o$ when all edges in $\mathbf{t}$ that have at least one end-vertex with \emph{number of offspring} at most $r$ are removed.
\end{definition}

    In particular, note that $\psi^{(1-\alpha)\log n}_{\rm deg}(\cT_n,1)=(\bC_\alpha(1),1)$ and $\psi^{(1-\alpha)\log n}_{\rm off}(\cT_n,1)=(C_\alpha(1),1)$.

\begin{remark}
\label{rem:op_monotonicity}
    The following monotonicity properties follow from the definitions.
    \begin{itemize}
        \item [(i)] For $(\mathbf{t},o),(\mathbf{t}',o)\in \cT^*$, if $\mathbf{t}'$ is a subtree of $\mathbf{t}$, then $\psi^r_{\rm deg}(\mathbf{t},o)\subseteq \psi^r_{\rm deg}(\mathbf{t}',o)$ and $\psi^r_{\rm off}(\mathbf{t},o)\subseteq \psi^r_{\rm off}(\mathbf{t}',o)$ as subtrees.
        \item [(ii)] For $(\mathbf{t},o)\in \cT^*$, and any $r\geq 1$, one has $\psi^r_{\rm off}(\mathbf{t},o)\subseteq \psi^r_{\rm deg}(\mathbf{t},o)$ and $\psi^r_{\rm deg}(\bt,o)\subseteq \psi^{r-1}_{\rm off}(\bt,o)$.
        \item [(iii)] For $(\mathbf{t},o)\in \cT^*$ and any $1\leq r\leq M$ $\psi^M_{\rm off}(\mathbf{t},o)\subseteq \psi^r_{\rm off}(\mathbf{t},o)$ and $\psi^M_{\rm deg}(\mathbf{t},o)\subseteq \psi^r_{\rm deg}(\mathbf{t},o)$.
    \end{itemize} 
\end{remark}

%The point of these operators is that, if we view $(\cT_n,1)$ as an element from $\cT^*$, then $\psi^{(1-\alpha)\log n}_{ \rm deg}(\cT_n,1)=(\mathbf{C}_\alpha(1),1)$ and $\psi^{(1-\alpha)\log n}_{ \rm off}(\cT_n,1)=({C}_\alpha(1),1)$. %\red{[ good to make an operator where the thresholding on the offspring of the root is one less than the thresholding for other offspring. (reply: this is just the degree op) ]}%\begin{lemma}[Truncating branching processes]\label{lem:trunc_BP}
    %Consider a BGW tree $T$, where the number of offspring is given by a random variable $M$. Viewing $(T,o)$ as an element of $\cT^*$, where $o$ is the root of $T$, let $\psi^r_{\rm off}(T,o)=(T',o)$. Then $T'$ is a BGW tree with root $o$ and with the number of offspring given by the random variable $M\ind{M> r}$.
%\end{lemma}
%\begin{proof}
%Consider exploring $\psi^r_{\rm off}$ starting from the root in a breadth-first manner. By the definition of $\psi^r_{\rm off}$, there is an edge in this tree between a parent and a child in $(T,o)$ if and only if both of them have at least $r$ children. Thus, the number of children of any vertex in $\psi^r_{\rm off}(T,o)$ is $0$ if it does not any child in $(T,o)$ $c_u$ Thus, for any edge $\{u,v\}$ in $\psi^r_{\rm off}(T,o)$, its end-vertices have number of children $c_u\ind{c_u>r}$ and $c_v\ind{c_v>r}$. In particular, these numbers are independent since $(T,o)$ is a BGW tree. The result follows.
%\end{proof}
%Let us also define another related useful operator. 
\begin{definition}[The operator $\chi_{>r}$]
    The operator $\chi_{>r}:\cT^*\to \cT^*$ is defined as follows. For any rooted tree $(\mathbf{t},o)\in \cT^*$, let $\chi_{>r}(\mathbf{t},o)$ be the connected component of $o$ in the subforest of $\mathbf{t}$ constructed by retaining only those edges that are incident to a vertex with more than $r$ offspring. We view $(\chi_{>r}(\mathbf{t},o),o)$ as a tree rooted at $o$.  
\end{definition}
\begin{figure}[H]
    \centering
    \resizebox{0.6\textwidth}{!}{
        \input{operators_psi_and_chi}}
    \caption{A rooted tree $(\bt,o)$ with the operators $\chi_{>5}$ and $\psi^5_{\rm off}$ applied to it. The red vertices are the vertices other than the root that have at least $5$ offspring. Note that the edges of $\psi^{5}_{\rm off}(\bt,o)$ form a subset of the edges of $\chi_{>5}(\bt,o)$.}
    \label{fig:operators chi and psi}
\end{figure}
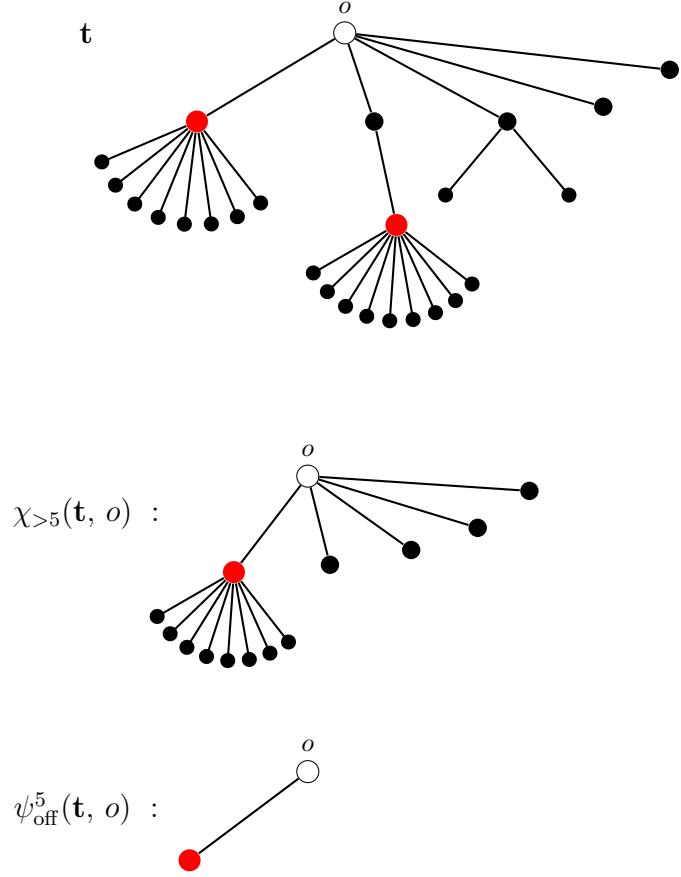

\begin{remark}\label{rem:op_inclusion}
    The edge set of $\psi^r_{\rm off}(\mathbf{t},o)$ is contained in the edge set of $\chi_{>r}(\mathbf{t},o)$, see Figure \ref{fig:operators chi and psi}.
\end{remark}
The next simple observation describes the effect of $\chi_{>r}$ on a Bienaym\' e--Galton--Watson tree.
\begin{lemma}\label{lem:trunc_BP}
    Let $(T,o)$ be a Bienaym\' e--Galton--Watson tree with the number of offspring given by a random variable $X$. Then $(\chi_{>r}(T,o),o)$ is a Bienaym\' e--Galton--Watson subtree of $T$, with number of offspring given by the random variable $X\ind{X>r}$.  
\end{lemma}
%\begin{proof}
%    First note that for any two vertices $u,v\in \chi_{>r}(T,o)$, their number of offspring in it are $c_u\ind{c_u>r}$ and $c_v\ind{c_v>r}$, where $c_u$ and $c_v$ are repsectively their number of offspring in $T$. Thus, the numbers of offspring of different vertices in $\chi_{>r}(T,o)$ are independent. In particular, $\chi_{>r}(T,o)$ is a tree, since it is by definition connected, with root $o$, and where each vertex has an independent number of offspring which is a copy of $X\ind{X>r}$. The result follows.
%\end{proof} 

\subsection{Upper bound of the size of the root component}
\label{sec:c1_UB}

In this section, we prove an upper bound on the size of the root component $|\mathbf{C}_\alpha(1)|$ in $\cF_n(\alpha)$.

\begin{proposition}
\label{prop:root_comp_UB}
For any $\alpha'>\alpha$, $\lim_{n \to \infty}\Prob{|\mathbf{C}_\alpha(1)|>n^{\alpha'}}=0$.
\end{proposition}

We need a few intermediate results. For $i\in [n]$, let $B_i \stackrel{d}{=}\mathrm{Ber}(1/i)$, be independent  random variables, and define
\begin{align*}
    X_i:=\sum_{j=i}^n B_j~. \numberthis \label{eq:Xi}
\end{align*}
$X_i$ is distributed as the number of offspring $c_{\cT_n}(i)$ of the vertex $i$ in $\cT_n$. Define, for any $v \in [n]$, the \emph{truncated} random variable
\begin{align*}
    \xi_v:=X_v\ind{X_v>(1-\alpha)\log n}. \numberthis \label{eq:truncated_degree}
\end{align*}
Recalling the $\alpha$-offspring forest $\rF_n(\alpha)$ and $E_v=\{\{u_1,u_2\}\in E(\cT_n):u_1,u_2\in[v]\}$, we have the following stochastic domination.

\begin{lemma}\label{lem:Bienaym\' e--Galton--Watson_dom_subtreeforest}
For any $v \in [n]$, 
$|\rF_n(\alpha)(v;E_v)|\preceq |\mathscr{T}(v,\alpha)|$,
where $\mathscr{T}(v,\alpha)$ is a Bienaym\' e--Galton--Watson tree with offspring distribution given by $\xi_v$.
\end{lemma}

%\red{[ N: in fact, we construct a coupling, such that $|\rF_n(\alpha)(v;E_v)|\leq  |\mathscr{T}(v,\alpha)|$ is true ]}

\begin{proof}
We show that the subtree $\cT_n(v;E_v)$ in $\cT_n$ rooted at $v$ can be coupled with a Bienaym\' e--Galton--Watson tree $\mathscr{T}(v)$ with offspring distribution $X_v$ as in \eqref{eq:Xi}, such that $\cT_n(v;E_v)$ is contained as a subset in $\mathscr{T}(v)$, with $v$ mapped to the root of $\mathscr{T}(v)$. As a result, observing that $\rF_n(\alpha)(v;E_v)=\psi^{(1 -\alpha) \log n}_{\rm off}(\cT_n(v;E_v),v)$ we have $\rF_n(\alpha)(v;E_v)\subseteq \psi^{(1 -\alpha) \log n}_{\rm off}(\mathscr{T}(v),v)$ by Remark \ref{rem:op_monotonicity}, where we recall $\psi^r_{\rm off}$ from Definition \ref{def:psi_off_op}. By Remark \ref{rem:op_inclusion}, we conclude $|\rF_n(\alpha)(v;E_v)|\leq |\chi_{>(1-\alpha)\log n}(\mathscr{T}(v),v)|$. Finally, by Lemma \ref{lem:trunc_BP}, the latter is the total size of a Bienaym\' e--Galton--Watson tree with number of offspring given by \eqref{eq:truncated_degree}, finishing the argument.  %by our coupling, it is contained in the connected component containing the root in the forest obtained by keeping those vertices of $\mathscr{T}(v)$ with at least $(1-\alpha) \log n$ offspring. However, note that for any Bienaym\' e--Galton--Watson tree $\mathscr{T}$ with offspring distribution $X$ and any $a>0$, the root component of it in the forest obtained by keeping vertices with at least $a$ many offspring, is itself a Bienaym\' e--Galton--Watson tree, with offspring distribution $X\ind{X>a}$. \red{[ issue: not itself, but is contained in ]} The lemma thus follows by taking $\mathscr{T}=\mathscr{T}(v)$, $X=X_v$ and $a=(1-\alpha)\log n$.

%\red{[ N: Tie this paragraph to the lemma \ref{lem:trunc_BP} before ]}

\noindent \textbf{The coupling.} Thus, it suffices to construct a coupling between $\cT_n(v;E_v)$ and $\mathscr{T}(v)$ such that
\begin{align*}
    \cT_n(v;E_v)\subseteq \mathscr{T}(v). \numberthis \label{eq:coup_stoch_incl}
\end{align*}
To do this, starting from $v$, we explore the offspring sets {in $\cT_n(v; E_v)$}, of the vertices in $\cT_n(v; E_v)$, recursively, by the increasing order of their labels, and at each step, we dominate the offspring set by the corresponding offspring set in $\mathscr{T}(v)$. Define, for each $i \in \{2,\dots,n\}$, independent random variables  $\cY_i\stackrel{d}{=} \mathrm{U}[i-1]$. Observe that these variables determine the \textsc{urrt} $\cT_n$ by letting $j>i$ connect to $i$ if $\cY_j=i$ for any $1\leq i<j\leq n$. In particular, the subtree $\cT_n(v;E_v)$ is determined by the variables $\cY_{v+1},\dots,\cY_n$.

To determine the offspring set of $v \in \cT_n(v;E_v)$, we reveal the random variables $\ind{\cY_p=v}$ for $p\geq v+1$. Let $\varnothing$ denote the root of $\mathscr{T}(v)$. Declare $\mathscr{C}(v):=\{u\geq v+1:\ind{\cY_u=v}=1\}$ to be the set of offspring of $\varnothing$ in $\mathscr{T}(v)$. Thus, the set of offspring of $v$ in $\cT_n(v;E_v)$ and $\varnothing$ in $\mathscr{T}(v)$ are exactly the same. Note also that the number of offspring of $v$ both in $\cT_n(v;E_v)$ and $\mathscr{T}(v)$ is distributed as $X_v$, as defined in \eqref{eq:Xi}.

Define $j_1:=\inf\{u\geq v+1: \cY_u=v \}$. We reveal the offspring of $j_1$ and couple it to its corresponding offspring set in $\mathscr{T}(v)$. The offspring set of $j_1$ in $\cT_n$ is determined by the variables $\{\cY_q:q\geq j_1+1,q\notin \mathscr{C}(v)\}$. Conditionally on the offspring set $\mathscr{C}(v):=\{u \geq v+1:\ind{\cY_u=v}=1\}$, the variables $\{\cY_q:q\geq j_1+1,q\notin \mathscr{C}(v)\}$ are independent, with $\cY_q\stackrel{d}{=} \mathrm{U}([q]\setminus \{v\})$. In particular, the offspring set of $j_1$ in $\cT_n$ is $\mathscr{C}(j_1):=\{q\geq j_1+1,q \notin \mathscr{C}(v):\ind{\cY_q=j_1}=1\}$. Observe that conditionally on $\mathscr{C}(v)$, for any $j_1+1\leq q \notin \mathscr{C}(v)$, $\Prob{\cY_q=j_1}=1/(q-1)$. Thus, the number of offspring of $j_1$ in $\cT_n$, conditionally on $\mathscr{C}(v)$, is a sum
\begin{align*}
    \sum_{j_1+1\leq q \notin \mathscr{C}(v)}\mathrm{Ber}\left(\frac{1}{q-1}\right), \numberthis \label{eq:offspring_j1}
\end{align*}
where the summands are independent. To construct the offspring set of $j_1$ in $\mathscr{T}(v)$, we first let it have all its \textsc{urrt} offspring, i.e., $\{q\geq j_1+1,q \notin \mathscr{C}(v):\ind{\cY_q=j_1}=1\}$. To make sure $j_1$ has offspring distribution $X_v$ in $\mathscr{T}(v)$, we let it have some `pseudo children' in $\mathscr{T}(v)$, which we explain next.

Observe that for any $j_1+1\leq q \notin \mathscr{C}(v)$, since $j_1\geq v+1$, we have $q-1\geq v+1$. For $A\subseteq [n]$ and $y \in [n]$, define $A-y:=\{a-y:a\in A\}$. For any $$r \in \Phi(j_1):=\{v+1,\dots,n\}\setminus \left((\{j_1+1,\dots,n\}\cap \mathscr{C}(v)^c)-1\right),$$ 
independently of everything else, with probability $1/r$, let $j_1$ give birth to a pseudo-child $z(j_1,r)$ in $\mathscr{T}(v)$, and these births are independent for different choices of $r \in \Phi(j_1)$. We call $\Phi(j_1)$ the set of \emph{potential pseudo-children} of $j_1$. For intuition on this set, note that by \eqref{eq:offspring_j1}, corresponding to any $x\in (\{j_1+1,\dots,n\}\cap \mathscr{C}(v)^c)-1$, $j_1$ already has a child with probability $1/x$ in $\cT_n$. However, we want it to have $X_v$ children in $\mathscr{T}(v)$. In the sum \eqref{eq:Xi} that represents $X_i$, the terms corresponding to $\{x\in (\{j_1+1,\dots,n\}\cap \mathscr{C}(v)^c)-1\}$ have already been accounted for by \eqref{eq:offspring_j1}. Letting independent $B^{j_1}_x\stackrel{d}{=} {\rm Ber}(1/x)$ for $x\in [n]$, also independent of everything else, observe that $\sum_{x\in \Phi(j_1)}B^{j_1}_x$ is the rest of the sum, so that from \eqref{eq:Xi},
\begin{align*}
\sum_{j_1+1\leq q \notin \mathscr{C}(v)}\mathrm{Ber}\left(\frac{1}{q-1}\right)
+\sum_{x\in \Phi(j_1)}B^{j_1}_x\stackrel{d}{=}X_v.
\end{align*}
In particular, by the above observation, combining the \textsc{urrt} offspring $\mathscr{C}(j_1)$ with the set of its pseudo-children, the total number of children of $j_1$ in $\mathscr{T}(v)$ is $X_v$-distributed and independent of the number of offspring of $v$, by construction. Furthermore, denoting the set of offspring of $j_1$ in $\mathscr{T}(v)$ by $\mathscr{C}(\mathscr{T}(v),j_1)$ we have the inclusion $\mathscr{C}(j_1)\subseteq \mathscr{C}(\mathscr{T}(v),j_1)$.

The general step follows recursively. Formally, let the increasing ordering of the labels of the vertices in $\cT_n(v; E_v)$ be $v=v_1<v_2<\dots<v_\ell$, where $\ell=|\cT_n(v; E_v)|$, and say we have constructed the offspring sets $\mathscr{C}(v_1),\dots,\mathscr{C}(v_k)$ and $\mathscr{C}(\mathscr{T}(v),v_1),\dots,\mathscr{C}(\mathscr{T}(v),v_k)$ with $\mathscr{C}(v_i)\subseteq \mathscr{C}(\mathscr{T}(v),v_i)$ for all $1\leq i \leq k$. Next, we construct the corresponding offspring sets $\mathscr{C}(v_{k+1})$ and $\mathscr{C}(\mathscr{T}(v),v_{k+1})$ of $v_{k+1}$, such that $\mathscr{C}(v_{k+1})\subseteq \mathscr{C}(\mathscr{T}(v),v_{k+1})$.

We first construct the offspring set of $v_{k+1}$ in $\cT_n(v;E_v)$, $$\mathscr{C}(v_{k+1}):=\{u\geq v_{k+1}+1,u \notin (\mathscr{C}(v_1)\cup \dots \cup \mathscr{C}(v_k)): Y_u=v_{k+1} \},$$ 
and observe that conditionally on the set $\mathscr{C}(v_1)\cup \dots \cup \mathscr{C}(v_k)$, for any $v_{k+1}\leq u \notin \mathscr{C}(v_1)\cup \dots \cup \mathscr{C}(v_k)$, $\Prob{Y_u=v_{k+1}}=1/(v_{k+1}-k)$. At this point, as before, we consider a set of potential pseudo-children of $v_{k+1}$, $$\Phi(v_{k+1}):=\{v+1,\dots,n\}\setminus \left((\{v_{k+1}+1,\dots,n\}\cap (\mathscr{C}(v_1)\cup \dots \cup \mathscr{C}(v_k))^c)-k\right).$$

As before, $\mathscr{C}(v_{k+1})$ already accounts for the set of offspring that $v_{k+1}$ begets in $\mathscr{T}(v)$ corresponding to each $x\in (\{v_{k+1}+1,\dots,n\}\cap (\mathscr{C}(v_1)\cup \dots \cup \mathscr{C}(v_k))^c)-k$ with probability $1/x$. The set $\Phi(v_{k+1})$ thus accounts for the remainder of the values $y$, such that when $v_{k+1}$ receives a child corresponding to each $y \in \Phi(v_{k+1})$ with probability $1/y$, the total number of children of $v_{k+1}$ in $\mathscr{T}(v)$ becomes equal in distribution to $X_v$ (as in \eqref{eq:Xi} with $i=v$).

%\textcolor{red}{LUC: I changed the notation PC by $\Phi$, because readers may think that PC is P times C.}
Thus, for any $r\in \Phi(v_{k+1})$, independently of everything else, we let $v_{k+1}$ give birth to a pseudo-child in $\mathscr{T}(v)$ with probability $1/r$, where the different births are independent, and note that $\mathscr{C}(\mathscr{T}(v),v_{k+1})$, the set of offspring (including the pseudo ones) of $v_{k+1}$ in $\mathscr{T}(v)$ contain the set of offspring $\mathscr{C}(v_{k+1})$ of $v_{k+1}$ in $\cT_n(v; E_v)$. 

We continue this procedure until we have constructed the sets $\mathscr{C}(v_i), \mathscr{C}(\mathscr{T}(v),v_i)$ for all $1\leq i \leq \ell$, where  $\ell=|\cT_n(v;E_v)|$. The union $\bigcup_{i=1}^\ell \mathscr{C}(\mathscr{T}(v),v_i)$ forms a tree, which is not the entire tree $\mathscr{T}(v)$, but only part of it, and it contains $\cT_n(v;E_v)=\bigcup_{i=1}^\ell\mathscr{C}(v_i)$ as a subtree. To obtain the entire tree $\mathscr{T}(v)$, we  drop independent Bienaym\' e--Galton--Watson trees $\xi_s$, each with offspring distribution $X_v$, for each leaf $s$ in the tree $\bigcup_{i=1}^\ell \mathscr{C}(\mathscr{T}(v),v_i)$. The resulting object is distributed as a Bienaym\' e--Galton--Watson tree $\mathscr{T}(v)$, and contains $\cT_n(v; E_v)$ as a subtree. This concludes the proof.  
\end{proof}

\begin{remark}[Stochastic ordering of Bienaym\' e--Galton--Watson sizes]\label{rem:stoch_dom_Bienaym\' e--Galton--Watson}
    Observe that $|\mathscr{T}(v_1,\alpha)|\preceq |\mathscr{T}(v_2,\alpha)|$ if $v_2\leq v_1$.
\end{remark}
Next, we prove a few properties of the offspring distribution $\xi_v$ as in \eqref{eq:truncated_degree} and the size of the tree $\mathscr{T}(v,\alpha)$. We begin with a general lemma on sums of independent Bernoulli random variables.

\begin{lemma}\label{lem:offspring_moments}
Let $X=\sum_{i=1}^\infty \mathrm{Ber}(p_i)$ where the summands are independent, $\mu:=\Exp{X}=\sum_{i=1}^\infty ip_i<\infty$, $k \geq 1$ and $x>\mu+k$. Then %there exists $C=C(k)>0$ with $C(1)=C(2)=1$ such that
\begin{align*}
    \Exp{X^k\ind{X\geq x}}\leq %C(k)\left(\sum_{j=0}^{k-1}\mu^{-j}\right) 
    (x-k)^k \exp\left(x-k-\mu-x\log\left(\frac{x-k}{\mu}\right)\right).
\end{align*}
\end{lemma}
\begin{proof} 
For any $k\geq 1$ and $\lambda,x>0$,
\begin{align*}
    e^{-\lambda x}\Exp{X^k e^{\lambda X}}
    &=e^{-\lambda x}\Exp{\left(\sum_{i \geq 1}X_i\right)^k e^{\lambda X}}\\&=e^{-\lambda x}\Exp{\sum_{i_1,\dots,i_k}\prod_{\ell\in \{1,\dots,k\}}(X_{i_\ell}e^{\lambda X_{i_\ell}})\prod_{m\notin\{i_1,\dots,i_k\}}e^{\lambda X_m}}\\
    & =e^{-\lambda x}\sum_{i_1,\dots,i_k}\prod_{\ell \in \{1,\dots,k\}}(p_{i_\ell}e^{\lambda })\prod_{m\notin\{i_1,\dots,i_j\}}(1-p_m+p_me^{\lambda })\\
    &\leq e^{-\lambda x}e^{k\lambda}\prod_{m}(1+p_m(e^\lambda-1))\sum_{i_1,\dots,i_k} \prod_{1\leq \ell \leq k}p_{i_\ell}\\
    & \leq \mu^k e^{-\lambda x} e^{k\lambda}e^{\mu(e^\lambda-1)}.
\end{align*}
%for $C=C(k)$ satisfying $C(1)=C(2)=1$. 
The exponent is minimized for $e^{\lambda}=\frac{x-k}{\mu}$. Thus, 
\begin{align*}
    \Exp{X^k\ind{X\geq x}}\leq e^{-\lambda x}\Exp{X^k e^{\lambda X}}\leq %\left(\sum_{j=0}^{k-1}\mu^{-j}\right) 
    (x-k)^k \exp\left(x-k-\mu-x\log\left(\frac{x-k}{\mu}\right)\right).
\end{align*} \end{proof}

\begin{remark}[Bienaym\' e--Galton--Watson offspring moments]\label{rem:Bienaym\' e--Galton--Watson offspring moments}
Recall the Bienaym\' e--Galton--Watson tree $\mathscr{T}(v,\alpha)$ from the statement of Lemma \ref{lem:Bienaym\' e--Galton--Watson_dom_subtreeforest}. For any $\alpha_1\in(\alpha,1)$ 
%, $\alpha_2\in (\alpha,\alpha_1)$, 
and a vertex $v\geq n^{\alpha_1}>n^{\alpha}$, 
%recalling the distribution $\xi_v$ from \eqref{eq:truncated_degree}, 
note that 
%\footnote{Here we are viewing both $n^{\alpha_1}$ and $n^{\alpha_2}$ as vertices of $\cT_n$.}
\begin{align*}
    \xi_{\floor{n^{\alpha_1}}}\succeq \xi_v \implies \Exp{\xi_{\floor{n^{\alpha_1}}}^k}\geq \Exp{\xi_v^k}.
\end{align*}
By a standard Riemann approximation, for any $\eps>0$, if $n$ is large enough,
    \begin{align*}
\Exp{X_{\floor{n^{\alpha_1}}}}=\sum_{j=\floor{n^{\alpha_1}}+1}^n\frac{1}{j-1}\in [(1-\alpha_1-\eps)\log n,(1-\alpha_1+\eps)\log n]
    \end{align*}
    
    Thus, applying Lemma \ref{lem:offspring_moments}, for any vertex $v\geq n^{\alpha_1}$, by choosing $\eps>0$ sufficiently small, since $\alpha_1>\alpha$, for all sufficiently large $n$,
    \begin{align*}
        \Exp{\xi_v^k}\leq \Exp{\xi_{\floor{n^{\alpha_1}}}^k}\leq (1-\alpha)^k(\log n)^kn^{f_\alpha(\alpha_1)+o_\eps(1)}%<(1-\alpha)^k(\log n)^kn^{f_\alpha(\alpha_1)}=o(1)
        \numberthis \label{eq:off_mom_UB}
    \end{align*}
where $o_\eps(1)$ denotes a term that vanishes as $\eps\to 0$. %For the second inequality above, we use that $f_\alpha(t)$ is decreasing on $(\alpha,1)$. 
   In particular, for any $\alpha \in (0,1), \alpha_1\in (\alpha,1)$ and $v\geq n^{\alpha_1}$, since $f_\alpha<0$ on $(\alpha,1)$, and the exponent of $n$ in the last display can be made negative by choosing $\eps>0$ small, a Bienaym\' e--Galton--Watson tree with offspring distribution $\xi_v$ dies out with probability one.
\end{remark}

Next, we control the moments of the total size $|\mathscr{T}(v,\alpha)|$ of the Bienaymé--Galton--Watson tree $\mathscr{T}(v,\alpha)$ for $v \geq n^{\alpha_1}$ for $\alpha_1>\alpha$.

\begin{lemma}[Bienaym\' e--Galton--Watson size moments]\label{lem:Bienaym\' e--Galton--Watson_size_moments}
Let $\mathscr{T}(v,\alpha)$ denote the total size of a Bienaym\' e--Galton--Watson tree with offspring distribution $\xi_{v}$, where we assume $v\geq n^{\alpha_1}$ with $\alpha_1 > \alpha$. Then for any $k \geq 1$, for all large $n$, $$\Exp{(\mathscr{T}(v,\alpha)-1)^k}\leq (\log n)^k n^{f_\alpha(\alpha_1)}.$$\end{lemma}

\begin{proof}
    Note that $\E[\mathscr{T}(v,\alpha)^k]=1+\Exp{\xi_v^k}\Exp{\mathscr{T}(v,\alpha)}^k$, by Wald's identity and the branching property. For $k=1$ we obtain $\Exp{\mathscr{T}(v,\alpha)}=\frac{1}{1-\Exp{\xi_{v}}}$, so that by Lemma \ref{lem:offspring_moments} and \eqref{eq:off_mom_UB}, for all large $n$,
    \begin{align*}
       \Exp{(\mathscr{T}(v,\alpha)-1)^k}\leq \Exp{\mathscr{T}(v,\alpha)^k}-1\leq \frac{\Exp{\xi_{v}^k}}{(1-\Exp{\xi_{v}})^k}\leq (\log n)^kn^{f_\alpha(\alpha_1)}.
    \end{align*}
\end{proof}

%We can finally provide the proof of Proposition \ref{prop:root_comp_UB}. For any vertex $u \in \cT_n$, denote by $p(u)$ its parent in $\cT_n$. 

\begin{proof}[Proof of Proposition \ref{prop:root_comp_UB}]
For fixed $\alpha_1 \in (\alpha,\alpha')$, by  Lemma \ref{lem:approx_offspring_forests} and the fact that $\bC_\alpha(1)\subseteq C_{\alpha_1}(1)$, it suffices to show that $\Prob{|C_{\alpha_1}(1)|>n^{\alpha'}}=o(1)$.
We claim that for any $M \in [n]$, 
    \begin{align*}
        C_{\alpha_1}(1)\subseteq [M]\cup \left(\bigcup_{v >M,p(v)\leq M}\rF_n(\alpha_1)(v;E_v) \right).
    \end{align*}
Indeed, for any vertex $u \in C_{\alpha_1}(1)$, either $u \in [M]$, or otherwise, since $C_{\alpha_1}(1)$ is a tree, we can trace the unique ancestral line of $u$ towards $1$ in $C_{\alpha_1}(1)$, and stop at the first ancestor $v$ which is at least $M$, with a parent $p(v)$ that is at most $M$.  Note that $u \in \rF_n(\alpha)(v; E_v)$.
In particular, for any $\alpha_1<\alpha_2<\alpha'$,
\begin{align*}
    |C_{\alpha_1}(1)|\leq \floor{n^{\alpha_2}}+\sum_{v>\floor{n^{\alpha_2}},p(v)\leq \floor{n^{\alpha_2}}}|\rF_n(\alpha)(v;E_v)|.\numberthis \label{eq:c_alpha_UB_rep}
\end{align*}
Since
\begin{align*}
    \Prob{|C_{\alpha_1}(1)|>n^{\alpha'}}\leq \Prob{\sum_{v>\floor{n^{\alpha_2}},p(v)\leq \floor{n^{\alpha_2}}}|\rF_n(\alpha)(v;E_v)|>\frac{n^{\alpha'}}{2}},
\end{align*}    
it suffices to show that the right-hand side above is small. Note that for any $v > \floor{n^{\alpha_2}}$, the random variables $\ind{p(v)\leq \floor{n^{\alpha_2}}}$ and $|\rF_n(\alpha)(v;E_v)|$ are independent. Thus, using Wald's identity, the stochastic domination from Remark \ref{rem:stoch_dom_Bienaym\' e--Galton--Watson} and Markov's inequality,
\begin{align}
    \Prob{|C_{\alpha_1}(1)|>n^{\alpha'}}\leq \frac{\Exp{|\{v> \floor{n^{\alpha_2}}:p(v)\leq \floor{n^{\alpha_2}}\}|}\Exp{|\mathscr{T}(\floor{n^{\alpha_2}},\alpha)|}}{n^{\alpha'}/2}.
\end{align}
Note that
\begin{align*}
    \Exp{|\{v> \floor{n^{\alpha_2}}:p(v)\leq \floor{n^{\alpha_2}}\}|}\leq  \sum_{j=\floor{n^{\alpha_2}}+1}^n\frac{n^{\alpha_2}}{j-1}\leq n^{\alpha_2}\int_{\floor{n^{\alpha_2}}-1}^n\frac{dx}{x}=O(n^{\alpha_2}\log n), \numberthis \label{eq:Ml_mean_UB} 
\end{align*}
as $n \to \infty$, and by Lemma \ref{lem:Bienaym\' e--Galton--Watson_size_moments}, $\Exp{|\mathscr{T}(\floor{n^{\alpha_2}},\alpha)|}=O(1)$, since $\alpha_2>\alpha$, so that
$\Prob{|C_{\alpha_1}(1)|>n^{\alpha'}}=o(1)$, 
 since $\alpha_2<\alpha'$. \end{proof}

\begin{remark}\label{rem:expec_UB_subtreeM(L)}
    Let us record here the expectation bound proved in the last argument:
    \begin{align*}
\Exp{\sum_{v>\floor{n^{\alpha_2}},p(v)\leq \floor{n^{\alpha_2}}}|\rF_n(\alpha)(v;E_v)|}=O(n^{\alpha_2}\log n)
    \end{align*}
    for any $\alpha_2>\alpha$, which follows from \eqref{eq:Ml_mean_UB}.
\end{remark}

Next, we establish an upper bound for the $k$-th moment on the root component size $|\bC_{\alpha}(1)|$ of $\cF_n(\alpha)$.

%\red{[N: this should be enough for the recursive argument. Is $P=P(l,\alpha'-\alpha)$?]}
\begin{proposition}\label{prop:higher_moment_UB}
    For any $k > 1$ and $\alpha'>\alpha$, there exists a function $P(k,\alpha'-\alpha)$, depending only on $k$ and the difference $\alpha'-\alpha$, such that $P(k,\cdot)$ is monotone decreasing for any $k \geq 1$, and 
    $\lim_{\eps\to 0}P(k,\eps)=\infty$,
    such that for all large $n$,
    \begin{align*}
        \Exp{|\bC_\alpha(1)|^k}
        \leq P(k,\alpha'-\alpha)   n^{k\left(2\alpha'-\alpha\right)}.
    \end{align*}
\end{proposition}

\begin{proof}
We choose $\alpha_1,\alpha_2$ with
$\alpha<\alpha_1<\alpha_2< \min(\alpha',
    2\alpha'-\alpha)$,
which is possible since $2\alpha'-\alpha>\alpha$. As in the proof of Proposition \ref{prop:root_comp_UB}, we have $\bC_\alpha(1)\subseteq C_{\alpha_1}(1)$, and thus, it suffices to prove the bound for $\Exp{|C_{\alpha_1}(1)|^k}$. Recall the representation \eqref{eq:c_alpha_UB_rep},
    \begin{align*}
        |C_{\alpha_1}(1)|\leq \floor{n^{\alpha_2}}+\sum_{v>\floor{n^{\alpha_2}},p(v)\leq \floor{n^{\alpha_2}}}|\rF_n(\alpha)(v;E_v)|.
    \end{align*}
    This implies,  
    \begin{align*}
        \Exp{|C_{\alpha_1}(1)|^k}&\leq \sum_{m=0}^k \binom{k}{m}n^{m\left(2\alpha'-\alpha \right)}\Exp{\left(\sum_{v>\floor{n^{\alpha_2}},p(v)\leq \floor{n^{\alpha_2}}}|\rF_n(\alpha)(v;E_v)| \right)^{k-m}},
    \end{align*}
    so that it suffices to show that for any $\ell\geq 1$,
    \begin{align*}
    \Exp{\left(\sum_{v>\floor{n^{\alpha_2}},p(v)\leq \floor{n^{\alpha_2}}}|\rF_n(\alpha)(v;E_v)| \right)^{\ell}}\leq Q(\ell,\alpha'-\alpha)n^{\ell\left(2\alpha'-\alpha \right)}, \numberthis \label{eq:RTP_UB_mom_bound}
    \end{align*}
    where the function $Q(\ell,\cdot)$ is monotone decreasing and satisfies $Q(\ell,\eps)\to \infty$ as $\eps\to 0$. The case $\ell=0$ is trivial, and the case $\ell=1$ follows from Remark \ref{rem:expec_UB_subtreeM(L)}. For the general case, note that
    \begin{align*}
        &\Exp{\left(\sum_{v>\floor{n^{\alpha_2}},p(v)\leq \floor{n^{\alpha_2}}}|\rF_n(\alpha)(v;E_v)| \right)^{\ell}}\\
        &=\Exp{\left(\sum_{v>\floor{n^{\alpha_2}}}\ind{p(v)\leq \floor{n^{\alpha_2}}}|\rF_n(\alpha)(v;E_v)| \right)^{\ell}}\\
        %&=\Exp{\sum_{v_1,\dots,v_\ell>\floor{n^{\alpha_2}}}\prod_{i=1}^\ell\left(\ind{p(v_i)\leq \floor{n^{\alpha_2}}}|\rF_n(\alpha)(v_i;E_{v_i})| \right)}\\
        &=\sum_{v_1,\dots,v_\ell>\floor{n^{\alpha_2}}} \Exp{\prod_{i=1}^\ell\left(\ind{p(v_i)\leq \floor{n^{\alpha_2}}}|\rF_n(\alpha)(v_i;E_{v_i})| \right)}. \numberthis \label{eq:l_mom_1}
    \end{align*}
By the generalized H\" older inequality, for any $\varepsilon>0$ we have
\begin{align*}
    &\Exp{\prod_{i=1}^\ell\left(\ind{p(v_i)\leq \floor{n^{\alpha_2}}}|\rF_n(\alpha)(v_i;E_{v_i})| \right) }\\
    & \leq \Prob{\max_{i: i\in [\ell]} p(v_i)\leq \floor{n^{\alpha_2}}}^{\frac{1}{1+\varepsilon}}
    \Exp{\left(\prod_{i=1}^\ell|\rF_n(\alpha)(v_i;E_{v_i})|\right)^{g(\varepsilon)}}^{\frac{1}{g(\varepsilon)}}\\
    &\leq n^{\frac{\ell\alpha_2}{1+\varepsilon}}
    \prod_{i=1}^\ell \frac{1}{(v_i-1)^{\frac{1}{1+\varepsilon}}}
    \prod_{i=1}^\ell \Exp{|\rF_n(\alpha)(v_i;E_{v_i})|^{\ell g(\varepsilon)}}^{\frac{1}{\ell g(\varepsilon)}},
\end{align*}
where $g(\varepsilon)=(1+\varepsilon)/\varepsilon$, and for the last line we use the independence of the events $\{p(v_i)\leq n^{\alpha_2}\}$ for $i \in [\ell]$. Using Lemma \ref{lem:Bienaym\' e--Galton--Watson_dom_subtreeforest}, Remark \ref{rem:stoch_dom_Bienaym\' e--Galton--Watson} and Lemma \ref{lem:Bienaym\' e--Galton--Watson_size_moments}, since $v_i>n^{\alpha_2}$ with $\alpha_2>\alpha$, we have
\begin{align*}
    \prod_{i=1}^\ell \Exp{|\rF_n(\alpha)(v_i;E_{v_i})|^{\ell g(\varepsilon)}}^{\frac{1}{\ell g(\varepsilon)}}\leq 1
\end{align*}
for all large $n$, for any $\varepsilon>0$. Thus, by the Riemann sum bound $$\sum_{v>\floor{n^{\alpha_2}}}\frac{1}{(v-1)^{\frac{1}{1+\varepsilon}}}\leq D\int_{\floor{n^{\alpha_2}}}^n\frac{1}{x^{\frac{1}{1+\varepsilon}}}dx$$ for a universal constant $D>0$, we conclude that
\begin{align*}
\Exp{\left(\sum_{v>\floor{n^{\alpha_2}},p(v)\leq \floor{n^{\alpha_2}}}|\rF_n(\alpha)(v;E_v)| \right)^{\ell}}
    &\leq n^{\frac{\ell\alpha_2}{1+\varepsilon}}\prod_{i=1}^\ell \left(\sum_{v>\floor{n^{\alpha_2}}}\frac{1}{(v-1)^{\frac{1}{1+\varepsilon}}} \right)\\
    &\leq D^\ell\left(\frac{1+\varepsilon}{\varepsilon} \right)^\ell n^{\frac{\ell\alpha_2}{1+\varepsilon}}n^{\frac{\ell\varepsilon}{1+\varepsilon}}.
\end{align*}
for all large $n$. Taking $\varepsilon=\alpha'-\alpha$, defining
\begin{align*}
    Q(\ell,\alpha'-\alpha):=D^\ell\left(\frac{1+\alpha'-\alpha}{\alpha'-\alpha} \right)^\ell,
\end{align*}
and bounding $\alpha_2<2\alpha'-\alpha$ proves \eqref{eq:RTP_UB_mom_bound}.
\end{proof}

\begin{remark}[Error bound]\label{rem:error_UB}
    The following bound follows from Proposition \ref{prop:higher_moment_UB}. For any $\beta>\alpha' > \alpha$ and $k\geq 1$,
\begin{align*}
\Prob{|\bC_{\alpha}(1)|>n^{\beta}}\leq P(k,\alpha'-\alpha) n^{k(2\alpha'-\alpha-\beta)}, \numberthis \label{eq:high_moment_Markov}
    \end{align*}
where $P(\cdot,\cdot)$ is the function appearing in Proposition \ref{prop:higher_moment_UB}.
\end{remark}

%Let us record the following corollary of Proposition \ref{prop:higher_moment_UB} for later use.

\begin{corollary}[Higher moment bounds for $C_\alpha(1)$]\label{cor:high_mom_off_forest}
    For any $k \geq 1$ and $\beta>\alpha' > \alpha$, $\Exp{|C_\alpha(1)|^k}\leq P(k,\alpha'-\alpha) n^{k(2\alpha'-\alpha)}$. Consequently, the corresponding tail bound of the previous remark is also true,
    \begin{align*}
        \Prob{|C_{\alpha}(1)|>n^{\beta}}\leq P(k,\alpha'-\alpha) \,n^{k(2\alpha'-\alpha-\beta)}.
    \end{align*}
\end{corollary}
\begin{proof}
    The result follows from the sandwiching of Remark \ref{rem:sandwiching_root_comps} and Proposition \ref{prop:higher_moment_UB}. 
\end{proof}
%\red{[ this error bound is not good enough for proving other components are small; need higher moment error bounds ]}

\subsection{Lower bound for the size of the root component}\label{sec:c1_LB}

In this section, we prove the following lower bound on the size of the root component $\bC_\alpha(1)$ of $\cF_n(\alpha)$. 

\begin{proposition}
\label{prop:lower_bound}
    For any $\alpha'<\alpha$, $\lim_{n \to \infty}\Prob{|\bC_\alpha(1)|<n^{\alpha'}}=0$.
\end{proposition}

Before proving this result, we state and prove a negative association property.
\begin{lemma}[Negative association]\label{lem:lem_negass}
    For any vertex $v$ in the \textsc{urrt} $\cT_n$, the degree
    $-d_{\cT_n}(v)$ of $v$ and the size $|\cT_n(v;E_v)|$ of the the descendant subtree of $v$ are negatively associated. In particular, for any $x>0$, $\Exp{|\cT_n(v;E_v)|\ind{d_{\cT_n}(v)<x}}\leq \Prob{d_{\cT_n}(v)<x}\Exp{|\cT_n(v;E_v)|}$.
\end{lemma}
\begin{proof}
    Recall the variables $\cY_v\stackrel{d}{=} \mathrm{U}[v-1]$ that determine the \textsc{urrt} $\cT_n$, by  $p(v)=\cY_v$ where $p(v)$ is the parent of $v$. Consider some $k \geq v+1$ such that $\cY_k=v$. Observe that changing $\cY_k=u\neq v$ makes $-d_{\cT_n}(v)$ to increase, while $|\cT_n(v;E_v)|$ either stays the same (if $j$ was a descendant of $v$ before the change) or decreases (if $j$ was not a descendant of $v$ before the change). This shows that $-d_{\cT_n}(v)$ and $|\cT_n(v;E_v)|$ are negatively associated, and proves the result.
\end{proof}

\begin{proof}[Proof of Proposition \ref{prop:lower_bound}]
    Consider exploring the component $\bC_\alpha(1)$ in $\cT_n$ starting from the root $1$ in a breadth-first manner. By the construction of $\rF_n(\alpha)$, all the vertices $u$ that we encounter in this exploration satisfy $c_{\cT_n}(u)>(1-\alpha)\log n$. This implies that whenever we encounter a vertex $v$ with degree at most $(1-\alpha)\log n$, we do not explore the subtree of $v$ in $\cT_n$, as it is not part of the root component $\bC_\alpha(1)$. We proceed in increasing order of the vertex labels, always exploring the yet unexplored vertex with the smallest label. Fix  $\alpha'<\alpha_1<\alpha$. 
    The following equality is a straightforward consequence of this exploration, when we explore only vertices in $[\lfloor n^{\alpha_1}\rfloor ]$: 
    \begin{align*}
        |\bC_\alpha(1)\cap [\lfloor n^{\alpha_1}\rfloor]|=\lfloor n^{\alpha_1}\rfloor -\sum_{v\in \mathrm{Bad}} |\cT_{\lfloor n^{\alpha_1}\rfloor}(v;E_v)|,
    \end{align*}
where the set $\mathrm{Bad}$ is defined by
\begin{align*} 
    \mathrm{Bad}:=\{v \in [\lfloor n^{\alpha_1}\rfloor]:c_{\cT_n}(v)<(1-\alpha)\log n\;\;\text{but all ancestors}\;\;u\;\;\text{of}\;\;v\;\;\text{satisfy}\;\;c_{\cT_n}(u)\geq (1-\alpha)\log n\}.
\end{align*}

\begin{figure}[H]
    \centering
    \resizebox{\textwidth}{!}{
        \input{LB_bad_vertex}}
    \caption{$\cT_n$ is the graph on the left and $\bC_\alpha(1)\cap[\floor{n^{\alpha_1}}]$ is on the right. The orange vertices have several offspring $>(1-\alpha)\log n$ ($=5$, say, for this illustrative example) in $\cT_n$. The green vertices have label $>\floor{n^{\alpha_1}}$, so are not part of $\bC_\alpha(1)\cap [\floor{n^{\alpha_1}}]$. The red vertices are all ``Bad'' vertices. The bad vertex that is the parent of the two green vertices \emph{breaks} the ancestral line from $1$, depicted in blue.}
    \label{fig:LB Bad v}
\end{figure}

Indeed, to count the size of the set $\bC_\alpha(1)\cap[\lfloor n^{\alpha_1}\rfloor]$ we may count all of the set $[\lfloor n^{\alpha_1}\rfloor]$, and then subtract the sizes of the descendant subtrees $|\cT_{\lfloor n^{\alpha_1}\rfloor}(v; E_v)|$ in $\cT_{\floor{ n^{\alpha_1}}}$ of any $v$ which \emph{breaks} an ancestral line coming down from $1$, in $\bC_\alpha(1)$, see Figure \ref{fig:LB Bad v}. The ``$\rm Bad$'' set comprises all such $v$. 
In particular, we obtain the lower bound
\begin{align*}
    |\bC_\alpha(1)|\geq |\bC_\alpha(1)\cap [\floor{n^{\alpha_1}}]|
    &=\floor{n^{\alpha_1}}-\sum_{v\in \mathrm{Bad}} |\cT_{\floor{n^{\alpha_1}}}(v;E_v)|\\&\geq \floor{n^{\alpha_1}}-\sum_{v \in [\floor{n^{\alpha_1}}]:c_{\cT_n}(v)<(1-\alpha)\log n} |\cT_{\floor{n^{\alpha_1}}}(v;E_v)|.
\end{align*}
Thus, 
\begin{align*}
    \Prob{|\bC_\alpha(1)|<n^{\alpha'}}\leq \Prob{\sum_{v \in [\floor{n^{\alpha_1}}]:c_{\cT_n}(v)<(1-\alpha)\log n} |\cT_{\floor{n^{\alpha_1}}}(v;E_v)|\geq n^{\alpha_1}/2}. \numberthis \label{eq:lb_conclusion}
\end{align*}

However, note that 
\begin{align*}
    &\Exp{\sum_{v \in [\floor{n^{\alpha_1}}]:c_{\cT_n}(v)<(1-\alpha)\log n}|\cT_{\floor{n^{\alpha_1}}}(v;E_v)|}\\
    &=\sum_{v \in [\floor{n^{\alpha_1}}]}\Exp{|\cT_{\floor{n^{\alpha_1}}}(v;E_v)|\ind{c_{\cT_n}(v)<(1-\alpha)\log n}}\\&\leq \sum_{v \in [\floor{n^{\alpha_1}}]}\Exp{|\cT_{\floor{n^{\alpha_1}}}(v;E_v)|}\Prob{c_{\cT_n}(v)<(1-\alpha)\log n}, 
\end{align*}
using Lemma \ref{lem:lem_negass}. Recalling $x(v)$ from \eqref{def:x(v)}, using Proposition \ref{prop:subtree_mom_UB} and Lemma \ref{lem:degUandLtail}, by a union bound, since for any $v\leq [\floor{n^{\alpha_1}}]$ we have $x(v)<\alpha$, the right-hand side above is at most 
\begin{align*}
    &\sum_{v\in [\floor{n^{\alpha_1}}]}\left(c_1+c_2(n^{\alpha_1-x(v)}-1)\right)n^{f_{\alpha}(x(v))}. 
\end{align*}
for some positive constants $c_1, c_2$.
By a Riemann integral bound, the last sum is at most
\begin{align*}
\int_1^{n^{\alpha_1}}&\left(c_1+c_2\left(n^{\alpha_1-\left(\frac{\log u}{\log n}\right)}-1\right)\right)n^{f_\alpha\left(\frac{\log u}{\log n}\right)}du\\
&=\log n\int_0^{\alpha_1}\left(c_1+c_2\left(n^{\alpha_1-y}-1\right)\right)n^{f_\alpha\left(y\right)}n^ydy\\
&=O(n^{\alpha_1+f_\alpha(\alpha_1)}\log n)=o(n^{\alpha_1}),
\end{align*}
where for the first equality above we change variables $u=n^y$, for the second one we use that $f_\alpha(\cdot)$ is increasing on $(0,\alpha)$ and for the third one that $f_\alpha(\alpha_1)<0$ since $\alpha_1<\alpha$. %Changing variables $v \mapsto n^u$, the above sum is bounded from above by the integral
%\begin{align*}
    %\int_{0}^{\alpha_1}\left(C_1+C_2(n^{\alpha_1-u}-1) \right)n^{f_\alpha(u)}n^u\log n du=O(n^{\alpha_1+f_\alpha(\alpha_1)}\log n)=o(n^{\alpha_1}),
%\end{align*}
%since $f_\alpha(\alpha_1)<0$ as $\alpha_1<\alpha$. 
Markov's inequality applied to the probability on the right-hand side of \eqref{eq:lb_conclusion} finishes the proof. \end{proof}

\subsection{Root component is a giant}\label{sec:root_comp_giant}
Thanks to Propositions \ref{prop:root_comp_UB} and \ref{prop:lower_bound}, with high probability 
\begin{align*}
    n^{\alpha_1}\leq |\bC_\alpha(1)| \leq n^{\alpha_2} \quad \text{and} \quad n^{\alpha_1}\leq |C_\alpha(1)| \leq n^{\alpha_2}, 
\end{align*}
 for any $\alpha_1<\alpha<\alpha_2$, where the second set of inequalities above holds due to Remark \ref{rem:sandwiching_root_comps}.

%\red{[ Sharpen bounds later: $n^{\alpha}$ with log corrections. ]}

In this section, we prove that all the other components in $\cF_n(\alpha)$ have significantly smaller size, so that the component of $\bC_\alpha(1)$ is clearly distinguishable when the forest $\cF_n(\alpha)$ is observed without vertex labels.

For this, we need an upper bound on the size of the component $\bC_\alpha(u)$ of any vertex $u$ in the forest $\cF_n(\alpha)$. %Recall by convention $\bC_\alpha(u)=\varnothing$ if $u \notin \cF_n(\alpha)$. 
We begin with an observation that all vertices up to $n^{\gamma}$ are part of the root component $\bC_\alpha(1)$, up to some $\gamma>0$. This is a simple corollary of Lemma \ref{lem:chernoff_deg_tail}.

\begin{corollary}\label{lem:all upto gamma}
    Recall $\gamma$ from \eqref{eq:def_gamma}. For any $\eps\in(0,\gamma)$, with high probability, $\bC_\alpha(1)\supseteq \cT_{\floor{n^{\gamma-\eps}}}$, where $\cT_{\floor{n^{\gamma-\eps}}}$ is the subtree of $\cT_n$ spanned by the first $\floor{n^{\gamma-\eps}}$ vertices.
\end{corollary}
\begin{proof}
    It suffices to show that $\Prob{\min_{i: i\in [\floor{n^{\gamma-\eps}}]} d_{\cT_n}(i)<(1-\alpha)\log n}=o(1)$, and this is a direct consequence of Lemma \ref{lem:chernoff_deg_tail}.\end{proof}

\begin{proposition}[Later components are small]\label{prop:later_small}
From \eqref{eq:def_gamma} we recall that $\gamma<\alpha$. For any constants $\beta$ and $\gamma'$ with
   $\beta\in (\alpha-\gamma,\alpha),\textrm{ and }\gamma'>\gamma,$
   %\red{[ the inequalities above don't make sense, we know $\gamma<\alpha$ ]}
   we have
\begin{align*}
        \lim_{n \to \infty}\Prob{\max_{u: u\geq n^{\gamma'},u \notin \bC_\alpha(1)} |\bC_\alpha(u)|\geq n^{\beta}}=0.
    \end{align*}
\end{proposition}

Before proving this result, let us first use it to show that one can distinguish the root component when the graph $\cF_n(\alpha)$ is observed without the vertex labels. 

\begin{lemma}[Distinguishing the root component]\label{lem:disting_root_comp}
    Observing the graph $\cF_n(\alpha)$ without its vertex labels, we can distinguish the root component as the unique connected component of $\cF_n(\alpha)$ with at least $n^{\alpha-\gamma/4}$ vertices, where $\gamma$ is defined in \eqref{eq:def_gamma}.
\end{lemma}

\begin{proof}
We apply Proposition \ref{prop:later_small} with $\beta=\alpha-\gamma/2$.
The statements we make below are to be interpreted as holding with high probability. 
Let $\gamma'$ be such that $\gamma<\gamma'<\alpha-\gamma/2$. (Recall from Lemma \ref{lem:gamma_cont} part (iv) that such a choice of $\gamma'$ is indeed possible.)
By Proposition \ref{prop:later_small}, the components of all vertices $u>n^{\gamma'}$ that are not in $\bC_\alpha(1)$,  have to have size at most $n^\beta=n^{\alpha-\gamma/2}\ll n^{\alpha-\gamma/4}$. On the other hand, by Corollary \ref{lem:all upto gamma}, all vertices $u \in [\floor{n^{\gamma-\eps}}]$ are part of $\bC_\alpha(1)$, which, by Proposition \ref{prop:lower_bound}, has size at least $n^{\alpha-\eps}\gg n^{\alpha-\gamma/4}$ by choosing $\eps>0$ small, since $\gamma/4<\alpha$. Thus, we need only consider the components of the vertices $u$ satisfying $n^{\gamma-\eps}<u<n^{\gamma'}$. If these vertices are either part of $\bC_\alpha(1)$ or of a component containing a vertex $u>n^{\gamma'}$, then $\bC_\alpha(1)$ is the unique component with size at least $n^{\alpha-\gamma/4}$. If these vertices form their own component in $\cF_n(\alpha)$, the size of such a component is at most of order $n^{\gamma'}\ll n^{\alpha-\gamma/4}$. In any case, $\bC_\alpha(1)$ is distinguishable as the unique component in $\cF_n(\alpha)$ with size at least $n^{\alpha-\gamma/4}$.
\end{proof}

\begin{comment}
To prove Proposition \ref{prop:later_small}, we first need an intermediate result on subtree sizes of later vertices in the \textsc{urrt} $\cT_n$.

\red{[ add (the modified form of) this result in preliminary results ]}

\begin{lemma}[Concentration of subtree sizes]\label{lem:conc_subtree_sizes}
    For any fixed $y\in(0,1)$ and $\eps>0$, there exists $M=M(\eps)>0$ large enough such that for all large $n$,
    \begin{align*}
        \Prob{\cE(n^y,M)^c}<\eps,\;\;\text{where}\;\;\cE(n^y,M):=\{\forall\;\;u\geq n^{y},|\cT_n(u;[u])|/(n/u)\in[1/M,M]\}. \numberthis \label{eq:conc_subtree_sizes}
    \end{align*}
\end{lemma}
\begin{proof}
    \red{[ to be added later, seems standard. 
    
    NOTE: we use this later with $y=\gamma-\eps$ for small $\eps>0$. ]}
\end{proof}
\end{comment}

\begin{proof}[Proof of Proposition \ref{prop:later_small}]
Consider the subtree $\cT_n(u;E_u)$ of $u$ in $\cT_n$ for any $u \in [n]$. %Define the subgraph $\cF_n(u;E_u)$, obtained by keeping those vertices of $\cT_n(u;[u])$ that have degree at least $(1-\alpha)\log n$. Note that $\cF_n(u;[u])$ can be empty graph. 
Denote $C(u,n)=\psi^{(1-\alpha)\log n-1}_{\rm off}(\cT_n(u;E_u),u)$, where $\psi^r_{\rm off}$ is as defined in Definition \ref{def:psi_off_op}.  %and with the convention that $C(u,n)$ is empty if $u \notin \cF_n(u;[u])$. 
Recall $\beta \in (\alpha-\gamma,\alpha)$. Thanks to Corollary \ref{lem:all upto gamma}, it suffices to show that for some $\eps>0$ sufficiently small such that $\gamma-\eps>0$ and $\beta\in (\alpha-\gamma+\eps,\alpha)$,
\begin{align*}
    \Prob{\left\{\max_{u: u\geq n^{\gamma'},u \notin \bC_\alpha(1)} |\bC_\alpha(u)|\geq n^{\beta}\right\}\cap\left\{\bC_\alpha(1)\supseteq \cT_{\floor{n^{\gamma-\eps}}}\right\}}=o(1).
\end{align*}
We have the following inclusion of events:
\begin{align*}
    \left\{\max_{u: u\geq n^{\gamma'},u \notin \bC_\alpha(1)} |\bC_\alpha(u)|\geq n^{\beta}\right\}
    \cap \left\{\bC_\alpha(1)\supseteq \cT_{\floor{n^{\gamma-\eps}}}\right\}
    \subseteq 
    \left\{\max_{v:  v \geq n^{\gamma-\eps}} |C(v,n)|\geq n^{\beta}
    \right\}. \numberthis \label{eq:later_inclusion}
\end{align*}
Indeed, assume that the event on the left-hand side holds. For any $u\geq n^{\gamma'}$ that is not in $\bC_\alpha(1)$, we take the smallest vertex $v \in \bC_\alpha(u)$. Since $\bC_\alpha(1)\supseteq \cT_{\floor{n^{\gamma-\eps}}}$, we must have $v\geq \floor{n^{\gamma-\eps}}+1$ as $v \in \bC_\alpha(u)\neq \bC_\alpha(1)$. Observe that $|C(v,n)|=|\bC_\alpha(v)|$. Indeed, since $v$ is the smallest vertex in $\bC_\alpha(v)$, the latter is a subtree of $\cT_n(v;E_v)$, and by definition of the construction of the forest $\cF_n(\alpha)$, since $v\neq 1$, this subtree is equivalent to the component of $v$ when one retains those edges in $(\cT_n(v;E_v),v)$ whose both end-vertices have at least $(1-\alpha)\log n-1$ children.
%precisely how the operator $\psi^{(1-\alpha)\log n-1}_{\rm off}$ acts on $(\cT_n(v;E_v),v)$. 
In particular, $|C(v,n)|=|\bC_\alpha(v)|=|\bC_\alpha(u)|\geq n^{\beta}$, so the right-hand side event in the last display above holds. 
%\red{[ fix later: small issue about this equality with degree vs offspring cutting: in the original degree cutting, we are actually cutting by offspring, but only for the root. ]} 
Thus, it suffices to show that for any $\delta>0$, for all large $n$
\begin{align*}
    \Prob{\max_{v: v \geq n^{\gamma-\eps}} |C(v,n)|\geq n^{\beta}}< \delta. \numberthis \label{eq:to_show_rootgiant}
\end{align*}
To this end, we exploit the recursive nature of the distribution of the tree $\cT_n$. From Lemma \ref{lem:cond_indep_subtrees}, conditionally on the event $\{|\cT_n(v; E_v)|=T\}$, the tree $\cT_n(v; E_v)$ has the law $\rho_T$, that is, it is distributed as a random recursive tree of size $T$ with root vertex $v$. Thus, applying the explicit error bound of Corollary \ref{cor:high_mom_off_forest} on this smaller recursive tree, we obtain that for any $T=T_n(v)$ with $T_n(v)\gg 1$, and for any fixed $k \geq 1$
\begin{align*}
    \CProb{|C(v,n)|>T^{\alpha'}}{|\cT_n(v;E_v)|=T}
    \leq P(k,\alpha_1-\alpha_*(T)) T^{k(2\alpha_1-\alpha_*-\alpha')}, \numberthis \label{eq:error_recursive}
\end{align*}
for any $\alpha',\alpha_1$ satisfying $\alpha'>\alpha_1>\alpha_*(T)$, where $\alpha_*(T)$ solves
\begin{align*}
    (1-\alpha_*(T))\log T=(1-\alpha) \log n \textrm{, i.e., }  \alpha_*(T) = 1-(1-\alpha)\frac{\log n}{\log T}. \numberthis \label{eq:def_alpha_recursive}
\end{align*}
From the last equation, observe  that
$\alpha_*(T)$ is monotone increasing in $T$.
%    \numberthis \label{eq:alpha_*_inc}
By Proposition \ref{prop:unif_subtree_tail}, 
\begin{align*}
&    \Prob{\max_{v: v \geq \floor{n^{\gamma-\eps}}} |C(v,n)|\geq n^{\beta}} 
\leq 
\Prob{\left\{\max_{v: v \geq \floor{n^{\gamma-\eps}}} |C(v,n)|\geq n^{\beta}\right\}\cap \cE(\gamma-\eps,n)}+\frac {\delta}{2}\\ \numberthis \label{eq:to_show_smallcomp_decomp}
\end{align*}
for all large $n$, where 
\begin{align*}
\cE({\gamma-\eps},n):=
\bigcap_{v: v\geq \floor{n^{\gamma-\eps}}}
    \left\{|\cT_n(v;E_v)|< \frac{6n \log n}{v}\right\}.
\end{align*}
Recalling $x(v)=\log v /\log n$ from \eqref{def:x(v)}, for $v \geq \floor{n^{\gamma-\eps}}$ we have $x(v) \in [\gamma-\eps,1]$. Thus, writing $u=x(v)$, for any $u\geq \gamma-\eps$,
\begin{align*}
    \Prob{\{|C(n^{u},n)|\geq n^{\beta}\}\cap\cE(\gamma-\eps,n)}\leq\CProb{|C(n^u,n)|\geq n^{\beta}}{\cE(\gamma-\eps,n)},
\end{align*}
and since on $\cE(\gamma-\eps,n)$ we have $n>(|\cT_n(n^u;E_{n^u})|/6\log n)^{1/(1-u)}$, we can further bound 
\begin{align*}
    &\CProb{|C(n^u,n)|\geq n^{\beta}}{\cE(\gamma-\eps,n)}\\
    &\leq \CProb{|C(n^u,n)|\geq (6 \log n)^{\frac{\beta}{u-1}}|\cT_n(n^u;E_{n^u})|^{\frac{\beta}{1-u}}}{\cE(\gamma-\eps,n)}\\&\leq \CProb{|C(n^u,n)|\geq |\cT_n(n^u;E_{n^u})|^{\frac{\beta_1}{1-u}}}{\cE(\gamma-\eps,n)}\\&=\CExp{\CProb{|C(n^u,n)|\geq |\cT_n(n^u;E_{n^u})|^{\frac{\beta_1}{1-u}}}{|\cT_n(n^u;E_{n^u})|}}{\cE(\gamma-\eps,n)},
\end{align*}
for any $\beta_1\in(\alpha-\gamma+\eps,\beta)$.
Denote
$\alpha_*(u)=\alpha_*(|\cT_n(n^u;E_{n^u})|)$.
By \eqref{eq:error_recursive}, for any $k \geq 1$, the probability inside the conditional expectation above is at most 
\begin{align*}
P\left(k,\delta'\right)|\cT_n(n^u;E_{n^u})|^{k\left(2\delta'+\alpha_*(u)-\frac{\beta_1}{1-u}\right)},\numberthis \label{eq:recursive_tail_bd}
\end{align*}
for any $\delta'=\delta'(u)>0$. This bound is non-trivial if we can choose such a $\delta'\in \left(0,\frac{\beta_1}{1-u}-\alpha_*(u) \right)$, which is possible if $\beta_1/(1-u)>\alpha_*(u)$, as we show next.
To check this, by definition of $\alpha_*$, note that on the event $\cE(\gamma-\eps,n)\supseteq \{|\cT_n(n^u;E_{n^u})|/n^{(1-u)}< 6 \log n\}$, 
\begin{align*}
    \alpha_*(u)\leq 1-\frac{1-\alpha}{1-u+\delta_1}\numberthis \label{eq:alpha*_UB}
\end{align*}
for $\delta_1>0$ arbitrary small but independent of $u \in [\gamma-\eps,1]$. 
 Thus $\alpha_*(u)<\beta_1/(1-u)$ holds since
\begin{align*}
    -u+\delta_1+\alpha\leq \delta_1+\eps+\alpha-\gamma<\frac{1-u+\delta_1}{1-u}\beta_1.
\end{align*}
Indeed, the first inequality above is true since $u\geq \gamma-\eps$, the second is true by letting $\delta_1,\eps>0$ sufficiently small and letting $\beta_1$ sufficiently close to $\beta$, and recalling that $\beta+\gamma>\alpha$, by definition of $\beta$. In particular, we can choose $\delta'=\delta'(u)$ such that
\begin{align*}
    \delta' \in \left(0, \frac{\beta_1}{1-u}-1+\frac{1-\alpha}{1-u+\delta_1} \right).
\end{align*}
 From now on, we work with $\delta'$ satisfying
\begin{align*}
    \delta'\in \left(0, \frac{\beta_1}{1-\gamma-\eps}-1+\frac{1-\alpha}{1-\gamma+\eps+\delta_1} \right),\numberthis \label{eq:delta'_interval}
\end{align*}
so that the $\delta'$ appearing in the upper bound expression \eqref{eq:recursive_tail_bd} is \emph{uniform} over $u\in [\gamma-\eps,1]$.

\begin{comment}
Additionally, note that since $P(\cdot,\cdot)$ is decreasing in the second argument, for any $u \geq n^{\gamma-\eps}$, using \eqref{eq:alpha*_UB} and the fact that $\alpha_*(\cdot)$ is monotone increasing \eqref{eq:alpha_*_inc},
\begin{align*}
    P\left(k,\frac{\beta_1}{1-u}-\alpha_*(u)\right)<Q(k,\beta_1,\gamma,\eps,\alpha,\delta_1):=P\left(k,\frac{\beta_1}{1-\gamma+\eps}-\left(1-\frac{1-\alpha}{1-(\gamma-\eps)+\delta_1}\right)\right).
\end{align*}
\end{comment}

Writing $Q=P(k,\delta')$, we have established that for any $k \geq 1$ and any vertex $v$ such that $x(v)=u\geq \gamma-\eps$, 
\begin{align*}
&\Prob{\{|C(n^u,n)|\geq n^{\beta}\}\cap\cE(\gamma-\eps,n)}\\
&\leq Q\,\Exp{|\cT_n(n^u;E_{n^u})|^{k\left(2\delta'+\alpha_*(u)-\frac{\beta_1}{1-u}\right)}\ind{\cE(\gamma-\eps,n)}}\\
&\leq Q \, \left(\frac{n^{(1-u)}}{6 \log n}\right)^{k\left(2\delta'+\left(1-\frac{1-\alpha}{1-u-\delta_1}\right)-\frac{\beta_1}{1-u}\right)}\\
&\leq Q\, \left(n^{(1-u-\delta_2)}\right)^{k\left(2\delta'+\left(1-\frac{1-\alpha}{1-u-\delta_1}\right)-\frac{\beta_1}{1-u}\right)},
\end{align*}
%\red{[ issue in second inequality with monotonicity of $\alpha_*$ ]}
for any $\delta_1,\delta_2,\delta'>0$, independent of $u \in [\gamma-\eps,1]$. Thus, by a union bound,
\begin{align*}
&\Prob{\left\{\max_{v: v \geq n^{\gamma-\eps}}|C(v,n)|\geq n^{\beta}\right\}\cap \cE(\gamma-\eps,n)}\\
&\leq \sum_{v \geq \floor{n^{\gamma-\eps}}}
Q\, \left(n^{(1-x(v)-\delta_2)}\right)^{k\left(2\delta'+\left(1-\frac{1-\alpha}{1-x(v)-\delta_1}\right)-\frac{\beta_1}{1-x(v)}\right)}\\
&\leq Q\log n\int_{\gamma-\eps}^1 n^y\cdot\left(n^{(1-y-\delta_2)}\right)^{k\left(2\delta'+\left(1-\frac{1-\alpha}{1-y-\delta_1}\right)-\frac{\beta_1}{1-y}\right)}dy, \numberthis \label{eq:recursive_integralUB}
\end{align*}
where for the last upper bound we applied a Riemann integral bound for the sum $\sum_{v \in \floor{n^{\gamma-\eps}}}$ and then changed variables $v=n^y$, or, equivalently, $x(v)=y$.

By choosing $\delta_1,\delta_2, \delta'$ sufficiently small, the exponent of $n$ in the integrand above can be made at most
$y+k(\alpha-y-\beta_1+\delta_3)$,
for $\delta_3>0$ arbitrary, so   bounding $y\geq \gamma-\eps$, the exponent can be made at most 
$(1-k)(\gamma-\eps)+k(\alpha-\beta_1+\delta_3)$
uniformly over $y\geq \gamma-\eps$. In particular, recalling $\beta>\alpha-\gamma$, by choosing $\beta_1$ sufficiently close to $\beta$, choosing $\eps,\delta_3>0$ sufficiently small, and $k$ sufficiently large ($k>\frac{\gamma}{\gamma-(\alpha-\beta)}$ suffices), this exponent can be made negative. For all these choices, if we let $n \to \infty$, it follows that the right-hand side of \eqref{eq:recursive_integralUB} is $o(1)$,  which proves \eqref{eq:to_show_rootgiant}.\end{proof}

\subsection{Finding the root inside the root component}\label{sec:root_finding_c1}
  In this section, we prove that just observing the graph structure of the  root component $\bC_\alpha(1)$, 
one can construct a confidence set for the root by taking the vertices with highest Jordan centrality in $\bC_\alpha(1)$.

\begin{theorem}\label{thm:root_finding_c1}[Root finding in $\bC_\alpha(1)$.]
Consider the root component $\bC_\alpha(1)$ in $\cF_n(\alpha)$ and fix $\delta>0$. Observing $\bC_\alpha(1)$ without the vertex labels, it is possible to construct a confidence set $K=K(\delta)$ whose size does not depend on $n$ such that
\begin{align*}
    \liminf_{n \to \infty}\Prob{1\in K}\geq 1-\delta.
\end{align*}
\end{theorem}

%To prove this theorem, we need a few preliminary results. We begin with a lemma that shows the existence of a $\gamma>0$ such that the root component $C(1)$ contains the random recursive tree up to time $n^{\gamma}$.

 We begin the proof by defining Jordan centrality of a vertex. %following \cite{findingadam}.

\begin{definition}[Jordan centrality.]
\label{def:JC}
    Given any tree $T$, a vertex $v \in T$ and a neighbor $u$ of $v$ in $T$, denote by $(T,v)_{u,\downarrow}$ the connected component containing $u$, when
    the vertex $v$ together with all edges incident to it is removed from $T$.
    Then the \emph{Jordan centrality} of $v$ in $T$ defined as 
    \begin{align*}
        \mathscr{J}_T(v):=\max_{u \in N_T(v)}|(T,v)_{u,\downarrow}|,
    \end{align*}
    where $N_T(v)$ is the \emph{neighborhood} of $v \in T$, that is, the set of all vertices incident to $v$ in $T$.
\end{definition}

\begin{remark}
For any tree $T$, a set of its edges $S$, and for any $v \in S$, we denote by $T(v; S)$ the subtree of $T$ containing $v$ when all the edges from $S$ have been removed from it. Note that
\begin{align*}
(T,v)_{u,\downarrow}=T(u;N_T(v)),\textrm{  where  }    N_T(v):=\{\{v,u\}\in E(T):u\in V(T)\}.
\end{align*}
\end{remark}

%\red{[ N: maybe different notation? Used $\psi$ for operators as well ]}
One expects the vertex $1$ to have a low value of $\mathscr{J}_T$ in the tree $\bC_\alpha(1)$. In particular, we construct the set $K=K(\delta)$ as claimed by Theorem \ref{thm:root_finding_c1} by taking the set of $M$ vertices with the smallest Jordan centrality in the sense of Definition \ref{def:JC}, where $M$ is a constant independent of $n$, and argue that with probability at least $1-\delta$ this set contains $1$. 

For a tree $T$, let the decreasing order (with ties broken arbitrarily) of the Jordan centrality measures of the vertices in $T$ be
$\mathscr{J}_T(u_1)\leq \mathscr{J}_T(u_2)\leq \dots \leq \mathscr{J}_T(u_{|T|})$.
For any $M\in \N$ and a tree $T$, define 
\begin{align*}
    H(M,T):=\{u_1,u_2,\dots,u_M\},\numberthis \label{eq:set H}
\end{align*}
that is, $H_{M}$ is the set of the $M$ most central vertices. 
%\red{[ $T(u;N_T(v))$ above? works because we are in a tree and there are no edges among neighbours of $v$. In general, I think better to make the definition $G(v;S)$ where $S$ is an edge-subset ]}
By  Lemma \ref{lem:all upto gamma}, with high probability, the vertices $1,2,\dots,\floor{n^{\gamma'}}$ are all in $\bC_\alpha(1)$ for any $\gamma'<\gamma$. Fix such a $\gamma'$. For any $i \in [\floor{n^{\gamma'}}]$, let 
\begin{align*}
    Y_\alpha(i)=|\bC_\alpha(1)(i;E_{\floor{n^{\gamma'}}})|,\,\textrm{ and }
    Z=\sum_{i=1}^{\floor{n^{\gamma'}}}Y_\alpha(i),
    %\text{ and }
%Y'_\alpha(i)=Y_\alpha(i)\big|_{\Xi,\mathscr{F}_{\floor{n^{\gamma'}}}}, 
\numberthis\label{eq:Y'}
\end{align*}
%\red{[ $Y(j,n^{\gamma'})=Y_j$ above ]}
%\red{[ $Y_\alpha(v)$ rather (for later) ]}
and for any $t\geq 0$, let $\mathscr{F}_t$ denote the $\sigma-$algebra generated by all the information up to step $t$ in the construction of $\cT_n$. 

\begin{lemma}[Exchangeability and negative correlation.]\label{lem:neg_cor}
%\textcolor{red}{
%\st{The collection of random variables  $\{Y'_\alpha(j):j \in [\floor{n^{\gamma'}}]\}$} 

%}
Conditionally on $Z,\mathscr{F}_{\floor{n^{\gamma'}}}$, 
$\{Y_\alpha(j):j \in [\floor{n^{\gamma'}}]\}$ is an exchangeable collection, and they are pairwise negatively correlated.
\end{lemma}
\begin{proof}
   Note that if $i\in [\floor{n^{\gamma'}}]$ has no child in $\cT_n$ with degree at least $(1-\alpha)\log n$, then it is a leaf in $\bC_\alpha(1)$ and $\bC_\alpha(1)(i;E_{\floor{n^{\gamma'}}})$ is thus the singleton graph on $\{i\}$. Otherwise, any edge in the graph $\bC_\alpha(1)(i;E_{\floor{n^{\gamma'}}})$ either connects $i$ to a vertex with label at least $\floor{n^{\gamma'}}+1$ which has degree at least $(1-\alpha)\log n$ in $\cT_n$, or connects two such vertices. Since $\bC_\alpha(1)(i; E_{\floor{n^{\gamma'}}})$ is connected, any such vertex has a path to $i$ in $\bC_\alpha(1)(i; E_{\floor{n^{\gamma'}}})$, and must be also in $\cT_n(i; E_{\floor{n^{\gamma'}}})$. 
   
   In any case, $\bC_\alpha(1)(i; E_{\floor{n^{\gamma'}}})$ is the connected component of $i$, in the subforest of $\cT_n(i; E_{\floor{n^{\gamma'}}})$ spanned by the edges in it that have both end-vertices with degree $>(1-\alpha)\log n$ in $\cT_n$. Equivalently, we can  retain edges in $\cT_n(i;E_{\floor{n^{\gamma'}}})$ with both end-vertices having number of offspring greater than $(1-\alpha)\log n-1$ in $\cT_n(i;E_{\floor{n^{\gamma'}}})$, and take the connected component of $i$ thus formed. We conclude,
   \begin{align*}
       \psi^{(1-\alpha)\log n-1}_{\rm off}(\cT_n(i;E_{\floor{n^{\gamma'}}}),i)=\bC_\alpha(1)(i;E_{\floor{n^{\gamma'}}}),\numberthis \label{eq:subtrees_in_comp_op_exp}
   \end{align*}
where the latter is seen as a rooted tree with root $i$. Since $Y_\alpha(i)=|\bC_\alpha(1)(i;E_{\floor{n^{\gamma'}}})|$, the result follows from Corollary \ref{cor:neg_corr}.\end{proof}

Fix a positive integer $M$. For any $i \in [M]$ consider the subtree $\bC_\alpha(1)(i;[M])$ of $\bC_\alpha(1)$. Let $C_{(\gamma')}$ be the subgraph of $\bC_\alpha(1)$ spanned by the set of vertices $\{v \in \bC_\alpha(1):v\in [\floor{n^{\gamma'}}]\}$. Note that by Lemma \ref{lem:all upto gamma}, with high probability, $C_{(\gamma')}=\cT_{\floor{n^{\gamma'}}}$. On the event $C_{(\gamma')}=\cT_{\floor{n^{\gamma'}}}$, consider the decomposition 
\begin{equation}\label{eq:key decomp}
|\bC_\alpha(1)(i;E_M)|=\sum_{v \in C_{(\gamma')}(i;E_M)}|\bC_\alpha(1)(v;E_{\floor{n^{\gamma'}}})|.    
\end{equation}
Indeed, we may view $(\bC_\alpha(1)(i;E_M),i)$ as a tree rooted at $i$ and consider any subtree $\bt$ of it, also rooted at $i$. Note that for any such subtree, if we remove all the edges from it, and sum over the sizes of the connected components of all the vertices in $(\bC_\alpha(1)(i; E_M), i)$ thus formed, we get back the total size of the tree $\bC_\alpha(1)(i; E_M)$. Taking $\bt=C_{(\gamma')}(i; E_M)$ and noting that on the event $C_{(\gamma')}=\cT_{\floor{n^{\gamma'}}}$ removing all the edges of $\cT_{\floor{n^{\gamma'}}}$ from $\bC_\alpha(1)(i; E_M)$ is the same as removing all edges from it that are also in $E_{\floor{n^{\gamma'}}}$, recovers the decomposition \eqref{eq:key decomp}.

%One way to decompose its size, is to consider a subtree $\bt$ of $\bC_\alpha(1)(i; E_M)$, also rooted at $i$, and to sum the number of \emph{descendants} (i.e., vertices that can be reached by a \emph{downward} path going away from the root $i$) that are in $\bC_\alpha(1)(i; E_M)$ but not in $\bt$, over all the vertices of $\bt$, see Figure \red{[ add ]}. 

%Taking this subtree $\bt$ to be $C_{(\gamma')}(i;E_M)$, we note that $|\bC_\alpha(1)(v;E_{n^{\gamma'}})|-1$ is precisely the number of all such descendants of $v\in C_{(\gamma')}(i;E_M)$ (the $-1$ coming from the fact that we count \emph{strict} descendants and thus leave out $v$ itself from $\bC_\alpha(1)(v;E_{n^{\gamma'}})$).

Lemma \ref{lem:neg_cor} is useful for the following reason.  To understand the $M$ most central vertices in $\bC_\alpha(1)$, we need to compare the sizes of 
$$|\bC_\alpha(1)(1;E_M)|,\dots,|\bC_\alpha(1)(M;E_M)|.
$$ 
Thanks to Lemma \ref{lem:neg_cor} and the decomposition \eqref{eq:key decomp}, we see that this boils down to comparing $C_{(\gamma')}(1; E_M),\dots, C_{(\gamma')}(M; E_M)$, which is the same as comparing $\cT_{\floor{n^{\gamma'}}}(1; E_M),\dots,\cT_{\floor{n^{\gamma'}}}(M, E_M)$ due to Lemma \ref{lem:all upto gamma}, and the latter comparison can be done by classical Pólya urn arguments.

With this goal in mind, let state a concentration inequality for the variables $|\bC_\alpha(1)(i;[M])|$. We simplify the notation slightly and write
\begin{align*}
    C^1_{i;M}=|\bC_\alpha(1)(i;E_M)|
    \quad \text{and} \quad C^{\gamma'}_{j;M}=|C_{(\gamma')}(j;E_M)|
\end{align*}
for any $M \in \N$, $i \in [M]$, $\gamma'>0$ and $j \in [\floor{n^{\gamma'}}]$.

\begin{proposition}[Concentration about conditional mean]\label{prop:conc_cond_mean}
    For any fixed $M\in \N$ and $\eps>0$, as $n \to \infty$,
    \begin{align*}
        \Prob{\bigcup_{j\in [M]}\left\{\left|C^1_{j;M}-C^{\gamma'}_{j;M}\CExp{Y_\alpha(1)}{Z}\right|>\eps\left( C^{\gamma'}_{j;M}\CExp{Y_\alpha(1)}{Z}\right)\right\}}=o(1).
    \end{align*}
\end{proposition}

\begin{lemma}[Moment bound]\label{lem:sec_mom_Yv}
    For any $\alpha\in (0,1)$, $k > 0$ a fixed integer, $v \in [\floor{n^{\gamma'}}]$ where $\gamma'<\gamma=\gamma(\alpha)$, and any $\varepsilon>0$, as $n \to \infty$, $$\Exp{Y_\alpha(v)^k}=O\left(n^{k(\alpha-\gamma'+\varepsilon)} \right).$$
\end{lemma}

\begin{proof}
We employ a recursive argument similar to the proof of Proposition \ref{prop:later_small}. Recall $Y_\alpha(v)=|\psi^{(1-\alpha)\log n-1}_{\rm off}(\cT_n(i;E_{\floor{n^{\gamma'}}}),i)|$ for any $v\in [\floor{n^{\gamma'}}]$ from \eqref{eq:subtrees_in_comp_op_exp}, and thus by Corollary \ref{cor:neg_corr}, taking $f\equiv \psi^{(1-\alpha)\log n-1}_{\rm off} $, we see that $(Y_\alpha(1),\dots,Y_\alpha(\floor{n^{\gamma'}}))$ are exchangeable. In particular, it suffices to prove the result for $$v=\floor{n^{\gamma'}}, $$ and we assume this for the rest of the proof.
Recall from Lemma \ref{lem:cond_indep_subtrees} the fact that conditionally on the event $\{|\cT_n(v;E_{\floor{n^{\gamma'}}})|=t_n\}$, the distribution of $(\cT_n(v;E_{\floor{n^{\gamma'}}}),v)$ seen as a random element of $\cT^*$ follows the law $\rho_{t_n}$ of a \textsc{urrt} of size $t_n$.
Thus, letting $\alpha_*(t_n)$ satisfy
\begin{align*}
    (1-\alpha_*(t_n))\log t_n=(1-\alpha)\log n
    \text{ i.e., }\alpha_*(t_n)=1-(1-\alpha)\frac{\log n}{\log t_n},
\end{align*}
and recalling the definition $\psi^{(1-\alpha)\log n}_{\rm off}(\cT_n,1)=C_\alpha(1)$, for any $k\geq 1$ and any $\alpha'>\alpha_*(t_n)$, we obtain by Corollary \ref{cor:high_mom_off_forest},
\begin{align*}
\CExp{Y_\alpha(v)^k}{|\cT_n(v;E_{\floor{n^{\gamma'}}})|=t_n}\leq P(k,\alpha'-\alpha_*(t_n)) \,
t_n^{k(2\alpha'-\alpha_*(t_n))}.
\end{align*}
  We choose
\begin{align*}
\alpha'=\alpha'(\delta,t_n)=1+\delta-(1-\alpha)\frac{\log n}{\log t_n}=\delta+\alpha_*(t_n),
\end{align*}
where $\delta>0$ is chosen later. Observe that $\alpha'(\delta,t_n)$ is increasing in $t_n$. Thus, we can write 
    \begin{align*}
        &\CExp{Y_\alpha(v)^k}{|\cT_n(v;E_{\floor{n^{\gamma'}}})|=t_n}\\&\leq P(k,\alpha'(\delta,t_n)-\alpha_*(t_n))n^{k(1-\gamma'+\delta_1)\left[2\delta+\alpha_*(t_n)\right]}\ind{t_n<n^{1-\gamma'+\delta_1}}+n^k\ind{t_n\geq n^{1-\gamma'+\delta_1}}
    \end{align*}
for any $\delta_1>0$, where for the second term above we use $Y_\alpha(v) \le n$. Since $\alpha_*(\cdot)$ is monotone increasing and $P(\cdot,\cdot)$ is monotone decreasing in the second argument, the first term above is at most
\begin{align*}
    P(k,\delta)n^{k(1-\gamma'+\delta_1)\left[2\delta+1-\frac{1-\alpha}{1-\gamma'+\delta_1}\right]}=O(n^{k(\alpha-\gamma'+\varepsilon)})
\end{align*}
as $n \to \infty$ by choosing $\delta,\delta_1$ sufficiently small depending on $\varepsilon$.
Further, since the bound in the last display is independent of $t_n$, multiplying by $\Prob{|\cT_n(v; E_{\floor{n^{\gamma'}}})|=t_n}$ and taking a sum over $t_n$, 
\begin{align*}
    \Exp{Y_\alpha(v)^k}\leq O(n^{k(\alpha-\gamma'+\varepsilon)})+n^k\Prob{\cT_n(v;E_{\floor{n^{\gamma'}}})|\geq n^{1-\gamma'+\delta_1}}.
\end{align*}
Let $M \ge \max (1, k/\delta_1)$. Then there exists a constant $Q(M)$ such that the second term above is at most
$Q(M)n^{k}n^{-M\delta_1}=o(1)$
   as $n \to \infty$.
   This can be seen by 
using Proposition \ref{prop:subtree_mom_UB} together with Markov's inequality. The result follows.
\end{proof}

\begin{proof}[Proof of Proposition \ref{prop:conc_cond_mean}]
    Since $M\in \N$ is finite, by a union bound, it suffices to show that
    \begin{align*}
        \Prob{\left|C^1_{j;M}-C^{\gamma'}_{j;M}\CExp{Y_\alpha(1)}{Z}\right|>\varepsilon\left( C^{1,\gamma'}_{j;M}\CExp{Y_\alpha(1)}{Z}\right)}=o(1).\numberthis \label{eq:conc_TS}
    \end{align*}
Note that by Chebyshev's inequality and \eqref{eq:key decomp},
\begin{align*}
    &\CProb{\left|C^1_{j;M}-C^{\gamma'}_{j;M}\CExp{Y_\alpha(1)}{Z}\right|>\varepsilon\left( C^{\gamma'}_{j;M}\CExp{Y_\alpha(1)}{Z}\right)}{Z,\mathscr{F}_{\floor{n^{\gamma'}}}}\\&\leq \frac{\CExp{\left|C^1_{j;M}-|C^{\gamma'}_{j;M}|\CExp{Y_\alpha(1)}{Z}\right|^2}{Z,\mathscr{F}_{\floor{n^{\gamma'}}}}}{\varepsilon^2(C^{\gamma'}_{j;M})^2\CExp{Y_\alpha(1)}{Z}^2}\\&=\frac{\CExp{\left(\sum_{v \in C_{(\gamma')}(j;E_M)}\left(Y_\alpha(v)-\CExp{Y_\alpha(v)}{Z}\right) \right)^2}{Z,\mathscr{F}_{\floor{n^{\gamma'}}}}}{\varepsilon^2(C^{\gamma'}_{j;M})^2\CExp{Y_\alpha(v)}{Z}^2},\numberthis \label{eq:cond_cheby_1}
\end{align*}
where we use in the last display $\CExp{Y_\alpha(i)}{Z}=\CExp{Y_\alpha(j)}{Z}$ for any $i,j\in [\floor{n^{\gamma'}}]$ by the exchangeability of Lemma \ref{lem:neg_cor}. Furthermore, note that again by exchangeability, $\CExp{Y_\alpha(1)}{Z}=\frac{Z}{\floor{n^{\gamma'}}}$. Next, note that
\begin{align*}
    &\CExp{\left(\sum_{v \in C_{(\gamma')}(j;E_M)}\left(Y_\alpha(v)-\CExp{Y_\alpha(v)}{Z}\right) \right)^2}{\mathscr{F}_{\floor{n^{\gamma'}}},Z}\\&=C^{\gamma'}_{j;M}\CExp{\left(Y_\alpha(v)-\CExp{Y_\alpha(v)}{Z}\right)^2}{\mathscr{F}_{\floor{n^{\gamma'}}},Z}\\&\hspace{10 pt}+\binom{C^{\gamma'}_{j;M}}{2}\CExp{\left(Y_\alpha(x)-\CExp{Y_\alpha(x)}{Z}\right)\left(Y_\alpha(y)-\CExp{Y_\alpha(y)}{Z}\right)}{\mathscr{F}_{\floor{n^{\gamma'}}},Z},
\end{align*}
where $x,y,v$ are arbitrary elements of $C_{(\gamma')}(j;E_M)$. Thanks to the negative correlation as claimed by Lemma \ref{lem:neg_cor}, the second term on the right-hand side above is negative, and we get 
\begin{align*}
    &\CProb{\left|C^1_{j;M}-C^{\gamma'}_{j;M}\CExp{Y_\alpha(1)}{Z}\right|>\varepsilon \left(C^{\gamma'}_{j;M}\CExp{Y_\alpha(1)}{Z}\right)}{Z,\mathscr{F}_{\floor{n^{\gamma'}}}} \\
    &\qquad\leq \frac{\CExp{(Y_\alpha(v))^2}{\mathscr{F}_{\floor{n^{\gamma'}}},Z}}{\varepsilon^2 C^{\gamma'}_{j;M}(Z/\floor{n^{\gamma'}})^2}.\numberthis \label{eq:cond_Cheby_2}
\end{align*}
Unconditioning, we obtain 
\begin{align*}
&\Prob{\left|C^1_{j;M}-C^{\gamma'}_{j;M}\CExp{Y_\alpha(1)}{Z}\right|>\varepsilon \left(C^{\gamma'}_{j;M}\CExp{Y_\alpha(1)}{Z}\right)}\\&\leq \Exp{\CProb{\left|C^1_{j;M}-C^{\gamma'}_{j;M}\CExp{Y_\alpha(1)}{Z}\right|>\varepsilon\left( C^{\gamma'}_{j;M}\CExp{Y_\alpha(1)}{Z}\right)}{\mathscr{F}_{\floor{n^{\gamma'}}},Z}\ind{Z\geq n^{\alpha-\delta}, C^{\gamma'}_{j;M}\geq n^{\gamma'-\delta}}}\\&\hspace{10 pt}+\Prob{Z< n^{\alpha-\delta}}+\Prob{ C^{\gamma'}_{j;M}< n^{\gamma'-\delta}}\\& \leq \frac{\Exp{(Y_\alpha(v))^2}n^{2\gamma'}}{\varepsilon^2 n^{\gamma'-\delta}n^{2(\alpha-\delta)}}+\Prob{Z< n^{\alpha-\delta}}+\Prob{ C^{\gamma'}_{j;M}< n^{\gamma'-\delta}},\numberthis \label{eq:cond_conc_ineq_UB}
\end{align*}
for any $0<\delta<\gamma'$. We continue by showing that the second and third terms above are $o(1)$. Note that since $j$ is fixed, $(\cT_{\floor{n^{\gamma'}}}(1;[M])/\floor{n^{\gamma'}},\dots,\cT_{\floor{n^{\gamma'}}}(M;[M])/\floor{n^{\gamma'}})$ converge in law to a Dirichlet random vector with parameters $(1,1,\dots,1)$. Using this  with Lemma \ref{lem:all upto gamma} gives that $C^{\gamma'}_{j; M}/\floor{n^{\gamma'}}$ converges in distribution to the $j$-th component of a Dirichlet. Thus, the third term above is $o(1)$. Further, by the lower bound of Proposition \ref{prop:lower_bound}, note that with high probability $|\bC_\alpha(1)|=\floor{n^{\gamma'}}+\sum_{i=1}^{\floor{n^{\gamma'}}}Y_\alpha(i)\geq n^{\alpha'}$ for any $\alpha'<\alpha$, so that $Z=\sum_{i=1}^{\floor{n^{\gamma'}}}Y_\alpha(i)\geq n^{\alpha-\delta}$, by choosing $\alpha'$ sufficiently close to $\alpha$ and since $\gamma'<\gamma<\alpha$. Thus, the second term in \eqref{eq:cond_conc_ineq_UB} is also $o(1)$.

Taking $k=2$ in Lemma \ref{lem:sec_mom_Yv} gives $\Exp{(Y_\alpha(v))^2}\leq n^{2(\alpha-\gamma'+2\delta)}$ for all large $n$. Thus, the first term in \eqref{eq:cond_conc_ineq_UB} is at most 
$n^{7\delta-\gamma'}/\varepsilon^2=o(1)$
whenever $\delta<\gamma'/7$. So, the right-hand side of \eqref{eq:cond_conc_ineq_UB} is $o(1)$ as $n \to \infty$, proving \eqref{eq:conc_TS}, which concludes the proof.     
\end{proof}

\medskip
\begin{proof}[Proof of Theorem \ref{thm:root_finding_c1}]
  Let $T=\bC_\alpha(1)$ and recall the set $H(M,T)$ from \eqref{eq:set H}. We show that we can choose $M$ sufficiently large such that $\Prob{1 \notin H(M, T)}\leq \delta$ for all large $n$. To this end, observe that 
    \begin{align*}
        \Prob{1 \notin H(M,T)}\leq \Prob{\min_{i: i>M}\psi_{T}(i)<\psi_{T}(1)}.
    \end{align*}
Let us write $\psi_{T}(i)=\psi(i)$ with $T=\bC_\alpha(1)$ for the rest of this proof, for any vertex $i$. It is easy to see that $\psi(i)\geq \min_{1\leq k \leq M}\sum_{j\in [M]\setminus\{k\}}C^1_{j;M}$, and that $C^1_{1;M}=\psi(1)$. Thus, the last probability is at most
\begin{align*}
    \Prob{\min_{1\leq k\leq M}\sum_{j\in [M]\setminus\{k\}}C^1_{j;M}\leq C^1_{1;M}}\leq \Prob{C^1_{2;M}+\dots+C^1_{M;M}\leq C^1_{1;M}}.\numberthis \label{eq:root_finding_1}
\end{align*}
    For any $\varepsilon>0$, defining the \emph{good event} 
    \begin{align*}
        \cE(\varepsilon)=\bigcap_{j\in [M]}\left\{\left|C^1_{j;M}-C^{\gamma'}_{j;M}\CExp{Y_\alpha(1)}{Z}\right|\leq \varepsilon\left( C^{\gamma'}_{j;M}\CExp{Y_\alpha(1)}{Z}\right)\right\},
    \end{align*}
    by Proposition \ref{prop:conc_cond_mean},
    \begin{align*}
        \Prob{1\notin H(M,\bC_\alpha(1))}\leq \Prob{C^1_{2;M}+\dots+C^1_{M;M}\leq C^1_{1;M},\cE(\varepsilon)}+o(1).
    \end{align*}
    On $\cE(\varepsilon)$, we have the upper bound
    \begin{align*}
    \Prob{C^1_{2;M}+\dots+C^1_{M;M}\leq C^1_{1;M},\cE(\varepsilon)}\leq \Prob{C^{\gamma'}_{2;M}+\dots+C^{\gamma'}_{M;M}\leq \frac{1+\varepsilon}{1-\varepsilon}C^{\gamma'}_{1;M}}.
    \end{align*}
However, due to Lemma \ref{lem:all upto gamma}, the random vector $\floor{n^{\gamma'}}^{-1}(C^{\gamma'}_{1;M},\dots,C^{\gamma'}_{M;M})$ is asymptotically close to the vector $\floor{n^{\gamma'}}^{-1}(|\cT_{n^{\gamma'}}(1;E_M)|,\dots,|\cT_{n^{\gamma'}}(M;E_M)|)$, and the latter as $n \to \infty$ approaches a Dirichlet random vector $(Z_1,\dots,Z_M)$ with parameters $(1,1,\dots,1)$. In particular, one can choose $M=M(\eps,\delta)$ large enough such that for all large $n$
\begin{align*} \Prob{C^{\gamma'}_{2;M}+\dots+C^{\gamma'}_{M;M}\leq \frac{1+\varepsilon}{1-\varepsilon}C^{\gamma'}_{1;M}}\leq\Prob{Z_2+\dots+Z_M<\frac{1+\varepsilon}{1-\varepsilon}Z_1}+\delta/2\leq \delta,
\end{align*}
proving the theorem. \end{proof}

\subsection{Proof of root finding in $\alpha$-forests}
Combining all the results from the previous subsections, the proof of Theorem \ref{thm:forest} is now straightforward.
\begin{proof}[Proof of Theorem \ref{thm:forest}]
    Observing the forest $\cF_n(\alpha)$, we look at its tree components and distinguish the component $\bC_\alpha(1)$ of $1$ in it by querying which components have size at least $n^{\alpha-\gamma/4}$. This procedure distinguishes $\bC_\alpha(1)$ from the other components thanks to Lemma \ref{lem:disting_root_comp}. Having the component $\bC_\alpha(1)$ at hand, given $\varepsilon>0$, we apply Theorem \ref{thm:root_finding_c1} to construct a finite confidence set $K(\eps)$ of vertices in $\bC_\alpha(1)$ that contain $1$ with probability at least $1-\varepsilon$.
\end{proof}

\section{Root finding under Erd\H{o}s-Rényi noise}\label{sec:root_find_ER}

In this section, relying on Theorem \ref{thm:forest}, we prove Theorem \ref{thm:ER_noise}. 

\subsection{Filtering}

The first step in the proof is to show that, with high probability, the subgraph of $\cT_n(G)=\cT_n\cup G$ spanned by vertices of sufficiently high degree does dot contain edges coming from the Erd\H{o}s-Rényi graph $G$, and it is a forest.

\begin{proposition}[No noise edges between high degree vertices]\label{prop:no_edge_high_deg}
 If $\alpha<1-\frac{1}{2\log 2}$, the subgraph of $G=\G(n,\lambda_n/n)$ spanned by the set $\{v\in [n]:d_{\cT_n(G)}>(1-\alpha)\log n\}$ is, with high probability, an independent set, whenever $\lambda_n=o(\log n)$.
\end{proposition}

The proof is based on the following lemma, whose straightforward proof is omitted.
%\textcolor{red}{G: Removed the proof of the lemma as it is trivial.} \textcolor{blue}{[N: fine by me]}

\begin{lemma}\label{lem:no_ER_edges}
    Let $V\subset [n]$ satisfy $|V|\leq n^{1/2-\eps}$ for some $\eps>0$. Then, with high probability, the subgraph of $G=\G(n,\lambda/n)$ spanned by $V$ is an independent set, whenever $\lambda_n=o(\log n)$. \end{lemma}

\begin{proof}[Proof of Proposition \ref{prop:no_edge_high_deg}]
    By Lemma \ref{lem:no_ER_edges}, since $\cT_n$ and $G=\G(n,\lambda/n)$ are independent, it suffices to check that, for $\alpha<1-\frac{1}{2\log 2}$, for some $\eps>0$,
    \begin{align*}
    \Prob{|\{v\in [n]:d_{\cT_n}(v)>(1-\alpha)\log n\}|>n^{1/2-\eps}}=o(1). \numberthis \label{eq:at_most_sqrtn_highdeg}
    \end{align*}
To check this, note that by a union bound and Lemma \ref{lem:degUandLtail}, for any $\eps>0$,
\begin{align*}
&\Exp{|\{v\in [n]:d_{\cT_n}(v)>(1-\alpha)\log n\}|}\\
&\leq n^{\alpha+\eps/4}+\sum_{v\geq n^{\alpha+\eps/4}}\Prob{d_{\cT_n}(v)>(1-\alpha)\log n}\\&\leq n^{\alpha+\eps/4}+\sum_{v \geq n^{\alpha+\eps/4}}n^{f_\alpha(x(v))},
%\leq n^{\alpha+\eps/4}+\log n \int_{\alpha+\eps/4}^1 n^{v +f_\alpha(v)}dv\\
         %&\leq n^{\alpha+\eps/4}+\log n \max_{v \in [0,1]}n^{v+f_\alpha(v)}.
     \end{align*}
     where, recalling $x(v)=\log v/ \log n$ from \eqref{def:x(v)}, in the last display we use that $v\geq n^{\alpha+\eps/4}$ implies $x(v)>\alpha$. By a Riemann integral bound,
     \begin{align*}
         \sum_{v \geq n^{\alpha+\eps/4}}n^{f_\alpha(\log v/\log n)}\leq \int_{n^{\alpha+\eps/4}}^n n^{f_\alpha(\log u/\log n)}du=\log n\int_{\alpha+\eps/4}^1 n^{y+f_\alpha(y)}dx,
     \end{align*}
     where in the last step above we change variables $u=n^y$.
Recall from Lemma~\ref{lem:gamma_cont} that the maximal value of $x+f_\alpha (x)$ on $[0,1)$ is $1-(1-\alpha)\log 2$.
This is strictly less than $1/2-2\eps$ when $\alpha<1-\frac{1}{2\log 2}$ by choosing $\eps>0$ small. On the other hand, since $\log 2<1$, $\alpha<1-\frac{1}{2\log 2}$ ensures $n^{\alpha+\eps/4}<n^{1/2-2\eps}$ by choosing $\eps>0$ small. Thus, for any $\alpha<1-\frac{1}{2\log 2}$ and small enough $\eps > 0$, we have
\begin{align}
\label{rem:bound_expec_high_deg}
\Exp{|\{v\in [n]:d_{\cT_n}(u)>(1-\alpha)\log n\}|}\leq n^{1/2-2\eps}    \end{align}
so that \eqref{eq:at_most_sqrtn_highdeg} follows from Markov's inequality.
\end{proof}

\begin{definition}[The forest $H_n(\alpha)$]\label{def:Hnalpha}
 Denote $H=\cT_n(G)$  with $G\stackrel{d}{=} \G(n,\lambda/n)$. Construct the graph $H_n(\alpha)$ from $H$ by only keeping edges $e=\{u,v\}\in E(H)$ satisfying
$d_{H}(u)\wedge d_H(v)>(1-\alpha)\log n$.
%Thus, if $G$ was the empty graph with no edges, $H_n(\alpha)$ would then coincide with $\cF_n(\alpha)$ from Definition \ref{def:alpha_high_deg_for}. 
\end{definition}

\begin{lemma}\label{lem:sandwiching_forests}
    For any $\delta>0$, we have,
    \begin{align*}
       \lim_{n \to \infty} \Prob{E(\cF_n(\alpha))\subseteq E(H_n(\alpha))\subseteq E(\cF_n(\alpha+\delta))}=1.
    \end{align*}
\end{lemma}
\begin{proof}
    The lemma follows if we show that any vertex $v\in [n]$ satisfying $d_H(v)>(1-\alpha)\log n$ must also satisfy $d_{\cT_n}(v)>(1-\alpha-\delta)\log n$ for any $\delta>0$. This follows from the fact that $\max_{v\in [n]}d_G(v)\leq \delta\log n$ for any $\delta>0$ whenever $\lambda_n=o(\log n)$ (e.g., see \cite[Theorem 3.2]{bollobas2011random}) and the observation that $d_H(v)\leq d_{\cT_n}(v)+d_G(v)$ for any $v\in [n]$.
\end{proof}
%\red{[ N: For the above we just need max degree $o(\log n)$ ]}

As a consequence, since $\cF_n(\alpha+\delta)$ is a forest, so is $H_n(\alpha)$. In fact, the sandwiching above lets us translate many properties of the forest $\cF_n(\alpha)$ to $H_n(\alpha)$,  which we gather in the next corollary.

\begin{corollary}[Properties of $H_n(\alpha)$]\label{cor:prop_Hforest}
    Consider the forest $H_n(\alpha)$ and let $\bC_{H,\alpha}(1)$ be the component containing $1$. The following properties hold with probability tending to $1$.
    \begin{itemize}
    \item[(i)] For any $\delta>0$, $\bC_\alpha(1)\subseteq \bC_{H,\alpha}(1)\subseteq \bC_{\alpha+\delta}(1)$.
    \item[(ii)] For any $\alpha_1<\alpha<\alpha_2$,
    \begin{align*}
        n^{\alpha_1}\leq |\bC_{H,\alpha}(1)|\leq n^{\alpha_2}.
    \end{align*}
    \item[(iii)] For any $\eps \in (0,\gamma(\alpha))$, $\cT_{\lfloor n^{ \gamma(\alpha)-\eps}\rfloor}\subseteq \bC_{H,\alpha}(1)$.
        \item[(iv)] $\bC_{H,\alpha}(1)$ is the largest component of $H_n(\alpha)$, and the second largest component $C_2(H,\alpha)$ of $H_n(\alpha)$ satisfies $|C_2(H,\alpha)|\ll n^{\alpha-\gamma(\alpha)/4}$. Consequently, the root component $\bC_{H,\alpha}(1)$ 
        is the unique connected component of $H_n(\alpha)$ of size at least $n^{\alpha-\gamma(\alpha)/4}$.
    \end{itemize}
\end{corollary}

%\begin{proof}
%Part (i) follows from Lemma \ref{lem:sandwiching_forests}. Part (ii) follows from Propositions \ref{prop:lower_bound} and \ref{prop:root_comp_UB} together with Lemma \ref{lem:sandwiching_forests} by choosing $\delta$ sufficiently small. Part (iii) follows from part (i) and Lemma \ref{lem:all upto gamma}. Part (iv) follows from Lemma \ref{lem:disting_root_comp} and part (i), by choosing $\delta>0$ sufficiently small, using the continuity of $\gamma(\alpha)$ as a function of $\alpha$ (recall Lemma \ref{lem:gamma_cont}). 
%\end{proof}

\begin{remark}[Using the continuity of $\gamma$]
    Throughout this section, when we write $\gamma$, we mean $\gamma(\alpha)$ for a fixed choice of $\alpha<1-\frac{1}{2\log 2}$. Note that because of Corollary \ref{cor:prop_Hforest} part (i), sometimes in our bounds $\gamma$ appears with the argument $\alpha+\delta$ instead of $\alpha$. By the continuity of $\gamma(\cdot)$ due to Lemma \ref{lem:gamma_cont}, we can always replace $\gamma(\alpha+\delta)$ by $\gamma(\alpha)$ up to a small additive constant, which does not create any problems for our arguments. We carry this out throughout this section, without mentioning it at every instance. 
\end{remark}

Armed with Corollary \ref{cor:prop_Hforest}(iv), observing $H$ without vertex labels, the statistician can first form the graph $H_n(\alpha)$, whose largest component is $\bC_{H,\alpha}(1)$. The remaining task is to show that the root vertex in among the most central vertices in $\bC_{H,\alpha}(1)$, according to Jordan centrality.

\paragraph{Key ideas.}
Because of the complicated structure of $H_n(\alpha)$, direct analysis of the Jordan centralities of the different vertices in $\bC_{H,\alpha}(1)$ is quite difficult. To this end, via coupling and sampling techniques, we argue that with high probability, $H_n(\alpha)$ equals another random graph $\kH_n^{(3)}(\alpha,\delta,\varepsilon)$, and centralities in the root component of the latter are easier to tackle. To arrive at $\kH_n^{(3)}(\alpha,\delta,\varepsilon)$ from $H_n(\alpha)$, we proceed via a sequence of intermediate steps.

\textbf{Step 1.} First, define $\kH_n(\alpha)$ to be the subgraph of $H$ spanned by those vertices $v\in [n]$ that satisfy $d_H(v)>(1-\alpha)\log n$. This is precisely the analogue of $\cF_n^{(\rm V)}(\alpha)$ from Remark \ref{rem:site_perco_forest}. Note that if we denote by $\bC_{\kH,\alpha}(1)$ the component containing $1$ in the graph $\kH_n(\alpha)$ (with the convention that $\bC_{\kH,\alpha}(1)$ is an empty graph if $d_H(1)\leq (1-\alpha)\log n$), then, with high probability, $\bC_{\kH,\alpha}(1)=\bC_{H,\alpha}(1)$.
Indeed, equality can only fail when $d_H(1)\leq (1-\alpha)\log n$, an event of vanishing probability by Lemma \ref{lem:degUandLtail}. Note also that (analogous to the discussion in Remark \ref{rem:site_perco_forest}) $H_n(\alpha)$ can be obtained from $\kH_n(\alpha)$ by including all the vertices $v\in [n]\setminus V(\kH_n(\alpha))$ as \emph{isolated nodes}. Thus, we have reduced our task to analyzing Jordan centrality on $\bC_{\kH,\alpha}(1)$. 

\textbf{Step 2.} The argument used in the proof of Lemma \ref{lem:sandwiching_forests} shows that, with high probability, every vertex of $\kH_n(\alpha)$
satisfies $d_{\cT_n}(v)>(1-\alpha-\delta)\log n$.
This forms the basis of the second reduction step. In particular, it implies that with high probability, the relevant vertices that ever participate in constructing $\kH_n(\alpha)$ are the ones in $D_{\cT_n}(\alpha+\delta)$ for any choice of $\delta>0$. Denoting by $\kH'_n(\alpha,\delta)$ the subgraph of $H$ spanned by the vertices $v\in D_{\cT_n}(\alpha+\delta)$ satisfying $d_H(v)>(1-\alpha)\log n$, we have the following corollary, whose proof we omit.
\begin{corollary}\label{cor:reduc_h'}
    For any $\delta>0$ we have that, with high probability,
$\kH'_n(\alpha,\delta)=\kH_n(\alpha)$
and 
$\bC_{\alpha,\delta}(1)=\bC_{H,\alpha}(1)$,
    where $\bC_{\alpha,\delta}(1)$  denotes the component of $1$ in $\kH'_n(\alpha,\delta)$.
\end{corollary} 

\textbf{Step 3.} Thanks to the last corollary, we now have reduced our task to analyzing Jordan centralities in $\bC_{\alpha,\delta}(1)$. To this end, we need to understand the structural aspects of the graph $\kH_n'(\alpha,\delta)$. This is the crucial step (Proposition \ref{prop:coup} and its proof below), where we sequentially reduce the graph $\kH_n'(\alpha,\delta)$ to the graph $\kH_n^{(3)}(\alpha,\delta,\varepsilon)$, via the following steps.
\begin{itemize}
    \item[(a.)] First, Lemma \ref{lem:ER_nbhds_large} below shows that the collection of neighborhoods $N_G(v)$ of the vertices $v$ in $D_{\cT_n}(\alpha+\delta)$ form an independent collection of random binomial subsets of $[n]$, where each element is retained with probability $\lambda/n$. Thus, if we assign to each vertex $v\in D_{\cT_n}(\alpha+\delta)$ such a random binomial subset $S_v$, and form the subgraph $\kH_n^{(1)}(\alpha,\delta)$ of $H$ by retaining vertices from $D_{\cT_n}(\alpha+\delta)$ satisfying $|S_v\cup N_{\cT_n}(v)|>(1-\alpha)\log n$, then $\kH'_n(\alpha,\delta)$ equals $\kH_n^{(1)}(\alpha,\delta)$ with high probability, and our task boils down to understanding Jordan centrality in the root component $\bC^{(1)}_{\alpha,\delta}(1)$ of $\kH^{(1)}_n(\alpha,\delta)$.
    \item[(b.)] Second, for any $\varepsilon \in (0,\gamma)$, with high probability, for all $v\in [\floor{n^{\gamma-\eps}}]$, by Lemma \ref{lem:chernoff_deg_tail} we have $d_{\cT_n}(v)=|N_{\cT_n}(v)|>(1-\alpha)\log n$, so that $S_v$ does not play any role in including these vertices in $\kH^{(1)}_n(\alpha,\delta)$. In other words, if we form the subgraph $\kH^{(2)}_n(\alpha,\delta,\varepsilon)$ of $H$ spanned by the vertices in $[\floor{n^{\gamma-\eps}}]$, together with those in $D_{\cT_n}(\alpha+\delta)\cap\{\floor{n^{\gamma-\eps}},\floor{n^{\gamma-\eps}}+1,\dots,n\}$ satisfying $|N_{\cT_n}(v)\cup S_v|>(1-\alpha)\log n$, then by (a.), $\kH^{(2)}_n(\alpha,\delta,\varepsilon)$ equals $\kH'_n(\alpha,\delta)$ with high probability, for any $\varepsilon\in (0,\gamma)$. Thus, we have reduced our task to applying Jordan centrality on $\bC^{(2)}_{\alpha,\delta,\varepsilon}(1)$, the root component of $\kH^{(2)}_n(\alpha,\delta,\varepsilon)$.
    \item[(c.)] Finally, recall that, by Lemma \ref{lem:sandwiching_forests}, with high probability, $H_n(\alpha)$ is a sub-forest of $\cT_n$. Hence, so is $\kH^{(2)}_n(\alpha,\delta,\varepsilon)$. Additionally, for any $v\in D_{\cT_n}(\alpha+\delta)$ with $v\geq \floor{n^{\gamma-\varepsilon}}+1$, one can write $|N_G(v)\cup S_v|=d_{\cT_n}(v)+Z_v$, where $Z_v\sim \textrm{Bin}\left(n-d_{\cT_n}(v),\lambda/n \right)$, and these latter variables are independent across such $v$'s. Thus, if we form the sub-forest $\kH^{(3)}_n(\alpha,\delta,\eps)$ of $\cT_n$ spanned by $[\floor{n^{\gamma-\eps}}]$ together with those $v\in D_{\cT_n}(\alpha+\delta)$ satisfying $d_{\cT_n}(v)+Z_v>(1-\alpha)\log n$, we have $\kH^{(3)}_n(\alpha,\delta,\eps)=\kH'_n(\alpha,\delta)$ with high probability. Thus, it suffices to understand Jordan centrality on $\bC^{(3)}_{\alpha,\delta,\varepsilon}(1)$, the root component of $\kH^{(3)}_n(\alpha,\delta,\eps)=\kH'_n(\alpha,\delta)$.   
\end{itemize}
The analysis of Jordan centrality in $\bC^{(3)}_{\alpha,\delta,\varepsilon}(1)$ is carried out in Section \ref{sec:root_find_C3}. The main idea is similar to the proof of Theorem \ref{thm:root_finding_c1}, albeit slightly complicated, because of the presence of the noise edges. Roughly speaking, we can write down a decomposition as in \eqref{eq:key decomp} with $\bC_\alpha(1)$ replaced now by $\bC_{\alpha,\delta,\eps}(1)$. To make an argument similar to the proof of Theorem \ref{thm:root_finding_c1} go through, one needs to argue that the summands appearing in the decomposition are exchangeable. This is not so straightforward as in the proof of Theorem \ref{thm:root_finding_c1}; however, it can be established using the description of the graph $\kH^{(3)}_{\alpha,\delta,\eps}(1)$. Informally, we can view the collection of these summands as a sample from a measure constructed from the variables $Z_v$ and the exchangeable collection of subtrees $\left(\cT_n(1;E_{\lfloor n^{\gamma-\eps}\rfloor }),\dots,\cT_n(\lfloor n^{\gamma-\eps}\rfloor;E_{\lfloor n^{\gamma-\eps}\rfloor })\right)$ so that the required exchangeability can be established using Proposition \ref{prop:sample_exch} (see Proposition \ref{prop:sampling_noisy_subtrees} and its proof below). The rest of the argument is similar to the proof of Theorem \ref{thm:root_finding_c1}.

\medskip

Next, we formally prove the steps we discussed above. We begin with Lemma \ref{lem:ER_nbhds_large}, that states that the neighborhoods in $G$ of the vertices in $D_{\cT_n}(\alpha+\delta)$ are i.i.d.\ binomial random subsets, used in \textbf{Step 3} (a.) above.

\begin{lemma}\label{lem:ER_nbhds_large}
Let $\alpha+\delta<1-\frac{1}{2\log 2}$, where $\alpha\in (0,1)$ and $\delta>0$. Then, there is a coupling between $\cT_n, G$ and i.i.d.\ random binomial subsets $S_1,S_2,\dots$  such that
\begin{align*}
    \Prob{N_{G}(v)=S_v\textrm{ for all }v\in D_{\cT_n}(\alpha+\delta)}\to 1.
\end{align*}
\end{lemma}
\begin{proof}
\begin{comment}
    Let $S'_v:=S_v\mid_{\{v\notin S_v\}}$. Note that
    \begin{align*}
        \Prob{\exists\;v\in D_{\cT_n}(\alpha+\delta):v\in S_v}\leq \frac{\lambda \Exp{|D_{\cT_n}(\alpha+\delta)|}}{n}=o(1)
    \end{align*}
    by Remark \ref{rem:bound_expec_high_deg}, since $\alpha+\delta<1-\frac{1}{2\log 2}$. Thus, $S'_v=S_v$ for all $v\in D_{\cT_n}(\alpha+\delta)$ with high probability, so that it suffices to couple the neighborhoods $N_G(v)$ to the sets $S'_v$. For each $v\in D_{\cT_n}(\alpha+\delta)$, note that $S'_v$ has the distribution of a random binomial subset of $[n]\setminus\{v\}$.
\end{comment}
Note that for any $A\subseteq [n]$, conditionally on the event that $A$ is an independent set in $G$, the collection of neighborhoods $\{N_G(v):v\in A\}$ has the same law as an i.i.d.\ collection of random binomial subsets of $[n]\setminus A$, where each element is retained with probability $\lambda/n$. In particular, $N_G(v)$ in this case has the same distribution as $S_v\mid_{\{S_v\cap A=\emptyset\}}$ for all $v\in A$.
Thus, conditionally on the event that $A$ is independent in $G$, we can couple such that $N_G(v)=S_v$ for all $v\in A$ with high probability, provided that 
\begin{align*}
    \Prob{\bigcup_{v: v \in A} [ S_v \cap A\neq \emptyset ] }=o(1).\numberthis \label{eq:Sv_doesnotintersect_A}
\end{align*}
To conclude the lemma, note that
 $A=D_{\cT_n}(\alpha+\delta)$ is independent by Lemma \ref{lem:no_ER_edges}.
On the other hand, \eqref{eq:Sv_doesnotintersect_A} holds for $A=D_{\cT_n}(\alpha+\delta)$. To see this, note that since $\cT_n$ is independent of $G$, by Remark \ref{rem:bound_expec_high_deg} it in fact suffices to check that \eqref{eq:Sv_doesnotintersect_A} holds for any $A\subseteq [n]$ with $|A|\leq n^{1/2-\delta}$ for $\delta>0$. For any such $A$, a union bound implies that
\begin{align*}
\Prob{\bigcup_{v: v \in A} [ S'_v \cap A\neq \emptyset ] }
\leq |A| \Prob{{\rm Bin}(|A|,\lambda/n)>0}\leq |A|^2\frac{\lambda}{n}=o(1).
\end{align*}
%since $\delta>0$.
This completes the proof.
\begin{comment}
Thus, the result follows if we can claim that $D_{\cT_n}(\alpha+\delta)$ forms an independent set in $G$, something 

We claim that 
\begin{align*}
    \{D_{\cT_n}(\alpha+\delta)\textrm{ is an independent set in }G \}\cup\{N_G(u)\cap N_G(v)=\emptyset\;\forall\;u,v\in D_{\cT_n}(\alpha+\delta)\}
\end{align*}
holds with high probability.
\end{comment}    
\end{proof}

%To state the key resampling result of this section, we begin with a definition.

Next, we define the graph $\kH_n^{(3)}(\alpha,\delta,\eps)$ and a related graph $H_n^{(3)}(\alpha,\delta,\eps)$.
\begin{definition}[The forests $\kH^{(3)}_n(\alpha,\delta,\eps)$ and $H_n^{(3)}(\alpha,\delta,\eps)$]\label{def:h3}
    Given $\cT_n$, let $(Z_v)_{v\in [n]}$ form a collection of conditionally independent random variables, with
\begin{align*}
    Z_v\stackrel{d}{=} \textrm{Bin}\left(n-d_{\cT_n}(v),\frac{\lambda}{n}\right).
\end{align*}
Recall $\gamma$ from \eqref{eq:def_gamma}. For any $\alpha\in (0,1),\delta>0$ and $\eps\in (0,\gamma)$, define $\kH_n^{(3)}(\alpha,\delta,\eps)$ to be the subforest of $\cT_n$ spanned by the vertices in %\red{[ $\eps$ dependence below? Better to define with a $\gamma'$ parameter where $\gamma'<\gamma$? ]}
\begin{align*}
\{1,2,\dots,\floor{n^{\gamma-\eps}}\}\cup\{v\in \{\floor{n^{\gamma-\eps}},\floor{n^{\gamma-\eps}}+1,\dots,n\}\cap D_{\cT_n}(\alpha+\delta):d_{\cT_n}(v)+Z_v>(1-\alpha)\log n\}.    
\end{align*}
Define further $H_n^{(3)}(\alpha,\delta,\eps)$ to be the subforest obtained from  $\kH_n^{(3)}(\alpha,\delta,\eps)$ by including in it all vertices in $[n]\setminus V(\kH_n^{(3)}(\alpha,\delta))$ as isolated vertices.
\end{definition}
Let us finally prove the proposition that reduces $\kH_n(\alpha)$ to $\kH_n^{(3)}(\alpha,\delta,\eps)$. 
\begin{proposition}\label{prop:coup}
    There exists a coupling of $\cT_n$, $G$ and the variables $Z_v$ as in Definition \ref{def:h3} such that whenever $\alpha+\delta<1-\frac{1}{2 \log 2}$, for any $\eps\in (0,\gamma)$, with high probability,
$\kH^{(3)}_n(\alpha,\delta,\eps)=\kH_n(\alpha)$,
$\bC^{(3)}_{\alpha,\delta,\eps}(1)=\bC_{H,\alpha}(1)$, and $H_n^{(3)}(\alpha,\delta,\eps)=H_n(\alpha)$,
where $\bC^{(3)}_{\alpha,\delta,\eps}(1)$ denotes the component of $1$ in $\kH^{(3)}_n(\alpha,\delta,\eps)$.
\end{proposition}

%\begin{remark}
    %The superscript $(3)$ appearing in the notations $\kH^{(3)}_n(\alpha,\delta,\eps)$ and $H_n^{(3)}(\alpha,\delta,\eps)$ has a special meaning; referring to the proof of Proposition \ref{prop:coup} below, we are able to reach $\kH^{(3)}_n(\alpha,\delta,\eps)$ from $\kH_n(\alpha)$ via coupling techniques through \emph{two} intermediate steps. To denote them, we use the superscripts $(1)$ and $(2)$, respectively.
%\end{remark}

\begin{proof}[Proof of Proposition \ref{prop:coup}]
As discussed above, motivated by Corollary \ref{cor:reduc_h'} and Lemma \ref{lem:ER_nbhds_large}, we construct the graph $\kH^{(1)}_n(\alpha,\delta)$ as follows: 
\begin{itemize}
    \item [(i)] Take $H$, assign independent random binomial subsets $S_v$ to all $v\in D_{\cT_n}(\alpha+\delta)$.
    \item [(ii)] Define $\kH^{(1)}_n(\alpha,\delta)$ to be the subgraph of $H$ spanned by $v\in D_{\cT_n}(\alpha+\delta)$ with $|N_{\cT_n}(v)\cup S_v|>(1-\alpha)\log n$.
\end{itemize}
Note that for any $v\in [n]$, $d_H(v)=|N_{\cT_n}(v)\cup N_{G}(v)|$. In particular, Corollary \ref{cor:reduc_h'} and Lemma \ref{lem:ER_nbhds_large} implies that for any $\alpha\in (0,1),\delta>0$ with $\alpha+\delta<1-\frac{1}{2\log 2}$ we have
$\kH^{(1)}_n(\alpha,\delta)=\kH_n(\alpha)$ and $\bC^{(1)}_{\alpha,\delta}(1)=\bC_{H,\alpha}(1)$, with high probability,
where $\bC^{(1)}_{\alpha,\delta}(1)$ denotes the component of $1$ in $\kH^{(1)}_n(\alpha,\delta)$.

 Since by Lemma \ref{lem:chernoff_deg_tail}, for any $\eps\in (0,\gamma)$, all  vertices $v\in [\lfloor n^{ \gamma-\eps}\rfloor]$ satisfy $d_{\cT_n}(v)>(1-\alpha)\log n$ with probability tending to $1$,
 their inclusion in the graph $\kH^{(1)}_n(\alpha,\delta)$ is automatic, without even the contribution from the set $S_v$. We conclude that if we assign to vertices $v\in \{\lfloor n^{\gamma-\eps}\rfloor,\lfloor n^{\gamma-\eps}\rfloor+1,\dots,n\}\cap D_{\cT_n}(\alpha+\delta)$ the sets $S_v$, and declare $\kH^{(2)}_n(\alpha,\delta,\eps)$ to be the subgraph of $H$ spanned by the vertices
\begin{align*}
\{1,2,&\dots,\lfloor n^{\gamma-\eps}\rfloor\}\\
&\cup  
\big\{ 
\{\lfloor n^{\gamma-\eps}\rfloor,\lfloor n^{\gamma-\eps}\rfloor+1,\dots,n\}\cap D_{\cT_n}(\alpha+\delta)
\cap\{v: |N_{\cT_n}(v)\cup S_v|>(1-\alpha)\log n\} \big\},    \end{align*}
then by the previous reduction step to $\kH^{(1)}_n(\alpha,\delta)$,
with high probability, 
$\kH^{(2)}_n(\alpha,\delta,\eps)=\kH_n(\alpha)$, and
$\bC^{(2)}_{\alpha,\delta,\eps}(1)=\bC_{H,\alpha}(1)$,
where $\bC^{(2)}_{\alpha,\delta,\eps}(1)$ denotes the component of $1$ in $\kH^{(2)}_n(\alpha,\delta,\eps)$.

Recall that, by Lemma \ref{lem:sandwiching_forests}, with high probability, $H_n(\alpha)$ is a subgraph of $\cT_n$ and it is obtained from $\kH_n(\alpha)$ by including some isolated vertices. Further, 
\begin{align*}
    |N_{\cT_n}(v)\cup S_v|=d_{\cT_n}(v)+Z_v,\numberthis \label{eq:coup_binomial_sets_and_vars}
\end{align*}
where given $\cT_n$, $(Z_v)_{v\in [n]}$ form a collection of conditionally independent random variables, with
$Z_v\stackrel{d}{=} \textrm{Bin}(n-d_{\cT_n}(v),\lambda/n)$.

Indeed, given $\cT_n$, to reveal the elements in $N_{\cT_n}(v)\cup S_v$ that are already not in $N_{\cT_n}(v)$, we only need to take a random binomial subset of $[n]\setminus N_{\cT_n}(v)$ where each element is retained with probability $\lambda/n$.
In particular, by extending the probability space if necessary, we can couple $S_v$ with $Z_v$ as appearing in Definition \ref{def:h3} such that \eqref{eq:coup_binomial_sets_and_vars} holds.
In this extended probability space, with high probability, 
we have $\kH^{(3)}_n(\alpha,\delta,\eps)=\kH^{(2)}_n(\alpha,\delta,\eps)$
and $\bC^{(3)}_{\alpha,\delta,\eps}(1)=\bC^{(2)}_{\alpha,\delta,\eps}(1)$,
where $\bC^{(3)}_{\alpha,\delta,\eps}(1)$ denotes the component of $1$ in $\kH^{(3)}_n(\alpha,\delta,\eps)$. We conclude that, with high probability,
$\kH^{(3)}_n(\alpha,\delta,\eps)=\kH_n(\alpha)$ and $\bC^{(3)}_{\alpha,\delta,\eps}(1)=\bC_{H,\alpha}(1)$.
Furthermore, $\Prob{H_n^{(3)}(\alpha,\delta,\eps)=H_n(\alpha)}\to 1$ since $H_n(\alpha)$ is constructed from $\kH_n(\alpha)$ by including the vertices in $[n]\setminus V(\kH_n(\alpha))$ as isolated nodes.
\end{proof}

 In particular, under the coupling of Proposition \ref{prop:coup}, since $\bC^{(3)}_{\alpha,\delta,\eps}(1)$ is asymptotically the same as $\bC_{H,\alpha}(1)$, the conclusions of Lemma \ref{cor:prop_Hforest} hold for it. Thus, it suffices to study root finding in $\bC^{(3)}_{\alpha,\delta,\eps}(1)$, discussed in the next section.

\subsection{Root finding in $\bC^{(3)}_{\alpha,\delta,\eps}(1)$}\label{sec:root_find_C3}

To analyze Jordan centrality of the root in $\bC^{(3)}_{\alpha,\delta,\eps}(1)$, we use a sampling argument. For any $i\in [\lfloor n^{ \gamma-\eps}\rfloor]$, denote
\begin{align*}
    t_n(i):=\kH^{(3)}_n(\alpha,\delta,\eps)(i;E_{\lfloor n^{\gamma-\eps}\rfloor}),
\end{align*}
where recall that $E_{\ell}:=\{\{u,v\}\in E(\cT_n):u,v\leq \ell\}$. We view $t_n(i)$ as a tree rooted at $i$. Since $\cT_{\lfloor n^{\gamma-\eps}\rfloor}$ is a subtree of $\bC^{(3)}_{\alpha,\delta,\eps}(1)$ with probability tending to $1$ (by Corollary \ref{cor:prop_Hforest}), the latter can be constructed as follows: given $\cT_{\lfloor n^{\gamma-\eps}\rfloor}$, $\bC^{(3)}_{\alpha,\delta,\eps}(1)$ is obtained by \emph{grafting} the tree $(t_n(i), i)$ to the vertex $i$ in $\cT_{\lfloor n^{\gamma-\eps}\rfloor}$ for each $i \in [\lfloor n^{\gamma-\eps}\rfloor ]$, that is, identifying the root of $t_n(i)$ with the vertex $i$ in the tree $\cT_{\lfloor n^{\gamma-\eps}\rfloor }$, see Figure \ref{fig:grafting}.

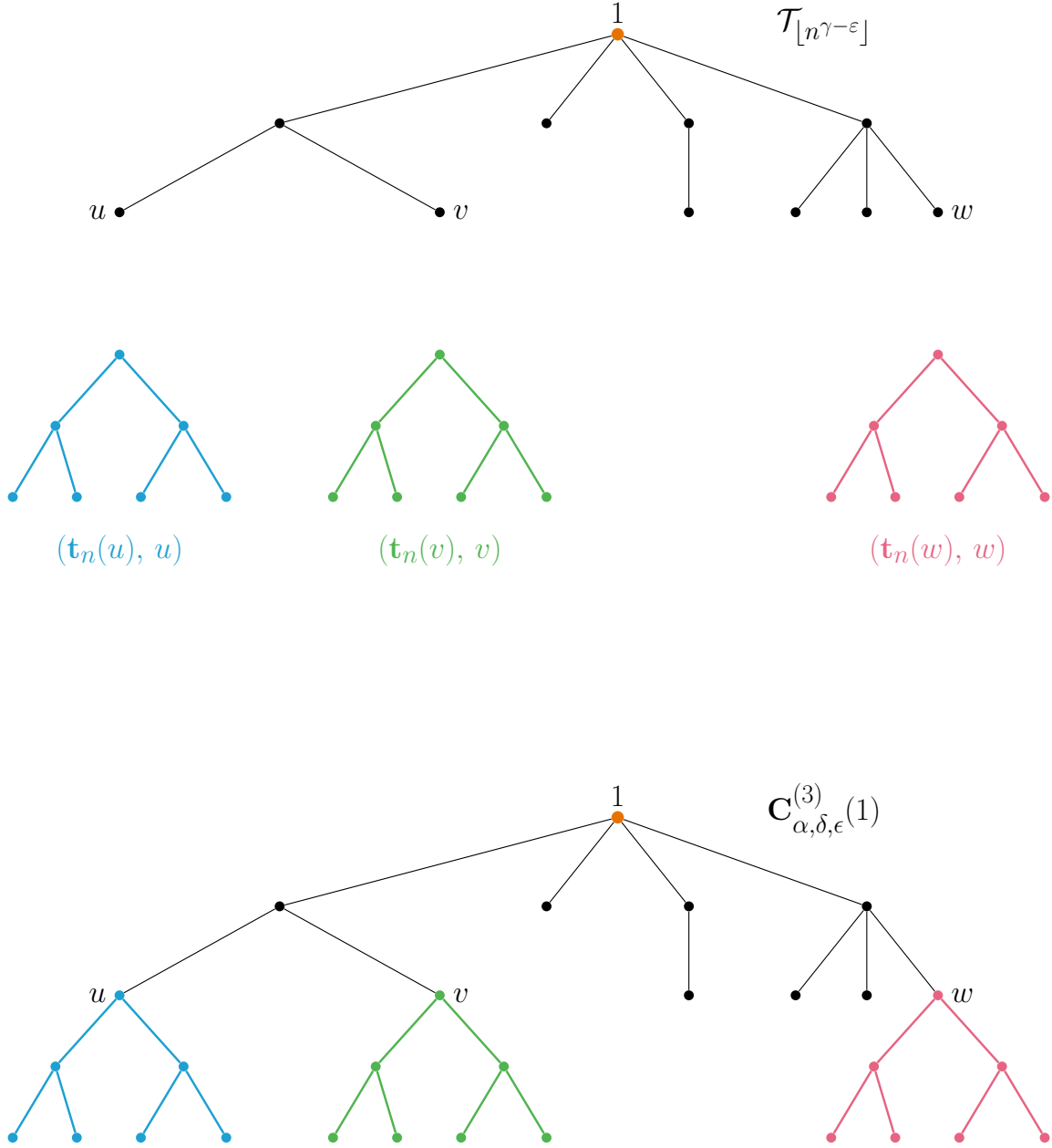
\begin{figure}[H]
    \centering
    \resizebox{\textwidth}{!}{
        \input{grafting}}
    \caption{ Grafting, from top to bottom: the black subtree is $\cT_{\floor{n^{\gamma-\eps}}}$, and $u$, $v$ and $w$ are the only vertices in it such that $t_n(u), t_n(v)$ and $t_n(w)$ are trees with at least two nodes, which are respectively as the blue, green and the pink trees in the first picture above. Grafting them at their appropriate locations in $\cT_{\lfloor n^{\gamma-\eps}\rfloor }$ constructs $\bC^{(3)}_{\alpha,\delta,\varepsilon}(1)$, as shown in the second picture.}
    \label{fig:grafting}
\end{figure}

\begin{definition}[Thresholding with extra random leaves]\label{def:threshold_meas_def}
    Fix $k,\ell\in \N\cup\{0\}$, and $p\in (0,1)$. Given a \emph{finite} rooted tree $(\bt,o)\in \cT^*$, construct a random tree as follows.
\begin{itemize}
    \item [(i)]To each vertex $v\in \bt$, add $X_v$ extra leaves, where
    \begin{align*}
        X_v \stackrel{d}{=}\begin{cases}
            &{\rm Bin}((k-d_\bt(v))\vee 0,p),\textrm{ if }v\neq o,\\&
            {\rm Bin}((k-d_\bt(v)-1)\vee 0,p),\textrm{ if }v= o.
        \end{cases} 
    \end{align*}
    and these variables form an independent collection over $v$. Call the obtained random tree $\bt'$.
    \item [(ii)] Construct another random tree   $\bT=\psi^\ell_{\rm off}(\bt',o)$.
\end{itemize}
We view the above procedure as a map $\mu^\ell_{k,p}:\cT^* \to \cP(\cT^*)$, where $\mu^\ell_{k,p}(\bt,o)$ denotes the law of the random rooted tree $(\bT,o)$ thus constructed.
\end{definition}

\begin{remark}[Obtaining $\psi_{\rm off}$ as a special case]\label{rem:psi_deg_sp_case}
Note that $\mu^\ell_{0,p}(\bt,o)$ is the Dirac measure on $\psi^\ell_{\rm off}(\bt,o)$, for any $(\bt,o)\in \cT^*$, $p\in (0,1)$ and $\ell \in \N \cup \{0\}$.
In Definition \ref{def:threshold_meas_def}, the random tree $\bt'$ constructed from $\bt$ after adding some extra leaves, almost surely, contains $\bt$, which means,
 $\bT\supseteq \psi^\ell_{\rm off}(\bt,o)$   almost surely, where $\bT \stackrel{d}{=} \mu^{\ell}_{k,p}(\bt,o)$.
In the context of Definition \ref{def:threshold_meas_def}, on the event $\cap_{v\in \bt}
\{X_v<\ell' \}$, we note that any vertex satisfying $c_{\bt'}(v)>\ell$ must also satisfy $c_{\bt}(v)>\ell-\ell'$, which means that, on this event,
$\bT\subseteq \psi^{\ell-\ell'}_{\rm off}(\bt,o)$.
%Let us denote by $\underline{0}=(0,0,\dots)\in \R^{\N}_{\geq 0}$ the constant $0$ sequence. Note that for any $(\bt,o)\in \cT^*_{>M}$, $\phi_K^M((\bt,o),\underline{0})$ simply outputs the component of $o$ in the subforest of $\bt$ spanned by vertices in $\bt$ with degree at least $M$. In other words, for any $(\bt,o)\in \cT^*_{>M}$, recalling the definition of the operator $\psi^M_{\rm deg}$ from Definition \ref{def:psi_deg_op} note that 
%\begin{align*}
%\phi_K^M((\bt,o),\underline{0})=\psi^M_{\rm deg}(\bt,o).
%\end{align*}
\end{remark}

%Recall the subtrees $(t_n(v),v)$ needed to graft to the vertices of $\cT_{n^{\gamma-\eps}}$ to construct $\bC^{(3)}_{\alpha,\delta,\eps}(1)$, for $v\in [\lfloor n^{\gamma-\eps}\rfloor ]$. Consider the subtrees $\cT_n(v;E_{\lfloor n^{\gamma-\eps}\rfloor })$ for $v\in [\lfloor n^{\gamma-\eps}\rfloor ]$. For any $v\in [\lfloor n^{\gamma-\eps}\rfloor ]$, we view $(\cT_n(v;E_{\lfloor n^{\gamma-\eps}\rfloor }),v)$ as an element in $\cT^*$. We remark that grafting the subtrees $(\cT_n(v;E_{\lfloor n^{\gamma-\eps}\rfloor }),v)$ to the respective vertices $v\in \cT_{\lfloor n^{\gamma-\eps}\rfloor }$ in fact simply gives us the \textsc{urrt} $\cT_n$.  
\begin{comment}
Consider also an i.i.d.\ collection of sequences of random variables independent of $\cT_n$ as
\begin{align*}
\{\underline{X^{(v)}}=(X^{(v)}_1,X^{(v)}_2,\dots):v\in [n^{\gamma-\eps}]\},   \end{align*}
where for each $v\in [n^{\gamma-\eps}]$, $X^{(v)}_1,\dots,X^{(v)}_n$ are independent random variables, and $X^{(v)}_m=0$ for all $m\geq n+1$. Additionally, for each such $v$ and $1\leq i \leq n$
\begin{align*}
    X^{(v)}_i\stackrel{d}{=} \textrm{Bin}(n-i,\lambda/n).
\end{align*}
\end{comment}

%The key result with the map $\mu^\ell_{k,p}$ is the following.

\begin{proposition}[Sampling subtrees in $\bC^{(3)}_{\alpha,\delta,\eps}(1)$]\label{prop:sampling_noisy_subtrees}
    Let $\alpha+\delta<1-\frac{1}{2\log 2}$, where $\alpha\in (0,1)$, $\delta>0$ and let $\eps\in (0,\gamma)$. Conditionally on $\cT_n(1;E_{\lfloor n^{\gamma-\eps}\rfloor}),\dots, \cT_n(\lfloor n^{\gamma-\eps}\rfloor;E_{\lfloor n^{\gamma-\eps}\rfloor})$, the law of the random vector
    $((t_n(1),1),\dots,(t_n(\lfloor n^{\gamma-\eps}\rfloor),\lfloor n^{\gamma-\eps}\rfloor ))$
is given by the following measure on $(\cT^*)^{\lfloor n^{\gamma-\eps}\rfloor }$:
    \begin{align*}
        \mu^{(1-\alpha)\log n-1}_{n,\lambda/n}(\cT_n(1;E_{\lfloor n^{\gamma-\eps}\rfloor }),1)\times \dots \times \mu^{(1-\alpha)\log n-1}_{n,\lambda/n}(\cT_n(\lfloor n^{\gamma-\eps}\rfloor ;E_{\lfloor n^{\gamma-\eps}\rfloor }),\lfloor n^{\gamma-\eps}\rfloor ).
    \end{align*}In particular, the real-valued random vector
$(|t_n(1)|,\dots,|t_n(\lfloor n^{\gamma-\eps}\rfloor )|)$
is exchangeable.
%\red{[from here]}
    %Then, the random objects $\cT_n$, $G$, $\{Z_v:v\in\{n^{\gamma-\eps},n^{\gamma-\eps}+1,\dots,n\}\cap D_{\cT_n}(\alpha+\delta)\}$ and $\{\underline{X^{(v)}}=(X^{(v)}_1, X^{(v)}_2,\dots):v\in [n^{\gamma-\eps}]\}$ can all be coupled in the same probability space, such that
    %\begin{align*}
        %\Prob{(t_n(v),v)=\left(\phi^{(1-\alpha)\log n}_{(1-\alpha-\delta)\log n}((\cT_n(v;E_{n^{\gamma-\eps}}),v),\underline{X^{(v)}}),v\right)\;\forall\;1\leq v \leq n^{\gamma-\eps}}\to 1.
    %\end{align*}
\end{proposition}
\begin{remark}\label{rem:extra_leaves_expl}
    Recall that for any $v\in [\lfloor n^{\gamma-\eps}\rfloor ]$, the random tree with the measure $\mu^{(1-\alpha)\log n-1}_{n,\lambda/n}$ is constructed by first assigning ${\rm Bin}(n-d_{\cT_n}(u),\lambda/n)$ leaves to each vertex $u$ of $\cT_n(v;E_{\lfloor n^{\gamma-\eps}\rfloor })$ before applying $\psi^{(1-\alpha)\log n-1}_{\rm off}(\cdot)$ on the obtained random tree. We may think of these extra leaves incident to each $v$ as the neighbors of $v$ in $G$ that are already not neighbors of $v$ in $\cT_n$.  
\end{remark}

\begin{proof}[Proof of Proposition \ref{prop:sampling_noisy_subtrees}]
    Recall the construction of the graph $\kH^{(3)}_n(\alpha,\delta,\eps)$ from Definition \ref{def:h3}. For any $v\in [\lfloor n^{\gamma-\eps}\rfloor ]$, observe that  any vertex in $u\in t_n(v)$ has $d_{\cT_n}(u)+Z_u>(1-\alpha)\log n$, and further, any vertex $w$ in the ancestral line of $u$ to $v$ in $(\cT_n(v;E_{\lfloor n^{\gamma-\eps}\rfloor }),v)$ satisfies $d_{\cT_n}(w)+Z_w>(1-\alpha)\log n$. These properties are equivalent to  $c_{\cT_n}(u)+Z_u>(1-\alpha)\log n -1$ and $c_{\cT_n}(w)+Z_w>(1-\alpha)\log n-1$. For any vertex $v\in \cT_n(v; E_{\lfloor n^{\gamma-\eps}\rfloor })$, its number of children in this tree is the same as in $\cT_n$. Furthermore, the degree of $v$ in $\cT_n(v;E_{\lfloor n^{\gamma-\eps}\rfloor })$ is one less than its degree in $\cT_n$.

    In particular, recalling $Z_u\stackrel{d}{=} {\rm Bin}(n-d_{\cT_n}(u),\lambda/n),Z_w\stackrel{d}{=} {\rm Bin}(n-d_{\cT_n}(w),\lambda/n)$ and the construction of the law $\mu^{(1-\alpha)\log n-1}_{n,\lambda/n}(\cT_n(v;E_{\lfloor n^{\gamma-\eps}\rfloor }),v)$, we note that $(t_n(v),v)$ has this law, given $\cT_n(v;E_{\lfloor n^{\gamma-\eps}\rfloor })$. Additionally, since for  $i\neq j\in [r]$, $u\in \cT_n(i,E_{\lfloor n^{\gamma-\eps}\rfloor })$, and $v\in \cT_n(j,E_{\lfloor n^{\gamma-\eps}\rfloor })$, the variables $Z_u$ and $Z_v$ are conditionally independent given $\cT_n$, given $\cT_n(1;E_{\lfloor n^{\gamma-\eps}\rfloor }),\dots,\cT_n(\lfloor n^{\gamma-\eps}\rfloor ;E_{\lfloor n^{\gamma-\eps}\rfloor })$, the law of the vector $((t_n(1),1),\dots, (t_n(\lfloor n^{\gamma-\eps}\rfloor ),\lfloor n^{\gamma-\eps}\rfloor ))$ is 
    \begin{align*}
        \mu^{(1-\alpha)\log n-1}_{n,\lambda/n}(\cT_n(1;E_{\lfloor n^{\gamma-\eps}\rfloor }),1)\times \dots \times \mu^{(1-\alpha)\log n-1}_{n,\lambda/n}(\cT_n(\lfloor n^{\gamma-\eps}\rfloor ;E_{\lfloor n^{\gamma-\eps}\rfloor }),\lfloor n^{\gamma-\eps}\rfloor ).
    \end{align*}
    By Proposition \ref{prop:sample_exch}, the vector $((t_n(1),1),\dots,(t_n(\lfloor n^{\gamma-\eps}\rfloor ),\lfloor n^{\gamma-\eps}\rfloor ))$ is  an exchangeable vector in $(\cT^*)^{\lfloor n^{\gamma-\eps}\rfloor }$, and thus so is $(|t_n(1)|,\dots,|t_n(\lfloor n^{\gamma-\eps}\rfloor )|)$ as an $\R^{\lfloor n^{\gamma-\eps}\rfloor }$-valued vector. This finishes the proof.
\end{proof}

The proof of the following corollary is identical to that of Corollary \ref{cor:neg_corr}.  Denoting by $\cF_t$ the natural filtration generated by all the information up to step $t$ in the construction of $\cT_n$, we write
$Y_v=|t_n(v)|$ and $Z=\sum_{v=1}^{\floor{n^{\gamma-\eps}}}Y_v$.

%The following exchangeability corollary of the last result is crucial for root finding in $\bC^{(3)}_{\alpha,\delta,\eps}(1)$.
\begin{corollary}[Exchangeability and negative correlation of noisy subtree sizes]\label{cor:exch_negcorr_noisy}
  Under the setting of Proposition \ref{prop:sampling_noisy_subtrees}, for any $i\neq j$, we have $\CExp{(Y_i-\CExp{Y_i}{Z})(Y_j-\CExp{Y_j}{Z})}{Z,\cF_{\floor{n^{\gamma-\eps}}}}\leq 0.$\end{corollary}

We need a result analogous to the conditional concentration of Proposition \ref{prop:conc_cond_mean}. For any fixed integer $M\geq 1$, denote for $v\in [M], u\in [\floor{n^{\gamma-\eps}}]$,
\begin{align*}
t_M(v):=\bC^{(3)}_{\alpha,\delta,\eps}(1)(v;E_M) \quad \textrm{and} \quad t(u,M,\gamma-\eps)=\cT_{\floor{n^{\gamma-\eps}}}(u;M).\numberthis \label{eq:analogous_Y_notations_noisy}
\end{align*}
As in \eqref{eq:key decomp}, we have the  decomposition 
$|t_M(v)|=\sum_{u \in t(v,M,\gamma-\eps)}Y_u$.
%\red{[ write with a $-1$ before as well ]}

\begin{proposition}[Conditional concentration of noisy subtree sizes]\label{prop:conc_cond_mean_noisy}
    For any $M\geq 1$, and $\varepsilon>0$, as $n \to \infty$,
    \begin{align*}
        \Prob{\bigcup_{j\in [M]}\left\{\left||t_M(j)|-|t(j,M,\gamma-\eps)|\CExp{Y_1}{Z}\right|>\varepsilon |t(j,M,\gamma-\eps)|\CExp{Y_1}{Z}\right\}}=o(1).
    \end{align*}
\end{proposition}
We need a concentration inequality before proving this proposition. %Recall the construction of the graph $\kH_n^{(3)}(\alpha,\delta,\eps)$ from Definition \ref{def:h3}.

\begin{lemma}\label{lem:max_deg_conc}
    For any $\delta'>0$ and  $r\geq 1$, we have
    \begin{align*}
        \Prob{\max_{v: v \in D_{\cT_n}(\alpha+\delta)} Z_v> \delta' \log n}\leq \frac{1}{n^r}.
    \end{align*}
\end{lemma}
\begin{proof}
    Since $$\max_{v \in D_{\cT_n}(\alpha+\delta)}Z_v\leq \sum_{v \in D_{\cT_n}(\alpha+\delta)}Z_v,$$ and the latter sum is stochastically dominated from above by $A_1+\dots+A_n$, where the $A_i$ are i.i.d.\ with $A_1\stackrel{d}{=} {\rm Bin}(n,\lambda/n)$, it suffices to show that
$\Prob{\sum_{i=1}^n A_i>\delta' \log n}\leq n^{-r}$.
    Fix a constant $t\geq 1$ sufficiently large so that $t\delta'>r+1$. Note that $\lambda<\frac{1}{(e^t-1)}\log n$ for all large $n$ since $\lambda=o(\log n)$. In particular, by a standard Chernoff bound,
    \begin{align*}
        \Prob{\sum_{i=1}^n A_i>\delta' \log n}
        \leq \exp(-t\delta' \log n)\left(1+\lambda/n(e^t-1) \right)^n
        \le\exp(\log n(1-t\delta'))\leq n^{-r}.
    \end{align*}
\end{proof}

\begin{proof}[Proof of Proposition \ref{prop:conc_cond_mean_noisy}]
    The proof is similar to that of Proposition \ref{prop:conc_cond_mean}.  Without repeating parts of the argument, we discuss the key changes. Firstly, since $M$ is bounded, it suffices to check that the probability of any event appearing in the union above is $o(1)$. Similarly to \eqref{eq:cond_cheby_1}, for any $j\in [M]$, we have
    \begin{align*}
        &\CProb{\left||t_M(j)|-|t(j,M,\gamma-\eps)|\CExp{Y_1}{Z}\right|>\varepsilon |t(j,M,\gamma-\eps)|\CExp{Y_1}{Z}}{Z,\cF_{n^{\lfloor \gamma-\eps\rfloor }}}\\&\leq \frac{\CExp{\left(\sum_{u\in t(v,M,\gamma-\eps)}\left(Y_u-\CExp{Y_u}{Z} \right) \right)^2}{Z,\cF_{\floor{n^{\gamma-\eps}}}}}{\varepsilon^2 |t(j,M,\gamma-\eps)|^2\CExp{Y_1}{Z}^2}.
    \end{align*}
    Next, using the exchangeability and negative correlation  of the variables $Y_v$ thanks to Corollary \ref{cor:exch_negcorr_noisy}, we obtain, akin to \eqref{eq:cond_Cheby_2},
    \begin{align*}
        &\CProb{\left||t_M(j)|-|t(j,M,\gamma-\eps)|\CExp{Y_1}{Z}\right|>\varepsilon |t(j,M,\gamma-\eps)|\CExp{Y_1}{Z}}{Z,\cF_{\floor{n^{\gamma-\eps}}}}\\&\leq\frac{\CExp{Y_u^2}{Z,\cF_{n^{\lfloor \gamma-\eps\rfloor }}}}{\varepsilon^2|t(j,M,\gamma-\eps)|(Z/(\floor{n^{\gamma-\eps}}))^2}.
    \end{align*}
    In particular, just like \eqref{eq:cond_conc_ineq_UB} we have the unconditional upper bound
    \begin{align*}
        &\Prob{\left||t_M(j)|-|t(j,M,\gamma-\eps)|\CExp{Y_1}{Z}\right|>\varepsilon |t(j,M,\gamma-\eps)|\CExp{Y_1}{Z}}\\&\leq \frac{\Exp{Y_u^2}n^{2(\gamma-\eps)}}{\varepsilon^2n^{\gamma-2\eps}n^{2(\alpha-\delta')}}+\Prob{Z<n^{\alpha-\delta'}}+\Prob{|t(j,M,\gamma-\eps)|<n^{\gamma-2\eps}},\numberthis \label{eq:cond_conc_ineq_UB_noisy}
    \end{align*}
    for any $\delta'>0$. The fact that the third term above is $o(1)$ follows from the fact that $|t(j, M,\gamma-\eps)|/\floor{n^{\gamma-\eps}}$ is tight (in particular, close to the $j$-th component of a Dirichlet random vector), as argued below \eqref{eq:cond_conc_ineq_UB}. For the second term, we note that
    \begin{align*}
        |\bC^{(3)}_{\alpha,\delta,\eps}(1)|=\floor{n^{\gamma-\eps}}+\sum_{v\in [\floor{n^{\gamma-\eps}}]}Y_v=\floor{n^{\gamma-\eps}}+Z.
    \end{align*}
    Now, Proposition \ref{prop:coup} and Corollary \ref{cor:prop_Hforest} implies that $|\bC^{(3)}_{\alpha,\delta,\eps}(1)|$ is at least $\alpha-\delta'/2$, with probability tending to $1$, 
    and therefore $Z$ is at least $n^{\alpha-\delta'}$. Thus, the second term on the right-hand side of \eqref{eq:cond_conc_ineq_UB_noisy} is also $o(1)$.

To bound the first term on the right-hand side of \eqref{eq:cond_conc_ineq_UB_noisy}, recall that, given $\cT_n$, recall $t_n(u)$ has the law $\mu^{(1-\alpha)\log n-1}_{n,\lambda/n}(\cT_n(u;E_{\floor{n^{\gamma-\eps}}}),u)$, by Proposition \ref{prop:sampling_noisy_subtrees}. Thus, to bound the second moment of its size, we write
    \begin{align*}
\Exp{Y_u^2}=\Exp{|t_n(u)|^2}=\Exp{|t_n(u)|^2\indE{E}}+\Exp{|t_n(u)|^2\indE{E^c}},
    \end{align*}
where $E=\{\max_{v: v\in D_{\cT_n}(\alpha+\delta)} Z_v<\delta \log n\}$. On the event $E$, by the third item of Remark \ref{rem:psi_deg_sp_case}, we have 
$\Exp{|t_n(u)|^2\indE{E}}\leq \Exp{Y_{\alpha+\delta}(u)^2}$,
where recall $Y_\alpha(u)$ from \eqref{eq:Y'}, and the representation of $Y_\alpha$ through the operator $\psi^{(1-\alpha)\log n-1}_{\rm off}$ from \eqref{eq:subtrees_in_comp_op_exp}. On the other hand, by trivially bounding $|t_n(u)|^2\leq n^2$, by Lemma \ref{lem:max_deg_conc}, we note that
$\Exp{|t_n(u)|^2\indE{E^c}}=o(1)$.

%Now, by the definition of the operator $\phi_K^M$, we note that $\phi^{(1-\alpha-\delta)\log n}_{(1-\alpha-\delta)\log n}((\cT_n(v; E_{n^{\gamma-\eps}}),v),\underline{0})$ is the component of $v$, when we construct the subtree of $\cT_n(v; E_{n^{\gamma-\eps}})$ spanned by those vertices in it that have degree at least $(1-\alpha-\delta)\log n$, equivalently, at least $((1-\alpha-\delta)\log n)-1$ children. 
Thus, by Lemma \ref{lem:sec_mom_Yv} and choosing $\delta>0$ sufficiently small, using the continuity of $\gamma(\alpha)$ as a function of $\alpha$ (recall Lemma \ref{lem:gamma_cont}), we conclude that  
$\Exp{Y_u^2}\leq n^{2(\alpha-\gamma+\eps+3\delta)}$.
As a consequence, by choosing $\eps,\delta,\delta'>0$ sufficiently small, we have
\begin{align*}
    \frac{\Exp{Y_u^2}n^{2(\gamma-\eps)}}{\varepsilon^2n^{\gamma-2\eps}n^{2(\alpha-\delta')}}\leq \frac{n^{-2\eps+6\delta+2\delta'-\gamma}}{\varepsilon^2}=o(1),
\end{align*}
proving the right-hand side of \eqref{eq:cond_conc_ineq_UB_noisy} is $o(1)$, and thus the proposition.\end{proof}

%We first prove that we can find the root vertex $1$ in $\bC^{(3)}_{\alpha,\delta,\eps}(1)$ when observing the graph without vertex labels. 

\begin{proposition}[Root finding in $\bC^{(3)}_{\alpha,\delta,\eps}(1)$]\label{prop:root_finding _in_noisyrootcomp}
    Let $\alpha\in (0,1)$ and $\delta>0$ be such that $\alpha+\delta<1-\frac{1}{2\log 2}$. Fix $\eps\in (0,\gamma)$. Consider observing the graph $\bC^{(3)}_{\alpha,\delta,\eps}(1)$ without vertex labels, and fix $\varepsilon>0$. It is possible to construct a confidence set $K=K(\varepsilon)\subset V\left(\bC^{(3)}_{\alpha,\delta,\eps}(1)\right)$ not depending on $n$ such that 
    \begin{align*}
        \Prob{1\in K(\varepsilon)}>1-\varepsilon.
    \end{align*}
\end{proposition}

Before proving Proposition \ref{prop:root_finding _in_noisyrootcomp}, we finish the proof of Theorem \ref{thm:ER_noise}.
\begin{proof}[Proof of Theorem \ref{thm:ER_noise}]
    Observing the graph $H=\cT_n(G)$, we choose  $\alpha<1-\frac{1}{2\log 2}$, and construct the graph $\kH_n(\alpha)$, spanned by the vertices in $H$ that have degree at least $(1-\alpha)\log n$. By Corollary \ref{cor:prop_Hforest}, with high probability,  $\bC_{\kH,\alpha}(1)$ is the unique connected component of size at least $n^{\alpha-\gamma/4}$. By Proposition \ref{prop:coup}, $\bC_{\kH,\alpha}(1)=\bC^{(3)}_{\alpha,\delta,\eps}(1)$ for any $\eps\in (0,\gamma)$ and for $\delta>0$ sufficiently small such that $\alpha+\delta<1-\frac{1}{2\log 2}$, with probability tending to $1$. Thus by Proposition \ref{prop:root_finding _in_noisyrootcomp}, we can choose a $K(\varepsilon)\subset V(\bC_{\kH,\alpha}(1))$ independent of $n$ such that $\Prob{1\in K(\varepsilon)}>1-\varepsilon$. 
\end{proof}

\begin{proof}[Proof of Proposition \ref{prop:root_finding _in_noisyrootcomp}]
    The proof is similar to that of Theorem \ref{thm:root_finding_c1}, so we only describe the key changes from that argument.
For $M=M(\varepsilon)$, we take $K(\varepsilon)=H(M,T)$, that is, the $M$ most central vertices according to Jordan centrality,
where
$T=\bC^{(3)}_{\alpha,\delta,\eps}(1)$.
As in \eqref{eq:root_finding_1}, we  derive
    \begin{align*}
        \Prob{1\notin H(M,T)}\leq \Prob{|t_M(2)|+\dots+|t_M(M)|\leq |t_M(1)|}.
    \end{align*}
Defining the good event 
\begin{align*}
\cE(\varepsilon)=\bigcap_{j\in [M]}\left\{\left||t_M(j)|-|t(j,M,\gamma-\eps)|\CExp{Y_j}{Z}\right|\leq \varepsilon\left( |t(j,M,\gamma-\eps)|\CExp{Y_1}{Z}\right)\right\},
    \end{align*}
    by Proposition \ref{prop:conc_cond_mean_noisy}, it is enough to show that
    \begin{align*}
        \Prob{|t_M(2)|+\dots+|t_M(M)|\leq |t_M(1)|,\cE(\varepsilon)}<\varepsilon~.
    \end{align*}
 On the event $\cE(\varepsilon)$, we have 
    \begin{align*}
    &\Prob{|t_M(2)|+\dots+|t_M(M)|\leq |t_M(1)|,\cE(\varepsilon)}\\&\leq \Prob{|t(2,M,\gamma-\eps)|+\dots+|t(M,M,\gamma-\eps)|\leq \frac{1+\varepsilon}{1-\varepsilon}|t(1,M,\gamma-\eps)|}.\numberthis \label{eq:subtree_ratio_noisy}
    \end{align*}
    Recalling $|t(j,M,\gamma-\eps)|=|\cT_{\floor{n^{\gamma-\eps}}}(j;E_M)|$ from \eqref{eq:analogous_Y_notations_noisy}, note that the random vector
    $$\frac{1}{\floor{n^{\gamma-\eps}}}  (|t(1,M,\gamma-\eps)|,\dots,|t(M,M,\gamma-\eps)|)$$ converges in law to a Dirichlet random vector. So, we can choose $M=M(\varepsilon)$ sufficiently large so that the right-hand side in \eqref{eq:subtree_ratio_noisy} is at most $\varepsilon$.
\end{proof}

\section{Root finding under random matching noise}\label{Sec:matching}

In this section, we prove Theorem \ref{thm:matching_noise}.  
%The standard algorithm to output a uniformly random perfect matching goes as follows. First, put an arbitrary order $v_1<\dots<v_n$ on the vertices of $K_n$.
%\begin{itemize}
%    \item [(i)] Let $M_0=\emptyset$. We start with $M_1=\{\{v_1,p(v_1)\}\}$, where $p(v_1)$ is uniformly distributed on $\{v_2,\dots,v_n\}$. 
%    \item [(ii)] At a given step $t$, given $M_t$, choose the least vertex $u$ not in $V_t:=\{v\in [n]:\exists\;e\in M_t\textrm{ such that }v\in e\}$. Sample $p(u)$ uniformly at random from $V\setminus(V_t\cup\{u\})$, and update $M_{t+1}=M_t\cup\{u,p(u)\}$.
%\end{itemize}
%Then, $\sM_n:=M_{n/2}$ has the law of a uniformly distributed perfect matching on $K_n$. It is easy to verify that for any fixed matching $M$,
%\begin{align*}
%    \Prob{\sM_n=M}=\frac{1}{(n-1)!!}.
%\end{align*}
%We note that the order $v_1<\dots<v_n$ can be \emph{arbitrary}. 
Recall that $\cT_n$ is a \textsc{urrt} on vertex set $[n]$, and $\sM_{n}$ is a uniformly random perfect matching on $K_n$, independent of $\cT_n$. 
In order to find the root upon observing the graph $H_n=\cT_n(\sM_n)$, we take  the subgraph $H_n(\alpha)$ of $H_n$, formed by retaining edges $\{u,v\}$ that satisfy
\begin{align*}
    d_{H_n}(u)\wedge d_{H_n}(v)>(1-\alpha) \log n.
\end{align*}

Th following proposition may be easily proved using standard properties of the degree distribution of a \textsc{urrt}, see, e.g., \citet{devroye1995strong}.

\begin{proposition}\label{prop:matching_props}
    Let $\alpha<1-\frac{1}{2\log 2}$ be fixed. The following hold with high probability:
    \begin{itemize}
        \item [(i)] There is no matching edge $\{u,v\}\in \sM_n$ with  $d_{\cT_n}(u) \wedge d_{\cT_n}(v)>(1-\alpha)\log n$.
        \item [(ii)] For any $u\in [n]$ with $d_{\cT_n}(u)>(1-\alpha)\log n$, $d_{H_n}(u)=d_{\cT_n}(u)+1$.
    \end{itemize}
    \end{proposition}

\begin{proof}[Proof of Theorem \ref{thm:matching_noise}]
We claim that, with high probability,
\[
        \{v\in [n]:d_H(v)>(1-\alpha)\log n+1\}=\{v\in [n]:d_{\cT_n}(v)>(1-\alpha)\log n\}.\numberthis \label{eq:matching_vtx_set_eq}
    \]
To see this, note that  $d_{\cT_n}(v)\leq d_H(v)\leq d_{\cT_n}(v)+1$ implies that
\begin{align*}
        \{v\in [n]:d_H(v)>(1-\alpha)\log n+1\}\subseteq\{v\in [n]:d_{\cT_n}(v)>(1-\alpha)\log n\}.
    \end{align*}
Next consider any $v$ with $d_{\cT_n}(v)>(1-\alpha)\log n$. By Proposition \ref{prop:matching_props}$(ii)$, $d_{H}(v)=d_{\cT_n}(v)+1>(1-\alpha)\log n+1$, so that \eqref{eq:matching_vtx_set_eq} follows.
Also, by Proposition \ref{prop:matching_props}$(i)$, any edge in $H$ between  vertices $u$ and $v$ satisfying $d_{\cT_n}(u) \wedge d_{\cT_n}(v)>(1-\alpha)\log n$ must be an \textsc{urrt} edge. We conclude that $\Prob{H_n(\alpha)=\cF_n(\alpha)}\to 1$.
Thus,  root finding in $H_n(\alpha)$ is equivalent to
root finding in $\cF_n(\alpha)$, and the theorem follows from Theorem \ref{thm:forest}. %\red{[ the vertex perspective is important here ]}
\end{proof}

\section{A general discussion on robust root finding}\label{Sec:disc}
In this paper, we developed a robust method of root finding that works in \textsc{urrt}s with additional noisy edges.  We focused on two specific examples of $\cT_n(G)=\cT_n\cup G$, with $G$ being the random perfect matching (Theorem \ref{thm:matching_noise}) and $G=\G(n,\lambda_n/n)$ (Theorem \ref{thm:ER_noise}).
The technique can be adapted
for more general choices of $G$,  as long as $G$ satisfies certain properties. In order to guarantee that the approach works, it suffices the check the following steps:

\subsection{Sandwiching}\label{sec:discussion}
Consider any $G$, with $\max_{v \in [n]}d_G(v)=o(\log n)$. 
The subgraph $H_n(\alpha)$ spanned by edges with both end-vertices having degree at least $(1-\alpha)\log n$ in $\cT_n(G)$ satisfies the "sandwitching" properties  of Lemma \ref{lem:sandwiching_forests} and Corollary \ref{cor:prop_Hforest}.  This means that $H_n(\alpha)$ is a forest, whose largest component $\bC_{H,\alpha}(1)$ contains the root, with high probability.
Moreover, it contains the \textsc{urrt} $\cT_{\lfloor n^{\gamma-\eps}\rfloor}$ up to the first $\lfloor n^{\gamma-\eps}\rfloor$ steps. 

In the proof of Theorem \ref{thm:ER_noise}, the sandwiching step is where the condition $\lambda_n=o(\log n)$ is used. Indeed, when $\lambda_n=\Omega(\log n)$, the maximum degree of the noise graph starts competing with the maximum degree of the \textsc{urrt}, and the sandwiching step is not valid anymore.  

\subsection{Exchangeability} 
%Once the sandwiching step is true, we are thus tasked to find the root in $\bC_{H,\alpha}(1)$. 
By the sandwiching step, we can construct the component $\bC_{H,\alpha}(1)$ by grafting to each vertex $v\in [\lfloor n^{\gamma-\eps}\rfloor]$ of $\cT_{\lfloor n^{\gamma-\eps}\rfloor}$, the subtree $t_n(v):=\bC_{H,\alpha}(1)(v;E_{\lfloor n^{\gamma-\eps}\rfloor})$. The key question is whether the collection $(t_n(1),\dots,t_n(\lfloor n^{\gamma-\eps}\rfloor))$ is exchangeable. In that case, the arguments of Propositions \ref{prop:conc_cond_mean_noisy} and \ref{prop:root_finding _in_noisyrootcomp} go through, and root finding is possible. 

We have worked with two somewhat closely related methods to verify this exchangeability. In the proof of Theorem \ref{thm:forest}, we established exchangeability by  writing $(t_n(v),v)$ as $\psi^{(1-\alpha)\log n-1}_{\rm off}(\cT_n(v; E_{\lfloor n^{\gamma-\eps}\rfloor}),v)$ and using that such functionals are exchangeable, due to Proposition \ref{prop:exch_subtree_funcs}. For Theorem \ref{thm:ER_noise}, this exchangeability is established using Proposition \ref{prop:sample_exch}, which says that `samples from exchangeable measures are exchangeable'. If exchangeability can be established by a model-dependent argument, then 
our techniques can be used for proving that root-finding is possible.

\subsection{Example: random regular graphs}

Another example where our techniques can be applied is when $G$ is a \emph{random $d_n$-regular graph} with $d_n=o(\log n)$. (The special case with $d_n=1$ reduces to random perfect matchings discussed above.) The condition on $d_n$ ensures that the sandwiching step (Lemma \ref{lem:sandwiching_forests} and Corollary \ref{cor:prop_Hforest}) go through when we construct the high-degree subgraph $H_n(\alpha)$. A calculation analogous to Lemma \ref{lem:ER_nbhds_large} shows that all vertices with degree at least $(1-\alpha)\log n$ in $\cT_n$ have neighborhoods that are independent size-$d_n$ random subsets of $[n]$. 
Consequently, for any $L\in \{0,1,\dots,n\}$, if we define 
$X_L:=|S\setminus [L]|$,
  where $S$ is a size-$d_n$ random subset of $[n]$, we can construct $t_n(v)$ in a similar way as in Proposition \ref{prop:sampling_noisy_subtrees}:
\begin{itemize}
    \item [(i)] To each vertex in $\cT_n(v;E_{\lfloor n^{\gamma-\eps}\rfloor})$, add $X_{d_{\cT_n}(v)}$ extra leaves to construct the random tree $(\T'_v,v)$. The leaves represent the neighbors of $v$ in $G$ that are \emph{not} neighbors in $\cT_n$.
    \item [(ii)] $(t_n(v),v)$ equals $\psi^{(1-\alpha)\log n-1}_{\rm off}(\T'_v,v)$ with high probability.
\end{itemize}
%\textcolor{red}{LUC: below and elsewhere, $n^\gamma$ is not integer-valued...}
In particular, $(t_n(1),\dots,t_n(\lfloor n^{\gamma-\eps}\rfloor))$ is exchangeable, and Propositions \ref{prop:conc_cond_mean_noisy} and \ref{prop:root_finding _in_noisyrootcomp} go through.

%\subsubsection{Inhomogeneous Random Graphs} \red{[N: To write. Exchangeability of edge indicators is required]}

\paragraph{Acknowledgments.} 
Luc Devroye was supported by the Natural Sciences and Engineering Research Council of Canada (\textsc{nserc}) under grant number \textsc{rgpin}-2024-04164.
Gábor Lugosi acknowledges the support of Spanish Ministry of Economy and Competitiveness grant PID2022-138268NB-I00, financed by MCIN/AEI/10.13039/501100011033,
FSE+MTM2015-67304-P, and FEDER, EU). Neeladri Maitra gratefully acknowledges the support of an \textsc{ams} Simons Travel Grant for a visit to McGill University and the hospitality of both the School of Computer Science and the Department of Mathematics and Statistics there.

\bibliographystyle{abbrvnat}
\bibliography{ref}
\end{document}

%% file: operators_psi_and_chi.tex
\begin{tikzpicture}[
  every node/.style={circle, fill, inner sep=2pt},
  root/.style={circle, draw, fill=none, inner sep=3pt},
  rednode/.style={circle, fill=red, inner sep=3pt},
  blacknode/.style={circle, fill=black, inner sep=2.5pt},
  edge/.style={thick},
  leafedge/.style={thick, postaction={decorate, decoration={markings,
    mark=at position 1 with {\filldraw[black] circle (2.5pt);}}}},
  label/.style={draw=none, fill=none, font=\large},
  scale=1.0
]

%% -------------------------------------------------------
%% TOP: The rooted tree (\mathbf{t}, o)
%% -------------------------------------------------------

\node[label] at (-3.5, 0) {$\mathbf{t}$};

% Root o (open circle)
\node[root] (root) at (0, 0) {};
\node[label, font=\normalsize] at (0, 0.35) {$o$};

% Level-1 children
\node[rednode]   (r1)    at (-2.0, -1.2) {};
\node[blacknode] (b1)    at ( 0.4, -1.2) {};
\node[blacknode] (b2)    at ( 2.2, -1.2) {};
% Two extra leaf children of root
\node[blacknode] (leaf1) at ( 3.5, -1.0) {};
\node[blacknode] (leaf2) at ( 4.4, -0.5) {};

\draw[edge] (root) -- (r1);
\draw[edge] (root) -- (b1);
\draw[edge] (root) -- (b2);
\draw[edge] (root) -- (leaf1);
\draw[edge] (root) -- (leaf2);

% Fan of leaves from r1 with dot at tip
\foreach \angle in {203, 218, 233, 248, 263, 278, 293, 308} {
  \draw[leafedge] (r1) -- ++(\angle:1.4);
}

% Middle branch: b1 -> red child r2 -> leaf fan
\node[rednode] (r2) at (0.7, -2.6) {};
\draw[edge] (b1) -- (r2);

\foreach \angle in {210, 224, 238, 252, 266, 280, 294, 308, 322} {
  \draw[leafedge] (r2) -- ++(\angle:1.3);
}

% Rightmost black node b2: only two leaf children with dot at tip
\foreach \angle in {230, 310} {
  \draw[leafedge] (b2) -- ++(\angle:1.3);
}

%% -------------------------------------------------------
%% MIDDLE: \chi_{>5}(\mathbf{t}, o)
%% -------------------------------------------------------

\node[label] at (-3.5, -6.5) {$\chi_{>5}(\mathbf{t},\, o)\ :$};

\node[root]    (root2) at (-0.5, -6.0) {};
\node[label, font=\normalsize] at (-0.5, -5.65) {$o$};

\node[rednode]   (r2b)          at (-1.5, -7.3) {};
\node[blacknode] (chi_mid_leaf) at (-0.2, -7.2) {};
\node[blacknode] (b2b)          at ( 0.9, -7.0) {};
\node[blacknode] (chi_leaf1)    at ( 1.8, -6.7) {};
\node[blacknode] (chi_leaf2)    at ( 2.5, -6.2) {};

\draw[edge] (root2) -- (r2b);
\draw[edge] (root2) -- (chi_mid_leaf);
\draw[edge] (root2) -- (b2b);
\draw[edge] (root2) -- (chi_leaf1);
\draw[edge] (root2) -- (chi_leaf2);

% Fan of leaves from r2b with dot at tip
\foreach \angle in {210, 224, 238, 252, 266, 280, 294, 308} {
  \draw[leafedge] (r2b) -- ++(\angle:1.2);
}

%% -------------------------------------------------------
%% BOTTOM: \psi^{K}_{\mathrm{off}}(\mathbf{t}, o)
%% -------------------------------------------------------

\node[label] at (-3.5, -10.5) {$\psi^{5}_{\mathrm{off}}(\mathbf{t},\, o)\ :$};

\node[root]    (root3) at (-0.5, -10.0) {};
\node[label, font=\normalsize] at (-0.5, -9.65) {$o$};

\node[rednode] (r3) at (-2.1, -11.2) {};

\draw[edge] (root3) -- (r3);

\end{tikzpicture}

%% file: LB_bad_vertex.tex
\begin{tikzpicture}[
  orangenode/.style={circle, fill=orange!90!black, inner sep=3.5pt},
  rednode/.style={circle, fill=red!80!black, inner sep=3.5pt},
  greennode/.style={circle, fill=green!60!black, inner sep=3.5pt},
  blackleaf/.style={circle, fill=black, inner sep=2pt},
  edge/.style={thick},
  blueedge/.style={thick, blue!70!black, line width=2pt},
  lbl/.style={draw=none, fill=none},
  scale=1.0
]

%% ================================================================
%% LEFT TREE:  \mathcal{T}_n
%%
%% Layout strategy: divide horizontal space into sectors
%%   x in [-9, -2]  : subtree of A  (left orange child of root)
%%   x in [-1,  0]  : black leaves directly from root
%%   x in [ 1,  3]  : black leaves of B (centre orange, blue path)
%%   x in [ 4, 10]  : subtrees of C and D (right blue path)
%%
%% Root at (0, 0).  All y coordinates are strictly decreasing
%% level by level so no horizontal band overlaps another.
%% ================================================================

%% -- Root --
\node[orangenode] (root) at (0, 0) {};
\node[lbl, font=\small] at (-0.45, 0.4) {$1$};

%% ----------------------------------------------------------------
%% Subtree of A  (leftmost orange child, BLUE edge from root)
%% A is at (-5, -2).  Its 8 black leaves fan out at y = -4,
%% strictly within x in [-9, -2], evenly spaced.
%% ----------------------------------------------------------------
\node[orangenode] (A) at (-5.0, -2.0) {};

\draw[edge] (root) -- (A);

% 8 black leaves of A, spread across x in [-8.5, -1.5], y = -4.2
\node[rednode] (Al1) at (-8.5, -4.2) {};
\node[rednode] (Al2) at (-7.5, -4.2) {};
\node[rednode] (Al3) at (-6.5, -4.2) {};
\node[rednode] (Al4) at (-5.5, -4.2) {};
\node[rednode] (Al5) at (-4.5, -4.2) {};
\node[rednode] (Al6) at (-3.5, -4.2) {};
\node[rednode] (Al7) at (-2.5, -4.2) {};
\node[rednode] (Al8) at (-1.5, -4.2) {};

\draw[edge] (A) -- (Al1);
\draw[edge] (A) -- (Al2);
\draw[edge] (A) -- (Al3);
\draw[edge] (A) -- (Al4);
\draw[edge] (A) -- (Al5);
\draw[edge] (A) -- (Al6);
\draw[edge] (A) -- (Al7);
\draw[edge] (A) -- (Al8);

%% ----------------------------------------------------------------
%% Black leaf children directly from root (right side, no subtrees)
%% lf1, lf2, lf3 spread to the right of root at y = -2,
%% well to the right so they don't touch B's region.
%% ----------------------------------------------------------------
\node[rednode] (lf1) at ( 4.5, -2.2) {};
\node[rednode] (lf2) at ( 6.0, -1.0) {};
\node[rednode] (lf3) at ( 7.2, -2.0) {};

\draw[edge] (root) -- (lf1);
\draw[edge] (root) -- (lf2);
\draw[edge] (root) -- (lf3);

%% ----------------------------------------------------------------
%% B  (centre orange child of root, BLUE edge)
%% B at (1.5, -2).  6 black leaves fan within x in [0.5, 3.5],
%% y = -4.  C (child of B, blue) is placed further right/down.
%% ----------------------------------------------------------------
\node[orangenode] (B) at (1.5, -2.0) {};

\draw[blueedge] (root) -- (B);

% 6 black leaves of B, spread at y = -4.0, x in [-0.5, 3.0]
% (kept to the LEFT of C's zone which starts at x ~ 5)
\node[rednode] (Bl1) at (-0.5, -4.0) {};
\node[rednode] (Bl2) at ( 0.3, -4.0) {};
\node[rednode] (Bl3) at ( 1.1, -4.0) {};
\node[rednode] (Bl4) at ( 1.9, -4.0) {};
\node[rednode] (Bl5) at ( 2.7, -4.0) {};
\node[rednode] (Bl6) at ( 3.5, -4.0) {};

\draw[edge] (B) -- (Bl1);
\draw[edge] (B) -- (Bl2);
\draw[edge] (B) -- (Bl3);
\draw[edge] (B) -- (Bl4);
\draw[edge] (B) -- (Bl5);
\draw[edge] (B) -- (Bl6);

%% ----------------------------------------------------------------
%% C  (orange grandchild via B, BLUE edge B->C)
%% C at (6.5, -4).  6 black leaves at y = -6, x in [4.5, 9.0].
%% D (red child of C, blue) is directly below C.
%% ----------------------------------------------------------------
\node[orangenode] (C) at (6.5, -4.0) {};

\draw[blueedge] (B) -- (C);

% 6 black leaves of C spread at y = -6.0, x in [4.5, 9.0]
\node[rednode] (Cl1) at (4.5, -6.0) {};
\node[rednode] (Cl2) at (5.3, -6.0) {};
\node[rednode] (Cl3) at (6.1, -6.0) {};
\node[rednode] (Cl4) at (7.0, -6.0) {};
\node[rednode] (Cl5) at (7.9, -6.0) {};
\node[rednode] (Cl6) at (8.8, -6.0) {};

\draw[edge] (C) -- (Cl1);
\draw[edge] (C) -- (Cl2);
\draw[edge] (C) -- (Cl3);
\draw[edge] (C) -- (Cl4);
\draw[edge] (C) -- (Cl5);
\draw[edge] (C) -- (Cl6);

%% ----------------------------------------------------------------
%% D  (red child of C, BLUE edge C->D)
%% D at (6.5, -7.5).  Two green leaf children below D.
%% ----------------------------------------------------------------
\node[rednode] (D) at (6.5, -7.5) {};

\draw[blueedge] (C) -- (D);

\node[greennode] (G1) at (5.7, -9.2) {};
\node[greennode] (G2) at (7.3, -9.2) {};

\draw[edge] (D) -- (G1);
\draw[edge] (D) -- (G2);

%% -- Label for left tree --
\node[lbl, font=\Large] at (-6.5, -5.5) {$\mathcal{T}_n$};

%% ================================================================
%% RIGHT TREE: \mathbf{C}_\alpha(1) \cap [n^{\alpha_1}]
%% Four orange nodes connected only by blue edges.
%% Root -> RA (left) and Root -> RB (right) -> RC (below RB).
%% ================================================================

\node[orangenode] (Rroot) at (14.0, 0.0) {};
\node[lbl, font=\Large] at (13.55, 0.4) {$1$};

\node[orangenode] (RA) at (12.5, -2.0) {};
\node[orangenode] (RB) at (15.5, -2.0) {};
\node[orangenode] (RC) at (15.5, -4.0) {};

\draw[edge] (Rroot) -- (RA);
\draw[blueedge] (Rroot) -- (RB);
\draw[blueedge] (RB)    -- (RC);

\node[lbl, font=\Large] at (15.5, -5.5)
  {$\mathbf{C}_{\alpha}(1)\cap [\floor{n^{\alpha_1}}]$};

%% ================================================================
%% LEGEND
%% ================================================================

\node[greennode]  at (-5.0, -11.5) {};
\node[lbl, font=\Large, anchor=west] at (-4.4, -11.5)
  {: label larger than $\floor{n^{\alpha_1}}$};

\node[orangenode] at (-5.0, -12.7) {};
\node[lbl, font=\Large, anchor=west] at (-4.4, -12.7)
  {: vertices with degree $> (1-\alpha)\log n$
   $(= 5$, say, for this example$)$};

\node[rednode]    at (-5.0, -13.9) {};
\node[lbl, font=\Large, anchor=west] at (-4.4, -13.9)
  {: $\mathrm{Bad}$ vertices};

\node[circle, fill=blue!70!black, inner sep=3.5pt] at (-5.0, -15.1) {};
\node[lbl, font=\Large, anchor=west] at (-4.4, -15.1)
  {: the ancestral line a $\rm Bad$ vertex breaks};

\end{tikzpicture}

%% file: grafting.tex
\begin{tikzpicture}[
  orangenode/.style={circle, fill=orange!90!black, inner sep=3.5pt},
  blacknode/.style ={circle, fill=black,   inner sep=2.8pt},
  bluenode/.style  ={circle, fill=myblue,  inner sep=2.8pt},
  greennode/.style ={circle, fill=mygreen, inner sep=2.8pt},
  pinknode/.style  ={circle, fill=mypink,  inner sep=2.8pt},
  edge/.style     ={thick},
  blueedge/.style  ={thick, myblue,  line width=1.8pt},
  greenedge/.style ={thick, mygreen, line width=1.8pt},
  pinkedge/.style  ={thick, mypink,  line width=1.8pt},
  lbl/.style={draw=none, fill=none},
]

%% Both pictures share IDENTICAL relative coordinates.
%% Left picture origin: (0, 0)
%% Right picture origin: (0, Ry) where Ry = -22  (stacked below)
%% Horizontal centre of both pictures: x = 0

%% ============================================================
%% LEFT PICTURE  (origin 0, 0)
%% ============================================================

\node[orangenode](root) at (0.0, 0.0){};
\node[lbl,font=\Huge] at (0, 0.6){$1$};
\node[lbl,font=\Huge] at (5.8, 0.2){$\mathcal{T}_{\floor{n^{\gamma-\varepsilon}}}$};

\node[blacknode](A)   at (-9.5,-2.5){};
\node[blacknode](lf0) at (-2.0,-2.5){};
\node[blacknode](B)   at ( 2.0,-2.5){};
\node[blacknode](C)   at ( 7.0,-2.5){};
\draw[edge](root)--(A); \draw[edge](root)--(lf0);
\draw[edge](root)--(B); \draw[edge](root)--(C);

\node[blacknode](u) at (-14.0,-5.0){}; \node[lbl,font=\Huge,left=3pt of u]{$u$};
\node[blacknode](v) at ( -5.0,-5.0){}; \node[lbl,font=\Huge,right=3pt of v]{$v$};
\draw[edge](A)--(u); \draw[edge](A)--(v);

\node[blacknode](Bl1) at (2.0,-5.0){}; \draw[edge](B)--(Bl1);

\node[blacknode](Cl1) at ( 5.0,-5.0){};
\node[blacknode](Cl2) at ( 7.0,-5.0){};
\node[blacknode](w)   at ( 9.0,-5.0){}; \node[lbl,font=\Huge,right=3pt of w]{$w$};
\draw[edge](C)--(Cl1); \draw[edge](C)--(Cl2); \draw[edge](C)--(w);

%% Blue subtree (root at -14, -9)
\node[bluenode](bu)   at (-14.0,-9.0){};
\node[bluenode](bu_l) at (-15.8,-11.0){};
\node[bluenode](bu_r) at (-12.2,-11.0){};
\node[bluenode](bll1) at (-17.0,-13.0){};
\node[bluenode](bll2) at (-15.2,-13.0){};
\node[bluenode](brl1) at (-13.4,-13.0){};
\node[bluenode](brl2) at (-11.0,-13.0){};
\draw[blueedge](bu)--(bu_l);   \draw[blueedge](bu)--(bu_r);
\draw[blueedge](bu_l)--(bll1); \draw[blueedge](bu_l)--(bll2);
\draw[blueedge](bu_r)--(brl1); \draw[blueedge](bu_r)--(brl2);
\node[lbl,font=\Huge,text=myblue] at (-14.0,-14.5){$(\mathbf{t}_n(u),\,u)$};

%% Green subtree (root at -5, -9)
\node[greennode](gv)   at (-5.0,-9.0){};
\node[greennode](gv_l) at (-6.8,-11.0){};
\node[greennode](gv_r) at (-3.2,-11.0){};
\node[greennode](gll1) at (-8.0,-13.0){};
\node[greennode](gll2) at (-6.2,-13.0){};
\node[greennode](grl1) at (-4.4,-13.0){};
\node[greennode](grl2) at (-2.0,-13.0){};
\draw[greenedge](gv)--(gv_l);   \draw[greenedge](gv)--(gv_r);
\draw[greenedge](gv_l)--(gll1); \draw[greenedge](gv_l)--(gll2);
\draw[greenedge](gv_r)--(grl1); \draw[greenedge](gv_r)--(grl2);
\node[lbl,font=\Huge,text=mygreen] at (-5.0,-14.5){$(\mathbf{t}_n(v),\,v)$};

%% Pink subtree (root at 9, -9)
\node[pinknode](pw)   at ( 9.0,-9.0){};
\node[pinknode](pw_l) at ( 7.2,-11.0){};
\node[pinknode](pw_r) at (10.8,-11.0){};
\node[pinknode](pll1) at ( 6.0,-13.0){};
\node[pinknode](pll2) at ( 7.8,-13.0){};
\node[pinknode](prl1) at ( 9.6,-13.0){};
\node[pinknode](prl2) at (12.0,-13.0){};
\draw[pinkedge](pw)--(pw_l);   \draw[pinkedge](pw)--(pw_r);
\draw[pinkedge](pw_l)--(pll1); \draw[pinkedge](pw_l)--(pll2);
\draw[pinkedge](pw_r)--(prl1); \draw[pinkedge](pw_r)--(prl2);
\node[lbl,font=\Huge,text=mypink] at (9.0,-14.5){$(\mathbf{t}_n(w),\,w)$};

%% ============================================================
%% RIGHT PICTURE  (shifted down by Ry = -22)
%% Identical relative coordinates for the black tree.
%% Label placed top-right just like the left picture.
%% ============================================================

\def\Ry{-22.0}

\node[orangenode](Rroot) at (0.0, \Ry){};
\node[lbl,font=\Huge] at (0, \Ry+0.6){$1$};
\node[lbl,font=\Huge] at (5.8, \Ry+0.2){$\mathbf{C}^{(3)}_{\alpha,\delta,\epsilon}(1)$};

\node[blacknode](RA)   at (-9.5, \Ry-2.5){};
\node[blacknode](Rlf0) at (-2.0, \Ry-2.5){};
\node[blacknode](RB)   at ( 2.0, \Ry-2.5){};
\node[blacknode](RC)   at ( 7.0, \Ry-2.5){};
\draw[edge](Rroot)--(RA); \draw[edge](Rroot)--(Rlf0);
\draw[edge](Rroot)--(RB); \draw[edge](Rroot)--(RC);

\node[bluenode] (Ru) at (-14.0, \Ry-5.0){}; \node[lbl,font=\Huge,left=3pt of Ru]{$u$};
\node[greennode](Rv) at ( -5.0, \Ry-5.0){}; \node[lbl,font=\Huge,right=3pt of Rv]{$v$};
\draw[edge](RA)--(Ru); \draw[edge](RA)--(Rv);

\node[blacknode](RBl1) at (2.0, \Ry-5.0){}; \draw[edge](RB)--(RBl1);

\node[blacknode](RCl1) at ( 5.0, \Ry-5.0){};
\node[blacknode](RCl2) at ( 7.0, \Ry-5.0){};
\node[pinknode] (Rw)   at ( 9.0, \Ry-5.0){}; \node[lbl,font=\Huge,right=3pt of Rw]{$w$};
\draw[edge](RC)--(RCl1); \draw[edge](RC)--(RCl2); \draw[edge](RC)--(Rw);

%% Blue subtree rooted AT Ru (same offsets)
\node[bluenode](Rbu_l) at (-15.8, \Ry-7.0){};
\node[bluenode](Rbu_r) at (-12.2, \Ry-7.0){};
\node[bluenode](Rbll1) at (-17.0, \Ry-9.0){};
\node[bluenode](Rbll2) at (-15.2, \Ry-9.0){};
\node[bluenode](Rbrl1) at (-13.4, \Ry-9.0){};
\node[bluenode](Rbrl2) at (-11.0, \Ry-9.0){};
\draw[blueedge](Ru)--(Rbu_l);    \draw[blueedge](Ru)--(Rbu_r);
\draw[blueedge](Rbu_l)--(Rbll1); \draw[blueedge](Rbu_l)--(Rbll2);
\draw[blueedge](Rbu_r)--(Rbrl1); \draw[blueedge](Rbu_r)--(Rbrl2);

%% Green subtree rooted AT Rv (same offsets)
\node[greennode](Rgv_l) at (-6.8, \Ry-7.0){};
\node[greennode](Rgv_r) at (-3.2, \Ry-7.0){};
\node[greennode](Rgll1) at (-8.0, \Ry-9.0){};
\node[greennode](Rgll2) at (-6.2, \Ry-9.0){};
\node[greennode](Rgrl1) at (-4.4, \Ry-9.0){};
\node[greennode](Rgrl2) at (-2.0, \Ry-9.0){};
\draw[greenedge](Rv)--(Rgv_l);    \draw[greenedge](Rv)--(Rgv_r);
\draw[greenedge](Rgv_l)--(Rgll1); \draw[greenedge](Rgv_l)--(Rgll2);
\draw[greenedge](Rgv_r)--(Rgrl1); \draw[greenedge](Rgv_r)--(Rgrl2);

%% Pink subtree rooted AT Rw (same offsets)
\node[pinknode](Rpw_l) at ( 7.2, \Ry-7.0){};
\node[pinknode](Rpw_r) at (10.8, \Ry-7.0){};
\node[pinknode](Rpll1) at ( 6.0, \Ry-9.0){};
\node[pinknode](Rpll2) at ( 7.8, \Ry-9.0){};
\node[pinknode](Rprl1) at ( 9.6, \Ry-9.0){};
\node[pinknode](Rprl2) at (12.0, \Ry-9.0){};
\draw[pinkedge](Rw)--(Rpw_l);    \draw[pinkedge](Rw)--(Rpw_r);
\draw[pinkedge](Rpw_l)--(Rpll1); \draw[pinkedge](Rpw_l)--(Rpll2);
\draw[pinkedge](Rpw_r)--(Rprl1); \draw[pinkedge](Rpw_r)--(Rprl2);

\end{tikzpicture}